\documentclass[12pt, a4paper, reqno]{amsart}

\usepackage{mathtools}
\usepackage{amsthm}
\usepackage{amssymb}
\usepackage{mathrsfs}
\usepackage{braket}

\usepackage{enumitem}
\setlist[enumerate]{
  topsep=2pt,
  parsep=0pt,
  partopsep=0pt,
  itemsep=3pt,
  leftmargin=3em,
  label=\rm{(\roman*)}
}

\usepackage{graphicx}
\usepackage[centering, heightrounded]{geometry}
\usepackage{secdot}
\usepackage[small,bf]{caption}
\usepackage[subrefformat=parens]{subcaption}
\usepackage[svgnames]{xcolor}
\usepackage{hyperref}

\numberwithin{equation}{section}

\newtheorem{theorem}{Theorem}[section]
\newtheorem{proposition}[theorem]{Proposition}
\newtheorem{lemma}[theorem]{Lemma}

\theoremstyle{definition}
\newtheorem{definition}[theorem]{Definition}
\newtheorem*{definition*}{Definition}

\theoremstyle{remark}
\newtheorem{remark}[theorem]{Remark}

\newtheorem*{theorem*}{Theorem}
\newtheorem*{proposition*}{Proposition}
\newtheorem*{lemma*}{Lemma}
\newtheorem*{remark*}{Remark}
\newtheorem*{assumption*}{Assumption}

\newcommand{\Z}{\mathbb Z}

\newcommand{\R}{\mathbb R}

\newcommand{\eps}{\varepsilon}
\renewcommand{\Re}{\operatorname{Re}}

\newcommand{\scX}{\mathscr X}

\newcommand{\cF}{\mathcal F}

\def\tbra[#1,#2]{{\left\langle #1,#2\right\rangle}}
\def\rbra[#1,#2]{{\left( #1,#2\right)}}
\def\norm[#1]{{\left\Vert #1 \right\Vert}}
\def\abs[#1]{{\left\vert #1 \right\vert}}

\newcommand{\rad}{\mathrm{rad}}

\newcommand{\pt}{\partial}

\newcommand{\ce}{\mathrel{\mathop:}=}

\DeclareFontFamily{U}{mathx}{\hyphenchar\font45}
\DeclareFontShape{U}{mathx}{m}{n}{
  <5> <6> <7> <8> <9> <10>
  <10.95> <12> <14.4> <17.28> <20.74> <24.88>
  mathx10
}{}
\DeclareSymbolFont{mathx}{U}{mathx}{m}{n}
\DeclareMathAccent{\widebar}{0}{mathx}{"73}

\definecolor{Tuileries}{rgb}{0,0.65,0.84}
\hypersetup{
  colorlinks=true,
  linkcolor=black,
  citecolor=black,
  urlcolor=black,%
  pdftitle={Instability theory of stationary solutions for NLS}
}

\title[]{Instability theory of stationary solutions for nonlinear Schr\"odinger equations}

\author[N. Fukaya]{Noriyoshi Fukaya}
\address{Regional ICT Research Center for Human, Industry and Future, 
The University of Shiga Prefecture,
Shiga 522-8533, Japan
\newline\indent
Osaka Central Advanced Mathematical Institute, 
Osaka Metropolitan University, 
Osaka 558-8585, Japan
}%
\email{fukaya.n@e.usp.ac.jp}

\author[M. Hayashi]{Masayuki Hayashi}
\address{Graduate School of Human and Environmental Studies,
Kyoto University, Kyoto 606-8501, Japan
\newline\indent
Waseda Research Institute for Science and Engineering, Waseda University, Tokyo 169-8555, Japan
}%
 \email{hayashi.masayuki.3m@kyoto-u.ac.jp}

\date{\today}

\begin{document}
\begin{abstract}
We study the orbital instability of algebraically decaying stationary
solutions of nonlinear Schr\"odinger equations in dimensions $d\ge2$,
taking competing double-power nonlinearities as a canonical model.
At zero frequency, the essential spectrum of the linearized operator
touches zero, and the usual $H^1$ coercivity is lost.
We construct the zero-frequency limit $\eta_0$ of the frequency
derivatives $\pt_\omega\phi_\omega$ of standing-wave profiles
and determine its sharp asymptotic behavior.
For a stationary ground state $\phi_0\in H^1(\R^d)$, we prove instability
if $\braket{x}\phi_0\notin L^2(\R^d)$ or, when
$\braket{x}\phi_0\in L^2(\R^d)$, if the duality pairing
$\tbra[\eta_0,\phi_0]$ is negative.
This criterion is a zero-frequency counterpart of the negative slope
condition in the Grillakis--Shatah--Strauss theory.
The proof combines variation of parameters for the linearized ODE
with a shooting analysis that determines the global behavior of
the homogeneous solution regular at the origin.
We then localize $\eta_0$ to construct an unstable direction and
establish instability through a Lyapunov functional argument
based on the variational characterization of ground states.
The results extend with essentially the same proofs to a class of
multiple-power nonlinearities.
\end{abstract}

\maketitle
\tableofcontents

\section{Introduction}

\subsection{Setting of the problem}
We study stationary solutions of the nonlinear Schr\"odinger equation with
competing double-power nonlinearities
\begin{align}
\label{eq}
i\pt_t u+\Delta u-|u|^{p-1}u+|u|^{q-1}u=0,
\qquad (t,x)\in\R\times\R^d,
\end{align}
where $d\ge1$ and $1<p<q<(d+2)/(d-2)_+$. Here the upper bound on $q$ is
understood to be $\infty$ when $d=1,\,2$. 
The Cauchy problem
for \eqref{eq} is locally well-posed in
$H^1(\R^d)$; see \cite[Chapter~4]{C03}.
Moreover, the energy
\begin{align*}
  E(u(t))
  \ce\frac12\|\nabla u(t)\|_{L^2}^2
  +\frac{1}{p+1}\|u(t)\|_{L^{p+1}}^{p+1}
  -\frac{1}{q+1}\|u(t)\|_{L^{q+1}}^{q+1}
\end{align*}
and the mass
\begin{align*}
M(u(t))\ce\frac12\|u(t)\|_{L^2}^2
\end{align*}
are conserved along the flow. 
We denote the mass-critical exponent by
\[
  p_{\mathrm c}(d)\ce1+\frac4d.
\]
If $1<p<q<p_{\mathrm c}(d)$, which is the range of particular interest
here, the conservation laws and the Gagliardo--Nirenberg inequality
imply global existence for every initial datum in $H^1(\R^d)$.
We refer to \cite{TVZ07} for basic results on the global dynamics of
nonlinear Schr\"odinger equations with double-power nonlinearities.

Our interest in the double-power setting stems from the fact that
\eqref{eq} serves as a canonical model in which an algebraically decaying
stationary solution arises as the zero-frequency limit of a family of
exponentially decaying standing waves.
A similar phenomenon occurs in various physically relevant dispersive
models, where algebraic solitons emerge as parameter limits of families
of exponentially decaying solitons.
Examples include the derivative nonlinear Schr\"odinger equation
describing plasma waves \cite{KN78}, the massive Thirring model and
related coupled-mode equations in nonlinear optics \cite{BPZ98,HHP24},
and the Gardner equation describing oceanic internal waves
\cite{GPT99,HMHS11}. The competing double-power nonlinearity in \eqref{eq} is also of physical interest. In \cite{ZMMZ25}, the authors study a version of \eqref{eq} with an additional periodic potential for $d=1,2$ and $(p,q)=(3,5)$ in the context of nonlinear optics, referring to the nonlinearity as the inverted cubic-quintic nonlinearity.

Equation~\eqref{eq} admits standing waves $e^{i\omega t}\phi_\omega(x)$
for $\omega\ge0$, where $\phi_\omega$ is the unique positive, radial,
decreasing ground state of
\begin{equation}\label{eq:spomega}
  -\Delta\phi+\omega\phi+|\phi|^{p-1}\phi-|\phi|^{q-1}\phi=0,
  \qquad x\in\R^d.
\end{equation}
We collect the basic elliptic properties of these profiles in Section~\ref{sec:ellip}.
For $\omega>0$, the profile $\phi_\omega$ belongs to $H^1(\R^d)$ and decays exponentially at infinity. At $\omega=0$, the profile $\phi_0$ belongs to $\dot H^1(\R^d)\cap L^{p+1}(\R^d)$, and $u(t,x)=\phi_0(x)$ is a stationary solution of \eqref{eq}. The algebraic decay of $\phi_0$ exhibits three distinct regimes:
\begin{align}
\label{decay}
  \phi_0(r)\sim
  \begin{cases}
    r^{-2/(p-1)}, & p<p^*,\\[2pt]
    r^{-(d-2)}(\log r)^{-(d-2)/2}, & p=p^*,\\[2pt]
    r^{-(d-2)}, & p>p^*,
  \end{cases}
  \qquad r\to\infty,
\end{align}
where $p^*=p^*(d)\ce d/(d-2)_+$.
These sharp asymptotics yield the following criterion for $L^2$-integrability:
\begin{align}
\label{l2}
  \phi_0\in L^2(\R^d)
  \iff
  \begin{cases}
    1<p<p_{\mathrm c}(d), & 1\le d\le3,\\
    1<p\le p_{\mathrm c}(4), & d=4,\\
    1<p<\dfrac{d+2}{d-2}, & d\ge5.
  \end{cases}
\end{align}

\subsection{Stability and instability}
We now introduce the notion of stability modulo phase rotations and spatial translations, commonly referred to as orbital stability.
\begin{definition*}
When $\omega=0$, we assume that $\phi_0\in L^2(\R^d)$.
The standing wave $e^{i\omega t}\phi_\omega$ of \eqref{eq}
is \emph{stable} if, for every $\eps>0$, there exists $\delta>0$
such that whenever $u_0\in H^1(\R^d)$ satisfies
$\|u_0-\phi_\omega\|_{H^1}<\delta$,
the $H^1$-solution $u(t)$ of \eqref{eq} with $u(0)=u_0$
exists globally in time and satisfies
\[
   \sup_{t\in\R}\inf_{(\theta,y)\in\R\times\R^d}
 \|u(t)-e^{i\theta}\phi_\omega(\cdot-y)\|_{H^1}<\eps.
\]
Otherwise, the standing wave is \emph{unstable}.
\end{definition*}

Define $M(\omega)\ce\frac12\|\phi_\omega\|_{L^2}^2$ for
$\omega>0$. The general theory of Grillakis, Shatah, and Strauss
\cite{GSS87} relates stability to the slope
\begin{align}
\label{Mprime}
  M'(\omega)
  =\rbra[\pt_\omega\phi_\omega,\phi_\omega]_{L^2},
  \qquad\omega>0.
\end{align}
More precisely, $e^{i\omega t}\phi_\omega$ is stable if
$M'(\omega)>0$ and unstable if $M'(\omega)<0$. 
However, determining the sign of $M'(\omega)$ for double-power nonlinearities is a delicate problem, especially in higher dimensions, because the equation lacks the scaling symmetry of the pure-power case.

When $q\ge p_{\mathrm c}(d)$, the blow-up argument of \cite{BC81} can be adapted to the double-power setting. Indeed, it was proved in \cite{FH21} that $e^{i\omega t}\phi_\omega$ is strongly unstable for every \mbox{$\omega\ge0$}.\footnote{Here strong instability means that there are initial data arbitrarily close to the profile whose solutions blow up in finite time. If $\phi_0\notin L^2(\R^d)$, proximity of the initial data to $\phi_0$ is measured in the topology of $\dot H^1(\R^d)\cap L^{p+1}(\R^d)$.}
This strong instability implies instability in the above
sense whenever the profile belongs to $H^1(\R^d)$.
By contrast, when $q<p_{\mathrm c}(d)$, the instability problem becomes more delicate, not only because finite time blow-up is ruled out in this regime but also because stability can depend on the frequency.
Indeed, the standing wave is stable for sufficiently large $\omega>0$ by a perturbation argument \cite{F03}, but it can be unstable for sufficiently small $\omega>0$. Such a frequency-dependent change in stability was first demonstrated by Ohta \cite{O95ko} in the one-dimensional double-power setting.
We note that scaling invariance precludes this phenomenon for pure-power nonlinearities. In dimension one, the stability of standing waves with positive frequencies has now been completely classified for double-power nonlinearities; see \cite{M08,KLT22}.

In higher dimensions, recent advances have been made in the study of mass-constrained variational problems. We refer to \cite{JL22cv} for the general result and to \cite{BFG25} for the recent result on \eqref{eq}.
We note that, once the mass-constrained minimization problem is solved, the argument of \cite{CL82} yields stability of the corresponding set of minimizers.
Unlike the result in \cite{F03}, these variational results can also address stability at intermediate frequencies. However, since they are formulated in terms of prescribed mass, they do not directly yield a stability classification for each fixed frequency $\omega>0$.

At $\omega=0$, the general theory of stability and instability in \cite{GSS87} does not apply in our setting. The fundamental obstruction is the loss of $H^1$ coercivity of the linearized operator. To explain this, define
\begin{align}
\label{eq:lin}
  V_\omega\ce p\phi_\omega^{p-1}-q\phi_\omega^{q-1},
  \qquad L_\omega\ce-\Delta+\omega+V_\omega,
\end{align}
where $L_\omega$ is the real linearized operator about $\phi_\omega$.
Since $V_\omega$ tends to zero at infinity, the essential spectrum of $L_\omega$ is
\begin{equation}
\label{ess-spc}
  \sigma_{\mathrm{ess}}(L_\omega)=[\omega,\infty).
\end{equation}
In particular, $\sigma_{\mathrm{ess}}(L_0)=[0,\infty)$, so there is no spectral gap above zero. Consequently, the $H^1$ coercivity estimate required by \cite{GSS87} fails, even after imposing the usual orthogonality conditions. A further difficulty at zero frequency concerns the slope $M'(\omega)$: neither the existence of $\lim_{\omega\downarrow0}M'(\omega)$ nor a formula for this limit is automatic.

Nevertheless, in \cite{FH21}, the authors first proved instability at zero frequency in the mass-subcritical regime under the condition
\begin{equation}\label{eq:gamma}
  q>\gamma_d(p),
  \qquad
  \gamma_d(p)\ce
  \frac{16+d^2+6d-pd(d+2)}{d\bigl(d+2-(d-2)p\bigr)};
\end{equation}
see Figures~\ref{fig:1} and \ref{fig:sharp}.
This condition has the following characterization:
\begin{align}
{\eqref{eq:gamma}}\iff 
 \left.\frac{d^2}{d\lambda^2}E(\phi_0^\lambda)\right|_{\lambda=1}<0
\quad
(\phi_0^\lambda(x)\ce\lambda^{d/2}\phi_0(\lambda x) ),
\label{eq:scaling}
\end{align}
which is well suited to a virial argument combined with variational methods (see \cite{O95}). However, this condition is not expected to be sharp for instability in general.
\begin{figure}[t]
\centering
\begin{minipage}[b]{0.45\linewidth}
  \centering
\includegraphics[width=0.9\linewidth]{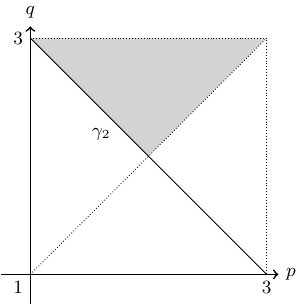}
  \subcaption{$d=2$}
  \label{fig:1a}
\end{minipage}
\begin{minipage}[b]{0.45\linewidth}
  \centering
\includegraphics[width=0.9\linewidth]{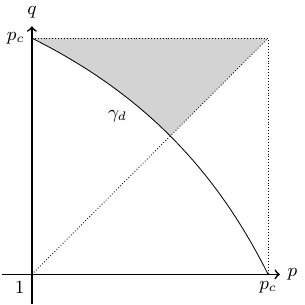}
  \subcaption{$d\ge3$}
  \label{fig:1b}
\end{minipage}
\caption{The shaded regions represent $q>\gamma_d(p)$ in the range
$1<p<q<p_{\mathrm c}(d)$. On $[1,p_{\mathrm c}(d)]$, the graph of $\gamma_d$ is a straight line for $d=2$ and strictly concave for $d\ge3$.
}%
\label{fig:1}
\end{figure}
In the subsequent work \cite{FH25}, the authors proved, for $d=1$
and $1<p<q<5$, that $\phi_0$ is unstable under the condition
\begin{align}
\label{om0nega}
  M'(0)\ce\lim_{\omega\downarrow0}M'(\omega)\in[-\infty,0).
\end{align}
In dimension one, the existence of this limit has been established,
and
\[
  \eqref{om0nega}\quad\Longleftrightarrow\quad 2p+q>7;
\]
see \cite{H21,KLT22}. 
Condition~\eqref{om0nega} is the natural counterpart of the negative slope criterion $M'(\omega)<0$ for positive frequencies. This correspondence suggests that \eqref{om0nega} should also be sharp for instability at zero frequency.\footnote{The sharpness of the positive-frequency criterion is understood away from the degenerate case $M'(\omega)=0$.}

\begin{figure}[t]
\centering
\includegraphics[width=0.4\linewidth]{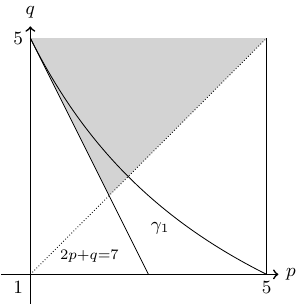}
\caption{Comparison of $q>\gamma_1(p)$ with the condition
$M'(0)<0$, equivalently $2p+q>7$, in $d=1$. The shaded region
represents the latter condition in $1<p<q<5$.}
\label{fig:sharp}
\end{figure}

\subsection{Main result}
The purpose of this paper is to develop a higher-dimensional
analogue of the zero-frequency instability result in \cite{FH25}.
We construct the canonical limit $\eta_0$ of the frequency
derivatives $\pt_\omega\phi_\omega$ and formulate the instability
criterion through its pairing with $\phi_0$. This pairing is the
natural zero-frequency counterpart of \eqref{Mprime}. 

The appropriate formulation involves a weighted $L^2$ condition.
Set
\[
  p_{\mathrm w}(d)\ce1+\frac{4}{d+2},
  \qquad \braket{x}\ce(1+|x|^2)^{1/2}.
\]
It follows from \eqref{decay} that
\begin{align}
\label{wl2}
  \braket{x}\phi_0\in L^2(\R^d)
  \iff
  \begin{cases}
    1<p<p_{\mathrm w}(d), & 1\le d\le5,\\
    1<p\le p_{\mathrm w}(6), & d=6,\\
    1<p<\dfrac{d+2}{d-2}, & d\ge7;
  \end{cases}
\end{align}
see also Table~\ref{t-p}.
\begin{table}[t]
\centering
\begin{tabular}{c|ccc}
  & $p_{\mathrm w}(d)$ & $p_{\mathrm c}(d)$ & $p^*(d)$\\ \hline
  $d=2$ & $2$ & $3$ & $\infty$\\[1pt]
  $d=3$ & $\frac95$ & $\frac73$ & $3$\\[3pt]
  $d=4$ & $\frac53$ & $2$ & $2$\\[3pt]
  $d=5$ & $\frac{11}{7}$ & $\frac95$ & $\frac53$\\[3pt]
  $d=6$ & $\frac32$ & $\frac53$ & $\frac32$
\end{tabular}
\vspace{10pt}
\caption{Comparison of the critical exponents for $2\le d\le6$.
For $d\ge7$, one has $p^*(d)<p_{\mathrm w}(d)<p_{\mathrm c}(d)$.}
\label{t-p}
\end{table}
For $s\in\R$, define
\[
  \cF(H^s)
  \ce\{f\in L^2_{\mathrm{loc}}(\R^d):\braket{x}^sf\in L^2(\R^d)\},
  \qquad
  \|f\|_{\cF(H^s)}\ce\|\braket{x}^sf\|_{L^2}.
\]
Thus \eqref{wl2} is equivalent to $\phi_0\in\cF(H^1)$.
For real-valued measurable functions $f$ and $g$,
we define
\[
  \tbra[f,g]\ce\int_{\R^d}f(x)g(x)\,dx,
\]
where the integral is understood in the Lebesgue sense.
For $f\in\cF(H^{-1})$ and $g\in\cF(H^1)$, this integral is
absolutely convergent and coincides with the usual duality pairing
between $\cF(H^{-1})$ and $\cF(H^1)$.

Our main result is as follows.

\begin{theorem}\label{thm:ode}
Let $d\ge2$. Assume that $\phi_0\in L^2(\R^d)$, equivalently, that the parameters
satisfy \eqref{l2}. Then there exists a unique function
$\eta_0\in C_{\rad}^2(\R^d)$ solving $L_0\eta_0=-\phi_0$ and having
the following behavior as $r\to\infty$:
\begin{alignat*}{2}
  \eta_0(r)&\sim-r^2\phi_0(r) &\qquad &\text{if}~(d,p)\ne(4,2),\\
  \eta_0(r)&\sim-1 &&\text{if}~(d,p)=(4,2).
\end{alignat*}
Moreover,
\[
  \pt_\omega\phi_\omega\to\eta_0
  \quad\text{in }C_{\mathrm{loc}}^2([0,\infty))
  \quad\text{as }\omega\downarrow0.
\]
This characterization of $\eta_0$ yields the following
instability criteria for $\phi_0$.
\begin{enumerate}
\item If $\phi_0\notin\cF(H^1)$, then
\[
  \tbra[\eta_0,\phi_0]=-\infty,
\]
and $\phi_0$ is unstable.

\item If $\phi_0\in\cF(H^1)$, then
$\eta_0\in\cF(H^{-1})$ and the duality pairing $\tbra[\eta_0,\phi_0]$
is finite. If, in addition, $\tbra[\eta_0,\phi_0]<0$, then $\phi_0$ is unstable.
\end{enumerate}
\end{theorem}

\begin{remark*}
The exceptional nature of the case $(d,p)=(4,2)$ stems from the coincidence of the two critical exponents: $p_{\mathrm c}(4)=p^*(4)=2$. This case also marks an endpoint of the parameter range in which our ODE representation of $\eta_0$ remains valid; see Section~\ref{sec:4} for details.
\end{remark*}

We briefly describe the main difficulties and outline the proof.
The central step is an ODE construction of $\eta_0$, together with
sharp estimates of its decay/growth at infinity. For
$\omega>0$, define the function $\eta_\omega$ by 
\begin{align*}
\eta_\omega\ce\pt_\omega\phi_\omega.
\end{align*}
This function satisfies $L_\omega\eta_\omega=-\phi_\omega$, whose radial form is
\begin{equation}\label{inhom-ode}
  \eta''+\frac{d-1}{r}\eta'
  -\bigl(\omega+V_\omega(r)\bigr)\eta=\phi_\omega,
  \qquad r>0.
\end{equation}
In dimension one, a quadrature formula for $\phi_\omega$ is available (see \cite{IK93}), and the construction in \cite{FH25} makes direct use of it. In dimensions $d\ge2$, however, the radial term $(d-1)r^{-1}\phi'$ in the profile equation prevents the same quadrature-based approach. This is not merely a technical obstruction: the term itself introduces genuine analytical difficulties, as seen in the more delicate uniqueness problem for ground states and the three distinct decay regimes in \eqref{decay}. 

We construct $\eta_0$ as a solution of \eqref{inhom-ode}
with $\omega=0$ by variation of parameters, using appropriate
fundamental solutions of the associated homogeneous equation.
A key step in this construction is to obtain an ODE representation
of $\eta_\omega$ for $\omega>0$ and then pass to the limit as
$\omega\downarrow0$.
This procedure selects the solution $\eta_0$ of \eqref{inhom-ode}
with $\omega=0$, which is precisely the one needed in the
instability criterion for $\phi_0$.
When $p\le p^*$, constructing the fundamental solutions
from infinity is rather delicate: it requires the precise
leading coefficient and remainder estimates for $\phi_0$
in Proposition~\ref{prop:V}, beyond the sharp decay rates in \eqref{decay}.
The asymptotic analysis of $\eta_0$ also requires global information
about the homogeneous solution regular at the origin.
To obtain this information, we use the shooting analysis
underlying the uniqueness argument in \cite{M93}.\footnote{
As discussed below, our approach extends beyond double-power
nonlinearities, but this shooting analysis imposes the main
structural restrictions on such extensions.}
The resulting asymptotic behavior of $\eta_0$ is also essential
for uniqueness: the inhomogeneous equation \eqref{inhom-ode}
alone does not determine a unique solution,
since one may add a solution of the homogeneous equation.
The prescribed behavior at infinity, however, singles out $\eta_0$.
The uniqueness argument for $\eta_0$ is also related to that
for radial solutions of inhomogeneous linear elliptic equations
in \cite{FMR06}.

The function $\eta_0$ need not belong to $L^2(\R^d)$. Under either
instability hypothesis in Theorem~\ref{thm:ode}, we localize it and
use the resulting functions to construct a real-valued
unstable direction $\psi\in H^1_{\rad}(\R^d)$ satisfying
\[
  \rbra[\psi,\phi_0]_{L^2}=0,
  \qquad \tbra[L_0\psi,\psi]<0.
\]
The sharp asymptotics of $\eta_0$ are needed both to compute this
quadratic form and to control the errors introduced by
localization. Once such a direction has been constructed,
instability follows from a Lyapunov functional argument. The
variational characterization of $\phi_0$ on the Nehari manifold
provides the energy comparison needed in this argument, thereby
replacing the unavailable $H^1$ coercivity of the linearized
operator.

\subsection{Further results}

In view of \eqref{Mprime} and Theorem~\ref{thm:ode}, it is natural to expect the identity
\begin{equation}
\label{M0-eta0}
  M'(0)=\tbra[\eta_0,\phi_0]\in[-\infty,\infty).
\end{equation}
Establishing this identity, however, is more delicate than it may appear:
even the existence of the limit defining $M'(0)$ is nontrivial.
For $d=1$, \eqref{M0-eta0} was established in
\cite[Theorem~3.1]{FH25}.
For $d\ge2$, the following theorem establishes \eqref{M0-eta0}
with both sides equal to $-\infty$ when $\phi_0\notin\cF(H^1)$,
and identifies the finite slope under additional assumptions when
$\phi_0\in\cF(H^1)$.
\begin{theorem}\label{thm:res}
Let $d\ge2$ and assume that $\phi_0\in L^2(\R^d)$.
\begin{enumerate}[beginpenalty=10000]
\item If $\phi_0\notin\cF(H^1)$, then
\[
  M'(0)=\lim_{\omega\downarrow0}M'(\omega)
  =\tbra[\eta_0,\phi_0]=-\infty.
\]

\item If $\phi_0\in\cF(H^1)$, $d\ge6$, and
\begin{equation}\label{eq:high-d}
  p=p^* \quad\text{if }d=6,
  \quad
  p\ge p^* \quad\text{if }d\ge7,
\end{equation}
then
\begin{equation}\label{eq:limM}
  M'(0)=\lim_{\omega\downarrow0}M'(\omega)
  =\tbra[\eta_0,\phi_0]\in\R.
\end{equation}
\end{enumerate}
\end{theorem}

The divergence statement in Theorem~\ref{thm:res}\,\textup{(i)} is closely related to a sharp threshold previously obtained in dimension one. Indeed, the second author proved in \cite{H21} that
\[
  M'(0)=-\infty\quad\Longleftrightarrow\quad p\ge\frac73,
\]
and we note that the threshold $7/3$ is precisely $p_{\mathrm w}(1)$. Our result
reveals the weighted integrability condition behind this
threshold and extends the implication
$\phi_0\notin\cF(H^1)\Rightarrow M'(0)=-\infty$ to higher
dimensions. The same weight is natural in the duality pairing
$\tbra[\eta_0,\phi_0]$: the leading behavior
$\eta_0\sim-r^2\phi_0$, together with the exceptional asymptotics
at $(d,p)=(4,2)$, determines precisely when this pairing is
absolutely convergent.

We prove part~\textup{(ii)} by following the resolvent argument in \cite[Section~4.5]{LN20}.
We note that the condition \eqref{eq:high-d} is equivalent to
\[
  \phi_0\in\cF(H^1)
  \quad\text{and}\quad
  V_0\in L^{d/2}(\R^d).
\]
When $V_0\in L^{d/2}(\R^d)$, the Hardy--Littlewood--Sobolev inequality
allows us to treat the potential as a perturbation of the Laplacian.
The identity \eqref{eq:limM} should hold under the sole assumption
$\phi_0\in\cF(H^1)$.
For $p<p^*$, however, $V_0(r)\sim r^{-2}$ as $r\to\infty$,
a decay rate known to be critical in the theory of linear
Schr\"odinger equations; see \cite{BPST04,GVV06}.
Our perturbative argument here does not cover this critical case.

The next theorem relates the instability criterion in
Theorem~\ref{thm:ode} to the sufficient condition in \cite{FH21},
expressed through the $L^2$-invariant scaling in \eqref{eq:scaling}.

\begin{theorem}\label{thm:scale}
Let $d\ge2$. Assume that $\phi_0\in \cF(H^1)$ and
\[
  \left.\frac{d^2}{d\lambda^2}E(\phi_0^\lambda)\right|_{\lambda=1}
  \le 0.
\]
Then $\tbra[\eta_0,\phi_0]<0$.
\end{theorem}

The same computation underlying \eqref{eq:scaling} shows
that the scaling condition in Theorem~\ref{thm:scale} is equivalent to $q\ge\gamma_d(p)$.
Thus Theorem~\ref{thm:scale} provides an explicit condition on $(p,q,d)$ under which the instability criterion in Theorem~\ref{thm:ode}\,\textup{(ii)} is satisfied. Combined with Theorem~\ref{thm:ode}\,\textup{(i)}, this yields instability of $\phi_0$ whenever $\phi_0\in L^2(\R^d)$ and $q\ge\gamma_d(p)$: part~\textup{(i)} applies if $\phi_0\notin\cF(H^1)$, while part~\textup{(ii)} and Theorem~\ref{thm:scale} apply if $\phi_0\in\cF(H^1)$. In particular, the boundary case $q=\gamma_d(p)$ is now covered, whereas the sufficient condition in \cite{FH21} requires $q>\gamma_d(p)$. This makes explicit how Theorem~\ref{thm:ode} improves upon the instability result of \cite{FH21}.
Moreover, on the parameter region where $\phi_0\in\cF(H^1)$,
the pairing $\tbra[\eta_0,\phi_0]$ is naturally expected to depend continuously
on $(p,q)$.\footnote{Establishing this continuity requires a separate
argument, which we do not pursue here.}
Its strict negativity on $q=\gamma_d(p)$ therefore suggests that
instability persists in a neighborhood of each point of this curve
where the weighted condition holds, including on the side
$q<\gamma_d(p)$.
This is consistent with the one-dimensional picture in
Figure~\ref{fig:sharp}.

Finally, we remark that all three main theorems remain valid with essentially the same proofs if the double-power nonlinearity in \eqref{eq} is replaced by
\begin{align}
\label{multip}
\begin{aligned}
g(u)&=\sum_{k=1}^N a_k|u|^{p_k-1}u,
\qquad N\ge2,
\\
1<p_1<&\cdots<p_N<\frac{d+2}{(d-2)_+},
\quad a_k<0~(1\le k<N),
\quad a_N>0.
\end{aligned}
\end{align}
In this extension, $p_1$ takes the place of $p$ in the asymptotic and integrability conditions, while $E$ and $L_\omega$ are defined using the corresponding nonlinearity. Nonlinearities of this form are given in \cite{KZ91,M93} as examples for which uniqueness of ground states holds. The sign configuration
in \eqref{multip} ensures the one-sign-change property of the
$I$-function (see Lemma~\ref{lem:Isign}), which plays a key role
in the shooting analysis.

\subsection{Organization of the paper}

The remainder of the paper, in which we assume $d\ge2$,
is organized as follows.

Section~\ref{sec:ellip} states the basic elliptic properties of the
ground states and records their sharp decay and integrability estimates.
Section~\ref{sec:3} constructs fundamental systems for the homogeneous
linearized equation at the origin and at infinity.
Using these systems, Section~\ref{sec:4} constructs $\eta_0$,
determines its asymptotic behavior, and establishes the convergence
$\eta_\omega\to\eta_0$ locally in $C^2$ as $\omega\downarrow0$.
It also proves that the mass slope diverges to $-\infty$ as
$\omega\downarrow0$ when the weighted $L^2$ condition fails.

Section~\ref{sec:dir} constructs an unstable direction orthogonal to
$\phi_0$ by localizing $\eta_0$.
The sharp asymptotic behavior of $\eta_0$ is essential for
establishing negativity of the quadratic form and controlling
the localization errors.
Section~\ref{sec:inst} then completes the proof of
Theorem~\ref{thm:ode} using a Lyapunov functional argument
based on the Nehari characterization of ground states.

Section~\ref{sec:res} proves Theorem~\ref{thm:res}\,\textup{(ii)} by a resolvent
argument, identifying the zero-frequency limit of $M'(\omega)$
with $\tbra[\eta_0,\phi_0]$ under additional assumptions.
Section~\ref{sec:scale} proves Theorem~\ref{thm:scale}
by establishing an $L^2$-scaling criterion for the negativity
of $\tbra[\eta_0,\phi_0]$.

Finally, Appendix~\ref{app:shoot} presents a systematic account of
the shooting method, providing a unified proof of uniqueness and
radial nondegeneracy of the ground states and determining the global
behavior of the homogeneous solution regular at the origin.

\subsection*{Notation}

For real-valued functions $f$ and $g$, we write $f(r)\sim g(r)$
as $r\to r_0$, where $r_0\in\{0,\infty\}$, if
\[
  \lim_{r\to r_0}\frac{f(r)}{g(r)} =C
\]
for some constant $C>0$. Thus, $\sim$ denotes asymptotic equivalence
up to a positive multiplicative constant. For a function $f$ on $\R^d$ and $\lambda>0$, define its
$L^2$-invariant scaling by
\[
  f^\lambda(x)=\lambda^{d/2}f(\lambda x).
\]

Generally, we use $\tbra[\cdot,\cdot]$ for duality pairings,
with the underlying spaces understood from context. Variational derivatives
and duality pairings on complex-valued function spaces are understood over
$\R$. We use the real $L^2$ inner product
\[
  \rbra[f,g]_{L^2}\ce\Re\int_{\R^d}f(x)\overline{g(x)}\,dx.
\]
For real-valued $L^2$ functions, this agrees with $\tbra[f,g]$.

The subscript $\rad$ denotes restriction to radial functions. In particular, $(L_\omega)_{\rad}$ is regarded as an operator on the appropriate radial function spaces.


\section{Stationary solutions}\label{sec:ellip}

This section collects the elliptic properties of the ground states
and the estimates needed in the ODE analysis. We state their Nehari
characterization, uniqueness, and radial nondegeneracy, introduce the
frequency derivative and its equation, and record the sharp decay and
weighted integrability properties of $\phi_0$. 

\subsection{Elliptic properties}
\label{sec:2.1}

Define
\[  X=X_\omega\ce
    \begin{cases}
    \hfil H^1(\R^d), & \omega>0,\\
    \dot H^1(\R^d)\cap L^{p+1}(\R^d), & \omega=0.
    \end{cases}  \]
The action functional associated with \eqref{eq:spomega} is given by
\[  S_\omega(v)
    \ce E(v)+\frac{\omega}{2}\|v\|_{L^2}^2.  \]
A function $\phi\in X_\omega$ solves \eqref{eq:spomega} if and only if $S_\omega'(\phi)=0$.
We define the set of all nontrivial solutions by
\[  \mathcal{A}_\omega
    \ce \{\phi\in X_\omega\setminus\{0\}:
    S_\omega'(\phi)=0\},  \]
and the set of all ground states of \eqref{eq:spomega} by
\[  \mathcal{G}_\omega
    \ce \{\phi\in\mathcal{A}_\omega:
    S_\omega(\phi)\le S_\omega(v) \text{ for all } v\in\mathcal{A}_\omega\}.  \] 
\begin{proposition} \label{prop:uniq}
For $\omega\ge 0$, there exists a unique positive, radial, decreasing function $\phi_\omega\in X_\omega$ such that
\[  \mathcal{G}_\omega
    =\{e^{i\theta}\phi_\omega(\cdot-y):
    \theta\in\R,\: y\in\R^d\}.  \]
\end{proposition}
\begin{proof}
The existence of ground states was established in \cite{BL83}. The symmetry of ground states follows from the results of \cite{GNN81, LN93}. The uniqueness of positive radial solutions follows from \cite{ST00} for $d\geq 3$ and \cite{PS98} for $d=2$. 
\end{proof}

We introduce the Nehari functional
\begin{align*}
    K_\omega(v)
    \ce {}&\langle S_\omega'(v), v\rangle
\\  ={}&\|\nabla v\|_{L^2}^2
    +\omega\|v\|_{L^2}^2
    +\|v\|_{L^{p+1}}^{p+1}
    -\|v\|_{L^{q+1}}^{q+1},
\end{align*}
and the corresponding Nehari manifold
\[  \mathcal{K}_\omega
    \ce \{\phi\in X_\omega\setminus\{0\}:
    K_\omega(\phi)=0\}.  \]

\begin{proposition}\label{prop:nehari}
For $\omega\ge 0$, the set of ground states of \eqref{eq:spomega} is characterized by
\[  \mathcal{G}_\omega
    =\{\phi\in\mathcal{K}_\omega:
    S_\omega(\phi)\le S_\omega(v) \text{ for all } v\in\mathcal{K}_\omega\}.  \]
\end{proposition}
\begin{proof}
See \cite[Theorem~2.1]{FH21}.
\end{proof}

We also have the following result on the nondegeneracy of the ground states.

\begin{proposition}\label{prop:nondeg}
For $\omega\ge 0$, $\phi_\omega$ is radially nondegenerate in the sense that
\[  \ker (L_\omega)_{\rad}
    \ce \{w\in (X_\omega)_{\rad}:
    L_\omega w=0\}=\{0\}.   \]
\end{proposition}

We give a unified proof of the uniqueness and radial nondegeneracy of the ground states for all $\omega\ge0$ in Appendix~\ref{app:shoot}.
As a consequence of Proposition~\ref{prop:nondeg}, we obtain the following result.
\begin{proposition} \label{prop:reg}
The map $(0,\infty)\ni\omega\mapsto\phi_\omega\in H_{\rad}^1(\R^d)$ determined by Proposition~\ref{prop:uniq} is of class $C^1$.
\end{proposition}

\begin{proof}
In view of Propositions~\ref{prop:uniq} and~\ref{prop:nondeg},
we apply the implicit function theorem as in the proof of
\cite[Theorem~18]{SS85} to conclude that
\[
  (0,\infty)\ni\omega\mapsto\phi_\omega\in H_{\rad}^1(\R^d)
\]
is a $C^1$ map.
\end{proof}

By Proposition~\ref{prop:reg}, the frequency derivative is well defined
in $H_{\rad}^1(\R^d)$ for every $\omega>0$:
\begin{equation}
 \label{eq:eta-def}
    \eta_\omega
    = \pt_\omega\phi_\omega\in H_{\rad}^1(\R^d),\qquad
    \omega>0.
\end{equation}
Differentiating the profile equation \eqref{eq:spomega} with respect
to $\omega$ gives
\begin{equation}\label{eq:Lometa}
  L_\omega\eta_\omega=-\phi_\omega.
\end{equation}
The radial profile of $\eta_\omega$ therefore satisfies the linear inhomogeneous ODE
\begin{equation}
\label{eq:etaode}
    \eta''
    +\frac{d-1}{r}\eta'
    -\bigl(\omega+V_\omega(r)\bigr)\eta
    =\phi_\omega,\qquad
    r>0,
\end{equation}
with the Neumann boundary condition $\eta_\omega'(0)=0$. In radial notation,
\begin{equation*}
    V_\omega(r)
    =p\phi_\omega^{p-1}(r)-q\phi_\omega^{q-1}(r).
\end{equation*}

\subsection{Sharp decay estimates}

We recall several ODE properties of the positive radial ground state $\phi_\omega$. By elliptic regularity, $\phi_\omega\in C^\infty(\R^d)$. Moreover, its radial profile $r\mapsto\phi_\omega(r)$ satisfies the ordinary differential equation
\begin{equation*}
    -\phi''
    -\frac{d-1}{r}\phi'
    +\omega\phi
    +\phi^p
    -\phi^q
    =0,\qquad r>0,
\end{equation*}
with the boundary conditions
\begin{align*}
   &\phi'(0)=0,\qquad \lim_{r\to\infty}\phi(r)=0.
\end{align*}
It is well known that, for each fixed $\omega>0$,
\begin{equation}\label{eq:phi-exp}
  \phi_\omega(r)
  \sim r^{-\frac{d-1}{2}}e^{-\sqrt{\omega}\,r}
  \qquad\text{as }r\to\infty,
\end{equation}
whereas $\phi_0$ exhibits algebraic decay. 
Recall that 
\begin{equation*}
p^*=p^*(d)=\frac{d}{(d-2)_+},
\end{equation*}
and for $p\ne p^*$ set
\begin{equation}\label{eq:alpha}
  \alpha=\alpha(d,p)
  \ce \begin{cases}
    \dfrac{2}{p-1}, & p<p^*,\\[2pt]
    d-2, & p>p^*.
  \end{cases}
\end{equation}
Then the sharp decay of $\phi_0$ can be written as
\begin{align}
\label{eq:phidec}
 \phi_0(r)
    \sim\begin{cases}
      r^{-\alpha}, & p\ne p^*,\\[2pt]
      r^{-(d-2)}(\log r)^{-(d-2)/2}, & p=p^*
    \end{cases}
\end{align}
as $r\to\infty$.
When $p\le p^*$, we use the following sharper asymptotics, including the leading coefficient and an estimate for the remainder.

\begin{proposition}[{\cite{V81}}]
\label{prop:V}
There exists $\varepsilon>0$ such that
\begin{align}
\label{eq:sharp}
    \phi_0(r)=
    l_p\begin{cases}
    r^{-\frac{2}{p-1}}(1+O(r^{-\varepsilon})), & p<p^*,\\[2pt]
    r^{-(d-2)}(\log r)^{-\frac{d-2}{2}}(1+O((\log r)^{-\varepsilon})), & p=p^*
    \end{cases}
\end{align}
as $r\to\infty$, where
\[  l_p
    =\left\{\begin{alignedat}{2}
   &\biggl(\frac{2}{p-1}\Bigl(\frac{2p}{p-1}-d\Bigr)\biggr)^{1/(p-1)},&\quad
   &p<p^*,\\
   &\Bigl(\frac{d-2}{\sqrt{2}}\Bigr)^{d-2},&\quad
   &p=p^*.
    \end{alignedat}\right.
\]
\end{proposition}

We also record the integrability conditions for $\phi_0$.

\begin{proposition}
The profile $\phi_0$ satisfies\label{prop:phi-int}
\begin{align*}
  \phi_0\in L^2(\R^d)
  &\iff
  \begin{cases}
    1<p<p_{\mathrm{c}}(d), & 1\le d\le3,\\
    1<p\le p_{\mathrm{c}}(4), & d=4,\\
    1<p<\dfrac{d+2}{d-2}, & d\ge5,
  \end{cases}
  \\[4pt]
  \braket{x}\phi_0\in L^2(\R^d)
  &\iff
  \begin{cases}
    1<p<p_{\mathrm{w}}(d), & 1\le d\le5,\\
    1<p\le p_{\mathrm{w}}(6), & d=6,\\
    1<p<\dfrac{d+2}{d-2}, & d\ge7.
  \end{cases}
\end{align*}
\end{proposition}

\begin{proof}
The first equivalence is proved in \cite[Corollary~1.4]{FH21}.
Since $\phi_0$ is smooth near the origin, the condition
$\braket{x}\phi_0\in L^2(\R^d)$ is equivalent to
\[
  \int_2^\infty r^{d+1}\phi_0(r)^2\,dr<\infty.
\]
For $p\ne p^*$, \eqref{eq:phidec} shows that this holds if and only
if $2\alpha>d+2$. For $p=p^*$, the same estimate reduces the
condition to
\[
  \int_2^\infty r^{5-d}(\log r)^{-(d-2)}\,dr<\infty,
  \quad\text{which holds if and only if }d\ge6.
\]
Combining these conditions with \eqref{eq:alpha} gives the second
equivalence.
\end{proof}
From \eqref{eq:phidec}, one can derive a sharp decay estimate for the first derivative of $\phi_0$.
\begin{lemma}
\label{lem:dphi0}
The profile $\phi_0$ satisfies
\[
  \phi_0'(r)=
  \begin{cases}
    O(r^{-\alpha-1}), & p\ne p^*,\\[2pt]
    O\bigl(r^{-(d-1)}(\log r)^{-(d-2)/2}\bigr), & p=p^*
  \end{cases}
\]
as $r\to\infty$.
\end{lemma}

\begin{proof}
We use a rescaling argument on the fixed interval $1/2\le s\le2$ (see, e.g., \cite[Proposition 2.3]{L93}).
Suppose first that $p\ne p^*$. For sufficiently large $R$, define
\[
  u_R(s)\ce R^\alpha\phi_0(Rs),
  \qquad \frac12\le s\le2.
\]
The decay estimate \eqref{eq:phidec} gives
\[
  \norm[u_R]_{C([1/2,2])}\lesssim1
\]
uniformly in $R$. Since
\[
  \phi_0''+\frac{d-1}{r}\phi_0'
  =\phi_0^p-\phi_0^q,
\]
the rescaled function satisfies
\[
  u_R''+\frac{d-1}{s}u_R'=f_R,
  \qquad
  f_R(s)\ce R^{\alpha+2}
  \bigl(\phi_0(Rs)^p-\phi_0(Rs)^q\bigr).
\]
If $p<p^*$, then $p\alpha=\alpha+2$, whereas if $p>p^*$,
then $\alpha=d-2$ and $p\alpha>d=\alpha+2$. Since $q>p$, in both
cases
\[
  \norm[f_R]_{C([1/2,2])}\lesssim R^{\alpha+2-p\alpha}\lesssim1
\]
uniformly in $R$.

Suppose next that $p=p^*$. In this case, set
\[
  u_R(s)
  \ce R^{d-2}(\log R)^{(d-2)/2}\phi_0(Rs),
  \qquad \frac12\le s\le2.
\]
Since $\log(Rs)\sim\log R$ uniformly for $1/2\le s\le2$, the critical
estimate in \eqref{eq:phidec} yields
\[
  \norm[u_R]_{C([1/2,2])}\lesssim1.
\]
The equation for $u_R$ is again
\[
  u_R''+\frac{d-1}{s}u_R'=f_R,
\]
where
\[
  f_R(s)
  \ce R^d(\log R)^{(d-2)/2}
  \bigl(\phi_0(Rs)^{p^*}-\phi_0(Rs)^q\bigr).
\]
Here $p^*(d-2)=d$, and hence the $p^*$-term is
$O((\log R)^{-1})$, while the $q$-term is $o((\log R)^{-1})$. Thus
\[
  \norm[f_R]_{C([1/2,2])}\lesssim1
\]
in the critical case as well.

It remains to obtain a uniform estimate for $u_R'$. By the mean value
theorem, there exists $s_R\in(1/2,2)$ such that
\[
  u_R'(s_R)
  =\frac{u_R(2)-u_R(1/2)}{2-1/2},
  \qquad
  \abs[u_R'(s_R)]\lesssim1.
\]
Writing the rescaled equation in divergence form gives
\[
  \bigl(s^{d-1}u_R'(s)\bigr)'=s^{d-1}f_R(s).
\]
Therefore, for every $s\in[1/2,2]$,
\[
  s^{d-1}u_R'(s)
  =s_R^{d-1}u_R'(s_R)
   +\int_{s_R}^s t^{d-1}f_R(t)\,dt.
\]
Since $s$ and $s_R$ remain in $[1/2,2]$, the preceding uniform bounds
imply
\[
  \norm[u_R']_{C([1/2,2])}\lesssim1.
\]
For $p\ne p^*$, we have
$u_R'(1)=R^{\alpha+1}\phi_0'(R)$, whereas for $p=p^*$,
\[
  u_R'(1)
  =R^{d-1}(\log R)^{(d-2)/2}\phi_0'(R).
\]
Evaluating at $s=1$ proves the desired estimates.
\end{proof}



\section{Construction of fundamental systems of solutions}
\label{sec:3}

In this section, we consider the homogeneous equation associated with \eqref{eq:etaode}:
\begin{equation} \label{eq:pdeom}
    u''
    +\frac{d-1}{r}u'
    -\bigl(\omega+V_\omega(r)\bigr)u
    =0,\qquad r>0,
\end{equation}
and construct two fundamental systems of solutions, $\{v_\omega^+,v_\omega^-\}$ and $\{w_\omega^+,w_\omega^-\}$, with prescribed asymptotic behavior at the origin and at infinity, respectively.

For $p\ne p^*$, we put
\begin{equation} \label{eq:cp}
    c_p
    \ce \begin{cases}
    pl_p^{p-1}
    =\dfrac{2p}{p-1}\Bigl(\dfrac{2p}{p-1}-d\Bigr), &
    p<p^*,
\\  \hfil 0, &
    p>p^*,
    \end{cases}
\end{equation}
and set
\begin{equation} \label{eq:lambdapm}
    \lambda_\pm
    \ce \frac{-(d-2)\pm\sqrt{(d-2)^2+4c_p}}{2}.
\end{equation}
Note that when $p>p^*$,
\begin{align*}
   c_p=0,\quad
   \lambda_+=0,
   \quad
   \lambda_-=-(d-2).
\end{align*}
The precise statements are given in the following two propositions.

\begin{proposition} \label{prop:vom0}
For $\omega\ge 0$, there exists a fundamental system of solutions $\{v_{\omega}^+, v_{\omega}^-\}$ to \eqref{eq:pdeom} such that
\begin{align*}
    v_{\omega}^+(r)
   &=1+O(r^2),
\\  v_{\omega}^-(r)
   &=\begin{cases}
    \log r^{-1}\bigl(1+O(r^2)\bigr), & d=2,
\\  r^{-(d-2)}\bigl(1+O(r^2)\bigr), & d=3 \text{ or } d\ge 5,
\\  r^{-2}\bigl(1+O(r^2\log r^{-1})\bigr), & d=4
    \end{cases}
\end{align*}
as $r\downarrow 0$.
\end{proposition}

\begin{proposition} \label{prop:wasym}
There exists a fundamental system of solutions $\{w_{\omega}^+, w_{\omega}^-\}$ to \eqref{eq:pdeom} such that
\begin{align}
    w_{\omega}^+(r)
   &=\begin{cases}
    r^{-\frac{d-1}{2}}e^{\sqrt{\omega}\,r}\bigl(1+O(r^{-1})\bigr), & \omega>0,\\
    r^{\lambda_+}\bigl(1+O(r^{-\varepsilon})\bigr), & \omega=0,\ p\ne p^*,\\
    (\log r)^{d/2}\bigl(1+O((\log r)^{-\varepsilon})\bigr), & \omega=0,\ p=p^*,
    \end{cases}
    \notag\\
    w_{\omega}^-(r)
   &=\begin{cases}
    r^{-\frac{d-1}{2}}e^{-\sqrt{\omega}\,r}\bigl(1+O(r^{-1})\bigr), & \omega>0,\\
    r^{\lambda_-}\bigl(1+O(r^{-\varepsilon})\bigr), & \omega=0,\ p\ne p^*,\\
    r^{-(d-2)}(\log r)^{-d/2}\bigl(1+O((\log r)^{-\varepsilon})\bigr), & \omega=0,\ p=p^*
   \end{cases}
   \label{eq:wasym}
\end{align}
as $r\to\infty$ for some $\varepsilon>0$.
\end{proposition}

\begin{remark}\label{rem:wprime}
The proofs of Propositions~\ref{prop:vom0} and~\ref{prop:wasym} also yield the following asymptotics for $(v_\omega^\pm)'$ and $(w_\omega^\pm)'$:
\begin{align*}
    (v_\omega^+)'(r)
   &=O(r),
\\  (v_\omega^-)'(r)
   &=\begin{cases}
    -r^{-1}\bigl(1+O(r^2\log r^{-1})\bigr), & d=2,
\\  -(d-2)r^{-(d-1)}\bigl(1+O(r^{2})\bigr), & d=3 \text{ or } d\ge 5,
\\  -2r^{-3}\bigl(1+O(r^2\log r^{-1})\bigr), & d=4
   \end{cases}
\end{align*}
as $r\downarrow 0$ and
\begin{align*}
    (w_{\omega}^+)'(r)
   &=\begin{cases}
    \sqrt{\omega}r^{-\frac{d-1}{2}}e^{\sqrt{\omega}\,r}\bigl(1+O(r^{-1})\bigr), & \omega>0,
\\  \lambda_+r^{\lambda_+-1}\bigl(1+O(r^{-\varepsilon})\bigr), & \omega=0,\ p<p^*,
\\  O(r^{-1-\varepsilon}), & \omega=0,\ p>p^*,
\\  \frac{d}{2}r^{-1}(\log r)^{d/2-1}\bigl(1+O((\log r)^{-\varepsilon})\bigr), & \omega=0,\ p=p^*,
    \end{cases}
\\  (w_{\omega}^-)'(r)
   &=\begin{cases}
    -\sqrt{\omega}r^{-\frac{d-1}{2}}e^{-\sqrt{\omega}\,r}\bigl(1+O(r^{-1})\bigr), & \omega>0,
\\  \lambda_-r^{\lambda_--1}\bigl(1+O(r^{-\varepsilon})\bigr), & \omega=0,\ p\ne p^*,
\\  -(d-2)r^{-(d-1)}(\log r)^{-d/2}\bigl(1+O((\log r)^{-\varepsilon})\bigr), & \omega=0,\ p=p^*
   \end{cases}
\end{align*}
as $r\to\infty$.
\end{remark}


Related results on the asymptotic behavior of fundamental solutions
to radial Schr\"odinger equations can be found in \cite{IKO17},
which apply to the case $\omega=0$ and $p\ne p^*$ in our setting.
Here we give a unified proof that also covers the case $\omega=0$ with $p=p^*$, as well as the case $\omega>0$.
To this end, we first establish the following general result.
\begin{proposition}\label{prop:ode}
Let $\omega\ge0$ and $a>0$. Set $I=(0,a]$ in \textup{(i)} and
\textup{(ii)}, and $I=[a,\infty)$ in \textup{(iii)} and
\textup{(iv)} below. Suppose that $Q\in C(I)$ and $U\in C^2(I)$ satisfy
\begin{equation}\label{eq:gweq}
  U''+\frac{d-1}{r}U'-Q(r)U=0,\qquad r\in I,
\end{equation}
and
\[
  U(r)\ne0,\qquad r\in I.
\]
Set
\[
  \rho(\tau)
  \ce \tau^{d-1}
    \bigl(\omega+V_\omega(\tau)-Q(\tau)\bigr)U(\tau)^2,
  \qquad \tau\in I.
\]
In each case below, \eqref{eq:pdeom} admits a solution
$u\in C^2((0,\infty))$ of the form $u=U(1+h)$
on $(0,\delta]$ for some $0<\delta<a$ in \textup{(i)} and
\textup{(ii)}, or on $[R,\infty)$ for some $R>a$ in
\textup{(iii)} and \textup{(iv)}.
The correction $h$ is of class
$C^2$ on this interval and satisfies the following estimate.

\begin{enumerate}
\item
If $\rho$ and the function
\[
  s\mapsto
  s^{1-d}|U(s)|^{-2}\int_0^s|\rho(\tau)|\,d\tau
\]
are integrable near the origin, then
\[
  |h(r)|
  \le
  2\int_0^r s^{1-d}|U(s)|^{-2}
    \int_0^s|\rho(\tau)|\,d\tau\,ds,
  \qquad 0<r\le\delta.
\]

\item
If the function
\[
  s\mapsto
  s^{1-d}|U(s)|^{-2}\int_s^a|\rho(\tau)|\,d\tau
\]
is integrable near the origin, then
\[
  |h(r)|
  \le
  2\int_0^r s^{1-d}|U(s)|^{-2}
    \int_s^a|\rho(\tau)|\,d\tau\,ds,
  \qquad 0<r\le\delta.
\]

\item
If the function
\[
  s\mapsto
  s^{1-d}|U(s)|^{-2}\int_a^s|\rho(\tau)|\,d\tau
\]
is integrable at infinity, then
\[
  |h(r)|
  \le
  2\int_r^\infty s^{1-d}|U(s)|^{-2}
    \int_a^s|\rho(\tau)|\,d\tau\,ds,
  \qquad r\ge R.
\]

\item
If $\rho$ and the function
\[
  s\mapsto
  s^{1-d}|U(s)|^{-2}\int_s^\infty|\rho(\tau)|\,d\tau
\]
are integrable at infinity, then
\[
  |h(r)|
  \le
  2\int_r^\infty s^{1-d}|U(s)|^{-2}
    \int_s^\infty|\rho(\tau)|\,d\tau\,ds,
  \qquad r\ge R.
\]
\end{enumerate}
In particular, $h(r)\to0$ as $r\downarrow0$ in \textup{(i)}
and \textup{(ii)}, and as $r\to\infty$ in \textup{(iii)}
and \textup{(iv)}.
\end{proposition}

\begin{proof}
On $I$, the function
$u=U(1+h)$ solves \eqref{eq:pdeom} if and only if $h$ solves
\begin{equation}\label{eq:wFVW}
  \bigl(r^{d-1}U^2h'\bigr)'=\rho(1+h).
\end{equation}
Indeed,
\[
  u'=U'(1+h)+Uh',
  \qquad
  u''=U''(1+h)+2U'h'+Uh''.
\]
Substituting these expressions into \eqref{eq:pdeom} and using
\eqref{eq:gweq}, we obtain
\[
  Uh''+\left(2U'+\frac{d-1}{r}U\right)h'
  =\bigl(\omega+V_\omega-Q\bigr)U(1+h).
\]
Multiplication by $r^{d-1}U$ gives \eqref{eq:wFVW}.
The calculation is reversible because $U\ne0$ on the interval
under consideration.

For \textup{(i)}, choose $\delta\in(0,a)$ such that
\[
  \int_0^\delta s^{1-d}|U(s)|^{-2}
    \int_0^s|\rho(\tau)|\,d\tau\,ds
  \le\frac12.
\]
On the complete metric space
\[
  Y\ce \left\{
    h\in C([0,\delta]):
    h(0)=0,\ \|h\|_{C([0,\delta])}\le1
  \right\},
\]
equipped with the metric induced by the supremum norm, define
\[
  \Phi(h)(r)
  \ce
  \int_0^r s^{1-d}U(s)^{-2}
    \int_0^s\rho(\tau)(1+h(\tau))\,d\tau\,ds,
  \qquad 0\le r\le\delta.
\]
The integrability assumptions imply that $\Phi(h)$ is continuous
on $[0,\delta]$ and vanishes at $0$. Moreover,
\[
  |\Phi(h)(r)|
  \le
  2\int_0^r s^{1-d}|U(s)|^{-2}
    \int_0^s|\rho(\tau)|\,d\tau\,ds
  \le1,
\]
and
\[
  \|\Phi(h_1)-\Phi(h_2)\|_{C([0,\delta])}
  \le\frac12\|h_1-h_2\|_{C([0,\delta])}.
\]
Thus $\Phi$ maps $Y$ into itself and is a contraction.
The Banach fixed point theorem yields a fixed point $h\in Y$,
and the preceding estimate gives the asserted bound.

For \textup{(ii)}, choose $\delta\in(0,a)$ such that
\[
  \int_0^\delta s^{1-d}|U(s)|^{-2}
    \int_s^a|\rho(\tau)|\,d\tau\,ds
  \le\frac12.
\]
On the same space $Y$, use instead
\[
  \Phi(h)(r)
  \ce
  -\int_0^r s^{1-d}U(s)^{-2}
    \int_s^\delta\rho(\tau)(1+h(\tau))\,d\tau\,ds,
  \qquad 0\le r\le\delta.
\]
Since
\[
  \int_s^\delta|\rho(\tau)|\,d\tau
  \le\int_s^a|\rho(\tau)|\,d\tau,
  \qquad 0<s\le\delta,
\]
the same estimates show that $\Phi$ is a contraction on $Y$.
Its fixed point satisfies the bound in \textup{(ii)}.

For \textup{(iii)}, choose $R>a$ such that
\[
  \int_R^\infty s^{1-d}|U(s)|^{-2}
    \int_a^s|\rho(\tau)|\,d\tau\,ds
  \le\frac12.
\]
Consider
\[
  Y\ce \left\{
    h\in C([R,\infty)):
    \sup_{r\ge R}|h(r)|\le1
  \right\},
\]
again equipped with the metric induced by the supremum norm, and define
\[
  \Phi(h)(r)
  \ce
  -\int_r^\infty s^{1-d}U(s)^{-2}
    \int_R^s\rho(\tau)(1+h(\tau))\,d\tau\,ds,
  \qquad r\ge R.
\]
Using
\[
  \int_R^s|\rho(\tau)|\,d\tau
  \le\int_a^s|\rho(\tau)|\,d\tau,
  \qquad s\ge R,
\]
we obtain the same contraction estimates and hence a fixed point
satisfying the bound in \textup{(iii)}.

For \textup{(iv)}, choose $R>a$ such that
\[
  \int_R^\infty s^{1-d}|U(s)|^{-2}
    \int_s^\infty|\rho(\tau)|\,d\tau\,ds
  \le\frac12.
\]
The required fixed point is obtained on the same space $Y$ from
\[
  \Phi(h)(r)
  \ce
  \int_r^\infty s^{1-d}U(s)^{-2}
    \int_s^\infty\rho(\tau)(1+h(\tau))\,d\tau\,ds,
  \qquad r\ge R.
\]
The contraction estimates and the bound in \textup{(iv)} follow as above.

In each case, differentiating the fixed point identity shows that
$h$ is of class $C^2$ on the interval adjacent to the relevant endpoint
and satisfies
\eqref{eq:wFVW}. The stated bounds also imply that $h$ tends to zero
at the relevant endpoint. Thus $u=U(1+h)$ solves \eqref{eq:pdeom}
on $(0,\delta]$ or $[R,\infty)$.
Since the coefficients of \eqref{eq:pdeom} are continuous on
$(0,\infty)$, the standard extension theorem for linear ODEs
extends $u$ uniquely to a solution in $C^2((0,\infty))$.
\end{proof}

\subsection{Fundamental solutions near the origin}

We now construct a fundamental system of solutions
$\{v_\omega^+,v_\omega^-\}$ to \eqref{eq:pdeom} with the prescribed
asymptotics at the origin for each $\omega\ge0$.

\begin{proof}[Proof of Proposition~\ref{prop:vom0}]
Set
\begin{align*}
    U^+(r)
   &\ce 1,\quad
    U^-(r)
   \ce  \begin{cases}
    \log r^{-1}, & d=2,
\\  r^{-(d-2)}, & d\ge 3.
  \end{cases}
\end{align*}
Then each $U^\pm$ solves
\begin{align*}
 U''+\frac{d-1}{r}U'=0,\qquad r>0.
\end{align*}
We put
\[
  \rho^{\pm}(\tau)
    \ce \tau^{d-1}\bigl(\omega+V_\omega(\tau)\bigr)U^{\pm}(\tau)^2.
\]
We apply Proposition~\ref{prop:ode} with $I=(0,1/2]$, where
\[
  U^\pm(r)>0,\qquad 0<r\le1/2.
\]
All estimates below hold for $0<s\le1/2$.

First, to construct $v_\omega^+$, we verify the assumptions of Proposition~\ref{prop:ode}\,\textup{(i)}. The function $\rho^+$ is integrable near the origin, and
\begin{align*}
    s^{1-d}|U^+(s)|^{-2}\int_0^s |\rho^+(\tau)|\,d\tau
    \lesssim s^{1-d}\int_0^s\tau^{d-1}\,d\tau
    \lesssim s.
\end{align*}

Next, to construct $v_\omega^-$ when $d=2$ or $d=3$, we verify the assumptions of Proposition~\ref{prop:ode}\,\textup{(i)}. In these cases, $\rho^-$ is integrable near the origin. If $d=2$, then
\begin{align*}
   &s^{-1}|U^-(s)|^{-2}\int_0^s |\rho^-(\tau)|\,d\tau
    \lesssim s^{-1}(\log s^{-1})^{-2}\int_0^s \tau(\log \tau^{-1})^2\,d\tau
    \lesssim s,
\end{align*}
where we used
\[
  \int_0^s \tau(\log \tau^{-1})^2\,d\tau
    =\frac{s^2}{2}(\log s^{-1})^2+\frac{s^2}{2}\log s^{-1}+\frac{s^2}{4}
    \lesssim s^2(\log s^{-1})^2.
\]
If $d=3$, then
\begin{align*}
   &s^{-2}|U^-(s)|^{-2}\int_0^s |\rho^-(\tau)|\,d\tau
    \lesssim \int_0^s 1\,d\tau
    = s.
\end{align*}

Finally, to construct $v_\omega^-$ when $d\ge 4$, we verify the assumptions of Proposition~\ref{prop:ode}\,\textup{(ii)}. If $d=4$, then
\begin{align*}
   &s^{-3}|U^-(s)|^{-2}\int_s^{1/2} |\rho^-(\tau)|\,d\tau
    \lesssim s\int_s^{1/2}\tau^{-1}\,d\tau
    \le s\log s^{-1}.
\end{align*}
If $d\ge 5$, then
\begin{align*}
   &s^{1-d}|U^-(s)|^{-2}\int_s^{1/2} |\rho^-(\tau)|\,d\tau
    \lesssim s.
\end{align*}
Therefore, by Proposition~\ref{prop:ode}, we obtain the desired solutions of the form $v_\omega^\pm=U^{\pm}(1+h)$.
\end{proof}

\subsection{Fundamental solutions at infinity for \texorpdfstring{$\omega>0$}{omega > 0}}

We first construct a fundamental system of solutions
$\{w_\omega^+,w_\omega^-\}$ to \eqref{eq:pdeom} for $\omega>0$,
with the asymptotic behavior stated in Proposition~\ref{prop:wasym}.

For $r>0$, set
\begin{align*}
    U^+(r)
   \ce r^{-\frac{d-1}{2}}e^{\sqrt{\omega}\,r},
   \qquad
    U^-(r)
   \ce r^{-\frac{d-1}{2}}e^{-\sqrt{\omega}\,r}.
\end{align*}
Each $U^\pm$ solves
\begin{align*}
    U''
   +\frac{d-1}{r}U'
   -\Bigl(\omega+\frac{b}{r^2}\Bigr)U
   =0,\qquad r>0,
\end{align*}
where $b\ce -(d-1)(d-3)/4$.

\begin{proof}[Proof of Proposition~\ref{prop:wasym} \textup{(}case $\omega>0$\textup{)}]
We apply Proposition~\ref{prop:ode} with $I=[2,\infty)$, where
$U^\pm(r)>0$ for $r\ge2$. Set
\begin{align*}
    \rho^\pm(\tau)
   &\ce \tau^{d-1}\Bigl(V_\omega(\tau)-\frac{b}{\tau^2}\Bigr)U^\pm(\tau)^2
\\ &=\tau^{d-1}\Bigl(p\phi_\omega(\tau)^{p-1}-q\phi_\omega(\tau)^{q-1}-\frac{b}{\tau^2}\Bigr)U^\pm(\tau)^2.
\end{align*}
Since $\phi_\omega$ decays exponentially,
$|\rho^\pm(\tau)|\lesssim \tau^{d-3}U^\pm(\tau)^2$ for $\tau\ge2$.

To construct $w_\omega^+$, we verify the assumptions of
Proposition~\ref{prop:ode}\,\textup{(iii)}. For $s\ge4$, we have
\begin{align*}
    s^{1-d}U^+(s)^{-2}\int_2^s |\rho^+(\tau)|\,d\tau
   &\lesssim e^{-2\sqrt{\omega}s}
     \Bigl(\int_2^{s/2}+\int_{s/2}^s\Bigr)\tau^{-2}e^{2\sqrt{\omega}\tau}\,d\tau
\\ &\le e^{-\sqrt{\omega}s}
    +\Bigl(\frac{s}{2}\Bigr)^{-2}e^{-2\sqrt{\omega}s}
      \int_{s/2}^s e^{2\sqrt{\omega}\tau}\,d\tau
\\ &\le e^{-\sqrt{\omega}s}
    +\frac{1}{2\sqrt{\omega}}\Bigl(\frac{s}{2}\Bigr)^{-2}
    \lesssim s^{-2}.
\end{align*}
Next, to construct $w_\omega^-$, we verify the assumptions of Proposition~\ref{prop:ode}\,\textup{(iv)}. The function $\rho^-$ is integrable at infinity, and, for $s\ge2$,
\begin{align*}
    s^{1-d}U^-(s)^{-2}\int_s^\infty |\rho^-(\tau)|\,d\tau
   &\lesssim e^{2\sqrt{\omega}\,s}\int_s^\infty \tau^{-2}e^{-2\sqrt{\omega}\tau}\,d\tau
\\ &\le e^{2\sqrt{\omega}\,s}s^{-2}\int_s^\infty e^{-2\sqrt{\omega}\tau}\,d\tau
    \lesssim s^{-2}.
\end{align*}
Thus, by Proposition~\ref{prop:ode}, we obtain the desired solutions $w_\omega^+$ and $w_\omega^-$.
\end{proof}

\subsection{Fundamental solutions at infinity for \texorpdfstring{$\omega=0$}{omega = 0}}

We next construct a fundamental system of solutions
$\{w_0^+,w_0^-\}$ to \eqref{eq:pdeom} for $\omega=0$,
with the asymptotic behavior stated in Proposition~\ref{prop:wasym}.
Before giving the proof, we collect the comparison equations
and functions for $p\ne p^*$ and $p=p^*$.

When $p\ne p^*$, \eqref{eq:sharp} and \eqref{decay} give
\begin{align} \label{eq:Vnc}
    V_0(r)
    =\frac{c_p}{r^2}+O(r^{-2-\varepsilon})
\end{align}
as $r\to\infty$, where $c_p$ is defined by \eqref{eq:cp}.
We therefore compare \eqref{eq:pdeom} with the Euler equation
\begin{equation} \label{eq:Euler}
    U''
    +\frac{d-1}{r}U'
    -\frac{c_p}{r^2}U
    =0,\qquad r>0.
\end{equation}
A fundamental system of solutions to \eqref{eq:Euler} is given by
\begin{align*}
    U^+(r)&\ce r^{\lambda_+},
    \qquad
    U^-(r)\ce r^{\lambda_-},
\end{align*}
where $\lambda_\pm$ are the roots of the characteristic equation
\begin{align}
\label{eq-lam}
   \lambda^2+(d-2)\lambda-c_p=0
\end{align}
and are given explicitly by \eqref{eq:lambdapm}.

When $p=p^*$, \eqref{eq:sharp} gives
\begin{align}
    \notag
    V_0(r)
   &=p^*\phi_0(r)^{p^*-1}-q\phi_0(r)^{q-1}
\\  \label{eq:Vcrit}
   &=\frac{(d-2)d}{2r^{2}\log r}
    +O\Bigl(\frac{1}{r^{2}(\log r)^{1+\varepsilon}}\Bigr)
\end{align}
as $r\to\infty$ for some $\varepsilon>0$.
In this case, for $r>1$, set
\begin{align*}
    U^+(r)&\ce (\log r)^{d/2},
    \qquad
    U^-(r)\ce r^{-(d-2)}(\log r)^{-d/2}.
\end{align*}
Then each $U^\pm$ solves
\begin{equation*}
    (U^\pm)''
    +\frac{d-1}{r}(U^\pm)'
    -Q^\pm(r)U^\pm
    =0,\qquad r>1,
\end{equation*}
where
\begin{align} \label{eq:Qcrit}
    Q^+(r)
   &\ce \frac{(d-2)d}{2r^2\log r}+\frac{\tfrac{d}{2}(\tfrac{d}{2}-1)}{r^2(\log r)^2},&
    Q^-(r)
   &\ce \frac{(d-2)d}{2r^2\log r}+\frac{\tfrac{d}{2}(\tfrac{d}{2}+1)}{r^2(\log r)^2}.
\end{align}

\begin{proof}[Proof of Proposition~\ref{prop:wasym} \textup{(}case $\omega=0$\textup{)}]
We distinguish two cases, using the corresponding comparison
functions $U^\pm$ defined above. In both cases, we apply
Proposition~\ref{prop:ode} with $I=[2,\infty)$, where
$U^\pm(r)>0$ for $r\ge2$.

\smallskip
\noindent\textit{Case 1: $p\ne p^*$.} We put
\[
  \rho^\pm(\tau)
    \ce \tau^{d-1}\Bigl(V_0(\tau)-\frac{c_p}{\tau^2}\Bigr)U^\pm(\tau)^2.
\]
By \eqref{eq:Vnc},
$|\rho^\pm(\tau)|\lesssim \tau^{d-3-\varepsilon}U^\pm(\tau)^2$
for $\tau\ge2$.

First, to construct $w_0^+$, we verify the assumptions of
Proposition~\ref{prop:ode}\,\textup{(iii)}. Reducing $\varepsilon>0$
if necessary so that $\sqrt{(d-2)^2+4c_p}-\varepsilon>0$, we have,
for $s\ge2$,
\begin{align*}
   &s^{1-d}U^+(s)^{-2}\int_2^s |\rho^+(\tau)|\,d\tau
\\ &\lesssim s^{-1-\sqrt{(d-2)^2+4c_p}}
        \int_2^s \tau^{-1+\sqrt{(d-2)^2+4c_p}-\varepsilon}\,d\tau
    \lesssim s^{-1-\varepsilon}.
\end{align*}

Next, to construct $w_0^-$, we verify the assumptions of Proposition~\ref{prop:ode}\,\textup{(iv)}. The estimate $|\rho^-(\tau)|\lesssim \tau^{-1-\sqrt{(d-2)^2+4c_p}-\varepsilon}$ for $\tau\ge2$ shows that $\rho^-$ is integrable at infinity, and, for $s\ge2$,
\begin{align*}
   &s^{1-d}U^-(s)^{-2}\int_s^\infty |\rho^-(\tau)|\,d\tau
\\ &\lesssim s^{-1+\sqrt{(d-2)^2+4c_p}}
        \int_s^\infty \tau^{-1-\sqrt{(d-2)^2+4c_p}-\varepsilon}\,d\tau
    \lesssim s^{-1-\varepsilon}.
\end{align*}
Thus, by Proposition~\ref{prop:ode}, we obtain the desired solutions $w_0^+$ and $w_0^-$.

\smallskip
\noindent\textit{Case 2: $p=p^*$.} We put
\[
  \rho^\pm(\tau)
    \ce \tau^{d-1}(V_0(\tau)-Q^\pm(\tau))U^\pm(\tau)^2.
\]
After reducing $\varepsilon$ so that $0<\varepsilon\le1$,
\eqref{eq:Vcrit} and \eqref{eq:Qcrit} give
\[
  |\rho^\pm(\tau)|
  \lesssim \tau^{d-3}(\log \tau)^{-1-\varepsilon}U^\pm(\tau)^2,
  \qquad \tau\ge2.
\]

First, to construct $w_0^+$, we verify the assumptions of
Proposition~\ref{prop:ode}\,\textup{(iii)}. For $s\ge2$, we have
\begin{align*}
   &s^{1-d}U^+(s)^{-2}\int_2^s |\rho^+(\tau)|\,d\tau
\\ &\lesssim s^{1-d}(\log s)^{-d}
    \int_2^s \tau^{d-3}(\log \tau)^{d-1-\varepsilon}\,d\tau
    \lesssim s^{-1}(\log s)^{-1-\varepsilon}.
\end{align*}
Here we used, for $s\ge2$,
\begin{align*}
    \int_2^s \tau^{d-3}(\log \tau)^{d-1-\varepsilon}\,d\tau
    &\le (\log s)^{d-1-\varepsilon}\int_2^s \tau^{d-3}\,d\tau
\\  &\lesssim s^{d-2}(\log s)^{d-1-\varepsilon}.
\end{align*}
Next, to construct $w_0^-$, we verify the assumptions of Proposition~\ref{prop:ode}\,\textup{(iv)}. The function $\rho^-$ is integrable at infinity, and, for $s\ge2$,
\begin{align*}
   &s^{1-d}U^-(s)^{-2}\int_s^\infty |\rho^-(\tau)|\,d\tau
\\ &\lesssim s^{d-3}(\log s)^d\int_s^\infty \tau^{-d+1}(\log \tau)^{-d-1-\varepsilon}\,d\tau
    \lesssim s^{-1}(\log s)^{-1-\varepsilon}.
\end{align*}
Therefore, by Proposition~\ref{prop:ode}, since $\int_r^\infty s^{-1}(\log s)^{-1-\varepsilon}\,ds\lesssim (\log r)^{-\varepsilon}$, we obtain the desired solutions $w_0^+$ and $w_0^-$.
\end{proof}

\section{Construction of \texorpdfstring{$\eta_0$}{eta0} and the zero-frequency limit}
\label{sec:4}

In this section, we first derive an integral representation
of $\eta_\omega$ for $\omega>0$ and then construct $\eta_0$ from the corresponding zero-frequency representation.
Next, we determine sharp asymptotics of $\eta_0$,
prove uniqueness under the asymptotic condition, and establish local
$C^2$ convergence $\eta_\omega\to\eta_0$ as $\omega\downarrow0$. We also prove the divergence
of $M'(\omega)$ to $-\infty$ when the weighted $L^2$ condition fails,
completing the ODE analysis required for the main theorems.

\subsection{Construction of \texorpdfstring{$\eta_0$}{eta0}}

In this subsection, we construct a regular solution $\eta_0$ of
\eqref{eq:Lometa} with $\omega=0$ using the fundamental systems
for the homogeneous equation \eqref{eq:pdeom}.
\begin{lemma}\label{lem:vw}
Let $\omega\ge0$. Then
\[
  v_\omega^+(r)\sim-w_\omega^+(r)
  \quad\text{as }r\to\infty.
\]
In particular, $v_\omega^+$ and $w_\omega^-$ are linearly independent.
Moreover,
\begin{align}
\label{w-asym0}
  |w_\omega^-(r)|\sim
  \begin{cases}
    \log r^{-1}, & d=2,\\[2pt]
    r^{-(d-2)}, & d\ge3
  \end{cases}
\end{align}
as $r\downarrow0$.
\end{lemma}
\begin{proof}
Fix $\omega\ge0$. By Theorem~\ref{thm:shoot},
$\phi_\omega=u_{y^*}$. The regular solutions $v_\omega^+$ and
$v_{y^*}$ satisfy the same equation and the same initial conditions
$v(0)=1$, $v'(0)=0$, so they coincide. Since
$\{w_\omega^+,w_\omega^-\}$ is a fundamental system, we can write
\[
  v_\omega^+=c_1w_\omega^++c_2w_\omega^-
\]
for some $c_1,c_2\in\R$. Proposition~\ref{prop:vstar} gives
$c_1<0$; in particular, $c_1\ne0$. Since
$w_\omega^-/w_\omega^+\to0$ as $r\to\infty$ by
Proposition~\ref{prop:wasym}, we obtain
\[
  \frac{v_\omega^+(r)}{w_\omega^+(r)}
  \to c_1<0.
\]
This proves the first assertion and also shows that
$v_\omega^+$ and $w_\omega^-$ are linearly independent.

To determine the behavior of $w_\omega^-$ near the origin, we write
\begin{align}
\label{vom-rep}
  w_\omega^-
  =a_1v_\omega^+ +a_2v_\omega^-.
\end{align}
If $a_2=0$, then $w_\omega^-$ and $v_\omega^+$ would be linearly
dependent. Hence $a_2\ne0$.
Since $v_\omega^+$ is regular at the origin whereas
\[
  v_\omega^-(r)\sim
  \begin{cases}
    \log r^{-1}, & d=2,\\
    r^{-(d-2)}, & d\ge3
  \end{cases}
\]
as $r\downarrow0$, the singular term in \eqref{vom-rep} dominates.
Hence \eqref{w-asym0} follows from $a_2\ne0$.
\end{proof}

For fixed $\omega\ge0$, let $W(r)$ be the Wronskian of $v_\omega^+$ and $w_\omega^-$:
\[  W(r)
    \ce \begin{vmatrix}
    v_\omega^+(r) & w_\omega^-(r)
\\  (v_\omega^+)'(r) & (w_\omega^-)'(r)
    \end{vmatrix},\qquad
    r>0.    \]

\begin{lemma}\label{lem:Wom}
$W(r)=W_\omega r^{-(d-1)}$ for some constant $W_\omega>0$.
\end{lemma}
\begin{proof}
The Wronskian identity (see \eqref{eq:W}) gives
\[
  \frac{d}{dr}\bigl(r^{d-1}W(r)\bigr)=0,\qquad r>0.
\]
Hence $W(r)=W_\omega r^{-(d-1)}$ for some $W_\omega\in\R$.
By the proof of Lemma~\ref{lem:vw}, we can write
\[
  v_\omega^+=c_1w_\omega^++c_2w_\omega^-,\qquad c_1<0.
\]
The terms involving $c_2$ cancel in the Wronskian, so
\[
  W(r)=c_1\bigl(w_\omega^+(r)(w_\omega^-)'(r)
                 -(w_\omega^+)'(r)w_\omega^-(r)\bigr).
\]
Proposition~\ref{prop:wasym} and Remark~\ref{rem:wprime} yield
\[
  \lim_{r\to\infty}r^{d-1}
    \bigl(w_\omega^+(r)(w_\omega^-)'(r)
          -(w_\omega^+)'(r)w_\omega^-(r)\bigr)
  =\begin{cases}
    -2\sqrt{\omega}, & \omega>0,\\[2pt]
    \lambda_--\lambda_+, & \omega=0,\ p\ne p^*,\\[2pt]
    -(d-2), & \omega=0,\ p=p^*.
  \end{cases}
\]
Each limit is strictly negative. Since $c_1<0$, we conclude that
$W_\omega>0$.
\end{proof}

\begin{lemma}\label{lem:int}
Assume that $\phi_0\in L^2(\R^d)$. Then, for every $\omega\ge0$,
$w_\omega^-\phi_\omega r^{d-1}\in L^1((0,\infty))$.
\end{lemma}

\begin{proof}
We first verify integrability near the origin.
Fix $\omega\ge0$ and choose $\delta_\omega\in(0,1)$
such that
\[
  \frac12\le v_\omega^+(r)\le2,\qquad 0\le r\le \delta_\omega.
\]
The Wronskian identity gives
\[
  \left(\frac{w_\omega^-}{v_\omega^+}\right)'
  =\frac{W_\omega r^{1-d}}{(v_\omega^+)^2},
\]
and hence, for $0<r\le \delta_\omega$,
\[
  \frac{w_\omega^-(r)}{v_\omega^+(r)}
  =\frac{w_\omega^-(\delta_\omega)}{v_\omega^+(\delta_\omega)}
   -W_\omega\int_r^{\delta_\omega}\frac{s^{1-d}}{v_\omega^+(s)^2}\,ds.
\]
Consequently, for $0<r\le\delta_\omega$,
\begin{equation}\label{eq:wnear}
  |w_\omega^-(r)|
  \le C\bigl(|w_\omega^-(\delta_\omega)|+W_\omega\bigr)
  \begin{cases}
    1+\log r^{-1}, & d=2,\\
    r^{-(d-2)}, & d\ge3.
  \end{cases}
\end{equation}
Here $C$ depends only on $d$. Since $\phi_\omega$ is bounded near
the origin, for the same range of $r$ we obtain
\[
  |w_\omega^-(r)\phi_\omega(r)r^{d-1}|
  \lesssim
  \begin{cases}
    r(1+\log r^{-1}), & d=2,\\
    r, & d\ge3.
  \end{cases}
\]
Both bounds are integrable near the origin.

We next verify integrability at infinity. For $\omega>0$, both
factors decay exponentially. Suppose that $\omega=0$. Choose $R_0>1$
so that $w_0^->0$ on $[R_0,\infty)$ and
\[
  V_0(r)>0,\quad
  L_0\phi_0(r)
  =(p-1)\phi_0(r)^p-(q-1)\phi_0(r)^q\ge0,
  \quad r\ge R_0.
\]
This choice is possible by Proposition~\ref{prop:wasym}
and the fact that $\phi_0(r)\to0$. Choose $C>0$ such that
$C\phi_0(R_0)\ge w_0^-(R_0)$, and set $h\ce C\phi_0-w_0^-$.
Then
\[
  L_0h\ge0\quad\text{on }\{r>R_0\},\quad
  h(R_0)\ge0,\quad \lim_{r\to\infty}h(r)=0.
\]
Since $V_0>0$ on this exterior domain, the maximum principle gives
\[
  0<w_0^-(r)\le C\phi_0(r),\qquad r\ge R_0.
\]
Consequently,
\[
  \int_{R_0}^\infty w_0^-(r)\phi_0(r)r^{d-1}\,dr
  \le C\int_{R_0}^\infty\phi_0(r)^2r^{d-1}\,dr<\infty
\]
by $\phi_0\in L^2(\R^d)$. Together with the estimate near the
origin and continuity on compact subintervals of $(0,\infty)$,
this proves the lemma.
\end{proof}

Recall that $\eta_\omega=\pt_\omega\phi_\omega$ for $\omega>0$,
defined in \eqref{eq:eta-def}, satisfies \eqref{eq:Lometa} and hence
\eqref{eq:etaode}. We first represent $\eta_\omega$ in terms of
$v_\omega^+$ and $w_\omega^-$, and then use this representation to
define the corresponding function $\eta_0$.
\begin{proposition}\label{prop:eta}
For $\omega>0$, the function $\eta_\omega$ defined by
\eqref{eq:eta-def} has the representation
\begin{align} \label{eq:etaom}
    \eta_\omega(r)
    =\begin{multlined}[t]
    \frac{v_\omega^+(r)}{W_\omega}
    \int_r^\infty w_\omega^-(s)\phi_\omega(s)s^{d-1}\,ds
\\  +\frac{w_\omega^-(r)}{W_\omega}
    \int_0^r v_\omega^+(s)\phi_\omega(s)s^{d-1}\,ds,
    \qquad r>0.
    \end{multlined}
\end{align}
\end{proposition}

\begin{proof}
Fix $\omega>0$. By variation of parameters, the general solution of
\eqref{eq:etaode} can be written as
\begin{align} \label{eq:eta-gen}
    \eta(r)
    =\begin{multlined}[t]
    C_1v_\omega^+(r)+C_2w_\omega^-(r)
    +\frac{v_\omega^+(r)}{W_\omega}
    \int_r^\infty w_\omega^-(s)\phi_\omega(s)s^{d-1}\,ds
\\  +\frac{w_\omega^-(r)}{W_\omega}
    \int_0^r v_\omega^+(s)\phi_\omega(s)s^{d-1}\,ds
    \end{multlined}
\end{align}
for some $C_1,C_2\in\R$. The first integral converges by exponential decay.

For $\eta=\eta_\omega$, we show that $C_1=C_2=0$.
First, we show that $C_1=0$. By Lemma~\ref{lem:vw} and
Proposition~\ref{prop:wasym}, fix a sufficiently large $R_1>0$ so that
\[
  v_\omega^+(r)<0,\quad w_\omega^-(r)>0,\quad r\ge R_1.
\]
Since $\phi_\omega>0$, the integrand
$v_\omega^+(s)\phi_\omega(s)s^{d-1}$ is negative on
$[R_1,\infty)$. 
Moreover, \eqref{eq:phi-exp} and \eqref{eq:wasym}
imply that
$\phi_\omega(r)\sim w_\omega^-(r)$ as $r\to\infty$.
Hence, we have
\begin{align*}
  v_\omega^+(r)
  \int_r^\infty w_\omega^-(s)\phi_\omega(s)s^{d-1}\,ds
  &\sim -r^{-\frac{d-1}{2}}e^{\sqrt{\omega}r}
       \int_r^\infty e^{-2\sqrt{\omega}s}\,ds
\\&\sim -r^{-\frac{d-1}{2}}e^{-\sqrt{\omega}r}.
\end{align*}
For the second integral term, we split
\[
  \int_0^r v_\omega^+(s)\phi_\omega(s)s^{d-1}\,ds
  =\int_0^{R_1}v_\omega^+(s)\phi_\omega(s)s^{d-1}\,ds
   +\int_{R_1}^r v_\omega^+(s)\phi_\omega(s)s^{d-1}\,ds.
\]
The first term on the right-hand side is constant, whereas
\[
  v_\omega^+(s)\phi_\omega(s)s^{d-1}\sim-1
  \quad\text{as }s\to\infty.
\]
Since this integrand is negative on $[R_1,\infty)$, it follows that
\[
  \int_{R_1}^r v_\omega^+(s)\phi_\omega(s)s^{d-1}\,ds\sim-r,
\]
and hence
\begin{align*}
  w_\omega^-(r)
  \int_0^r v_\omega^+(s)\phi_\omega(s)s^{d-1}\,ds
  &\sim-rw_\omega^-(r)
   \sim-r^{-\frac{d-3}{2}}e^{-\sqrt{\omega}r}.
\end{align*}
Thus both integral terms in \eqref{eq:eta-gen} tend to zero as
$r\to\infty$. Since $w_\omega^-(r)\to0$ and
$\eta_\omega\in H_{\rad}^1(\R^d)$ implies $\eta_\omega(r)\to0$,
whereas $v_\omega^+(r)\to-\infty$, it follows that $C_1=0$.

Next, we show that $C_2=0$. By Lemma~\ref{lem:vw},
$|w_\omega^-(r)|\to\infty$ as $r\downarrow0$. The same asymptotics and
the boundedness of $v_\omega^+$ and $\phi_\omega$ near the origin
show that the first integral term in \eqref{eq:eta-gen} remains
bounded there, while
\[
  |w_\omega^-(r)|
  \int_0^r |v_\omega^+(s)|\phi_\omega(s)s^{d-1}\,ds
  \lesssim
  \begin{cases}
    r^2(1+\log r^{-1}), & d=2,\\
    r^2, & d\ge3
  \end{cases}
\]
as $r\downarrow0$. In particular, this term tends to zero.
Since $\eta_\omega\in H^1(\R^d)$ satisfies \eqref{eq:Lometa},
elliptic regularity gives $\eta_\omega\in C^2(\R^d)$. Since $w_\omega^-$ is singular at the origin, we must have $C_2=0$.
This proves the representation.
\end{proof}

Assume now that $\phi_0\in L^2(\R^d)$. Motivated by
\eqref{eq:etaom}, we define
\begin{align}\label{eq:eta0}
    \eta_0(r)
    \ce \begin{multlined}[t]
    \frac{v_0^+(r)}{W_0}
    \int_r^\infty w_0^-(s)\phi_0(s)s^{d-1}\,ds
\\  +\frac{w_0^-(r)}{W_0}
    \int_0^r v_0^+(s)\phi_0(s)s^{d-1}\,ds,
    \qquad r>0,
    \end{multlined}
\end{align}
and set
\begin{align*}
 \eta_0(0)&\ce
  \frac1{W_0}\int_0^\infty w_0^-(s)\phi_0(s)s^{d-1}\,ds.
\end{align*}
Lemma~\ref{lem:int} and the asymptotics near the origin in
Lemma~\ref{lem:vw} show that this defines a continuous function on
$[0,\infty)$.
We will also prove in Theorem~\ref{thm:etalim} that
\[
  \eta_0(x)=\lim_{\omega\downarrow0}\eta_\omega(x)
  \quad\text{for all }x\in\R^d.
\]

The following proposition establishes that $\eta_0$ is a regular
radial solution of \eqref{eq:Lometa} with $\omega=0$
and determines its asymptotic behavior.
For $p\ne p^*$, the proof uses the polynomial
\[
  P(\lambda)\ce\lambda^2+(d-2)\lambda-c_p,
  \qquad c_p=(\alpha+2)(\alpha+2-d).
\]
By \eqref{eq-lam}, $P$ has roots $\lambda_-<\lambda_+$,
and hence
\[
  P(\lambda)<0
  \quad\Longleftrightarrow\quad
  \lambda\in(\lambda_-,\lambda_+).
\]
\begin{proposition}\label{prop:eta0}
Assume that $\phi_0\in L^2(\R^d)$. The function $\eta_0$ defined by
\eqref{eq:eta0} belongs to $C_{\rad}^2(\R^d)$ and solves \eqref{eq:Lometa} with $\omega=0$. If $(d,p)\ne(4,2)$, then
\begin{equation}\label{eq:eta-dec}
    \eta_0(r)\sim-r^2\phi_0(r)
    \quad\text{as }r\to\infty.
\end{equation}
In the case $(d,p)=(4,2)$, we have
\[
    \eta_0(r)\sim-1
    \quad\text{as }r\to\infty.
\]
The solution is unique among $C_{\rad}^2(\R^d)$ solutions with this
behavior at infinity.
\end{proposition}

\begin{proof}
A direct differentiation of \eqref{eq:eta0} shows that $\eta_0$ solves
\eqref{eq:etaode} with $\omega=0$ on $(0,\infty)$.
By its definition and Lemmas~\ref{lem:int} and~\ref{lem:vw},
$\eta_0(r)\to\eta_0(0)$ as $r\downarrow0$.
When we differentiate \eqref{eq:eta0}, the terms from the upper and
lower limits cancel, and we obtain
\begin{align}
\label{eta0prim}
  \eta_0'(r)
  ={}
\begin{multlined}[t]  
  \frac{(v_0^+)'(r)}{W_0}
       \int_r^\infty w_0^-(s)\phi_0(s)s^{d-1}\,ds
    \\  +\frac{(w_0^-)'(r)}{W_0}
       \int_0^r v_0^+(s)\phi_0(s)s^{d-1}\,ds,\qquad r>0.
 \end{multlined}      
\end{align}
The first integral is bounded as $r\downarrow0$ by
Lemma~\ref{lem:int}, and $(v_0^+)'(r)=O(r)$ by
Remark~\ref{rem:wprime}. The second integral is $O(r^d)$, since
$v_0^+$ and $\phi_0$ are bounded near the origin. Moreover,
\eqref{vom-rep} and Remark~\ref{rem:wprime} give
$(w_0^-)'(r)=O(r^{1-d})$. Thus both terms
are $O(r)$, and
\[
  |\eta_0'(r)|\lesssim r,\qquad 0<r\ll1.
\]
This implies that $\eta_0\in C^1([0,\infty))$ and $\eta_0'(0)=0$.

Set $F_0\ce V_0\eta_0+\phi_0\in C([0,\infty))$. Writing the ODE in
divergence form and using
$\lim_{r\downarrow0}r^{d-1}\eta_0'(r)=0$, we obtain
\[
  r^{d-1}\eta_0'(r)=\int_0^r s^{d-1}F_0(s)\,ds.
\]
Dividing by $r^d$ and substituting $s=rt$ yields
\begin{equation}\label{eq:etaratio}
  \frac{\eta_0'(r)}r
  =\int_0^1 t^{d-1}F_0(rt)\,dt,\qquad r>0.
\end{equation}
By continuity of $F_0$, the right-hand side converges to $F_0(0)/d$
as $r\downarrow0$. Since $\eta_0'(0)=0$, this proves that
$\eta_0''(0)=F_0(0)/d$. The ODE also gives
\[
  \eta_0''(r)
  =F_0(r)-(d-1)\frac{\eta_0'(r)}r
 \xrightarrow[r\downarrow0]{}  \frac{F_0(0)}d.
\]
Thus $\eta_0\in C^2([0,\infty))$.
Since $\eta_0'(0)=0$ and $\eta_0'(r)/r\to\eta_0''(0)$ as
$r\downarrow0$, the radial extension $x\mapsto\eta_0(|x|)$ also
belongs to $C^2(\R^d)$.

By Lemma~\ref{lem:vw} and Proposition~\ref{prop:wasym}, fix a
sufficiently large $R_2>0$ so that
\[
  v_0^+(r)<0,\quad w_0^-(r)>0,\quad r\ge R_2.
\]
Since $\phi_0>0$, the integrand
$v_0^+(s)\phi_0(s)s^{d-1}$ is negative on $[R_2,\infty)$.
Suppose first that $p\ne p^*$.
Since $\phi_0\in L^2(\R^d)$, we have $\alpha>d/2$.
Moreover, $\alpha\ge d-2$, with equality when $p>p^*$.
Hence
\begin{align*}
  P(\alpha-d)
  &=(\alpha-d)(\alpha-2)
    -(\alpha+2)(\alpha+2-d)\\
  &=-2(3\alpha+2-2d)<0.
\end{align*}
Since $P$ has roots $\lambda_-<\lambda_+$, we obtain
\begin{equation}\label{eq:lam-alpha}
  \lambda_-<\alpha-d<\lambda_+.
\end{equation}
Since $\lambda_-+\lambda_+=2-d$,
\eqref{eq:lam-alpha} also gives
\begin{equation}\label{eq:lam-alpha2}
  \lambda_-<2-\alpha<\lambda_+.
\end{equation}
Thus
\[
  w_0^-(r)\int_0^{R_2}v_0^+(s)\phi_0(s)s^{d-1}\,ds
  =O(r^{\lambda_-})=o(r^{2-\alpha}).
\]
Using these inequalities and $\phi_0(r)\sim r^{-\alpha}$, we obtain,
as $r\to\infty$,
\begin{align*}
  \frac{v_0^+(r)}{W_0}
  \int_r^\infty w_0^-(s)\phi_0(s)s^{d-1}\,ds
  &\sim-r^{\lambda_+}
       \int_r^\infty s^{\lambda_--\alpha+d-1}\,ds
\\&\sim-r^{\lambda_-+\lambda_+-\alpha+d}
       \sim-r^{2-\alpha}\sim-r^2\phi_0(r),
\\
  \frac{w_0^-(r)}{W_0}
  \int_0^r v_0^+(s)\phi_0(s)s^{d-1}\,ds
  &\sim-r^{\lambda_-}
       \int_{R_2}^r s^{\lambda_+-\alpha+d-1}\,ds
\\&\sim-r^{\lambda_-+\lambda_+-\alpha+d}
       \sim-r^{2-\alpha}\sim-r^2\phi_0(r).
\end{align*}
Here we used $W_0>0$ from Lemma~\ref{lem:Wom}. Both terms have
negative leading coefficients, so \eqref{eq:eta-dec} follows.

Suppose next that $p=p^*$. Then $d\ge4$ by
$\phi_0\in L^2(\R^d)$. As $r\to\infty$,
\[
  \int_r^\infty s^{3-d}(\log s)^{-(d-1)}\,ds
  \sim
  \begin{cases}
    (\log r)^{-2}, & d=4,\\[2pt]
    r^{4-d}(\log r)^{-(d-1)}, & d\ge5.
  \end{cases}
\]
Combining this with \eqref{eq:phidec}, \eqref{eq:wasym}, and
Lemma~\ref{lem:vw}, we obtain
\begin{align*}
  \frac{v_0^+(r)}{W_0}
  \int_r^\infty w_0^-(s)\phi_0(s)s^{d-1}\,ds
  &\sim
  \begin{cases}
    -1, & d=4,\\[2pt]
    -r^2\phi_0(r), & d\ge5,
  \end{cases}
\\
  \frac{w_0^-(r)}{W_0}
  \int_0^r v_0^+(s)\phi_0(s)s^{d-1}\,ds
  &\sim-r^{-(d-2)}(\log r)^{-d/2}
       \int_{R_2}^r s\log s\,ds
\\&\sim-r^{4-d}(\log r)^{-(d-2)/2}
       \sim-r^2\phi_0(r).
\end{align*}
In the second estimate, the integral over $(0,R_2)$ is a finite
constant independent of $r$. After multiplication by $w_0^-(r)/W_0$,
its contribution satisfies
\[
  \frac{w_0^-(r)}{W_0}
  \int_0^{R_2}v_0^+(s)\phi_0(s)s^{d-1}\,ds
  =O(w_0^-(r))=o(r^2\phi_0(r)),
\]
since $w_0^-(r)/(r^2\phi_0(r))=O(r^{-2}(\log r)^{-1})$ as $r\to\infty$.
Thus this part of the integral does not affect the leading term.
For $d\ge5$, both terms have negative leading coefficients and the
same order, so \eqref{eq:eta-dec} follows. For $d=4$, the second term
is $O((\log r)^{-1})$, and hence $\eta_0(r)\sim-1$.

Finally, to prove uniqueness, we first show that
\begin{align}
\label{eta0o}
  \eta_0(r)=o(w_0^+(r))\quad\text{as }r\to\infty.
\end{align}
By \eqref{eq:wasym}, $w_0^+(r)$ tends either to $1$ or to $+\infty$,
so any $o(1)$ bound for $\eta_0$ implies this comparison.
If $d\ge5$, \eqref{eq:phidec} gives $\phi_0(r)=O(r^{-(d-2)})$,
and hence $\eta_0(r)=O(r^{4-d})=o(1)$, independently of $p$.
If $d=4$, for $p<2$ we have
$\alpha>2$ and $\eta_0(r)=O(r^{2-\alpha})=o(1)$.
For $p=2$, the function $\eta_0$ is bounded, whereas
$w_0^+(r)\sim(\log r)^2\to\infty$.
For $d=2,3$, \eqref{l2} implies $p<p_{\mathrm{c}}(d)<p^*$ and
$\alpha=2/(p-1)>d/2$. If $p<2$, then $\alpha>2$ and again
$\eta_0(r)=o(1)$. 
If $2\le p<p_{\mathrm c}(d)$, then
\eqref{eq:lam-alpha2} and \eqref{eq:wasym} give
\[
  \frac{\eta_0(r)}{w_0^+(r)}
  =O(r^{2-\alpha-\lambda_+})\to0\quad\text{as }r\to\infty.
\]
This proves \eqref{eta0o} in every case.

The general regular radial solution of \eqref{eq:Lometa} with
$\omega=0$ has the form
\[
  \eta=c v_0^++\eta_0,\qquad c\in\R,
\]
because $\eta-\eta_0$ solves the homogeneous equation and is regular
at the origin. If $\eta$ has the stated behavior at infinity, the
same estimates give $\eta=o(w_0^+)$. Since
$v_0^+\sim-w_0^+$ by Lemma~\ref{lem:vw}, the relation
$c v_0^+=\eta-\eta_0=o(w_0^+)$ forces $c=0$.
This proves uniqueness.
\end{proof}

Using the representation in \eqref{eta0prim}, we obtain the following derivative estimate.

\begin{lemma}
\label{lem:deta0}
Assume that $\phi_0\in L^2(\R^d)$. Then,
\begin{equation}\label{eq:deta0}
  \eta_0'(r)=
  \begin{cases}
    O(r^{-\alpha+1}), & p\ne p^*,\\[2pt]
    O\bigl(r^{-d+3}(\log r)^{-(d-2)/2}\bigr), & p=p^*,
  \end{cases}
\end{equation}
as $r\to\infty$.
\end{lemma}

\begin{proof}
We estimate the two terms in \eqref{eta0prim}.
Suppose first that $p\ne p^*$. The connection formula
$v_0^+=c_1w_0^++c_2w_0^-$ in the proof of Lemma~\ref{lem:vw},
together with Remark~\ref{rem:wprime}, gives
\[
  |(v_0^+)'(r)|\lesssim r^{\lambda_+-1},
  \qquad
  |(w_0^-)'(r)|\lesssim r^{\lambda_--1}.
\]
By \eqref{eq:lam-alpha}, we have $\lambda_--\alpha+d<0<\lambda_+-\alpha+d$.
Hence
\begin{align*}
  \int_r^\infty |w_0^-(s)|\phi_0(s)s^{d-1}\,ds
  &\lesssim r^{\lambda_--\alpha+d},\\
  \int_0^r |v_0^+(s)|\phi_0(s)s^{d-1}\,ds
  &\lesssim r^{\lambda_+-\alpha+d}.
\end{align*}
By \eqref{eta0prim} and $\lambda_-+\lambda_+=-(d-2)$, we obtain
\[
  |\eta_0'(r)|
  \lesssim r^{\lambda_-+\lambda_+-\alpha+d-1}
  =r^{1-\alpha}.
\]

Suppose next that $p=p^*$, so that $d\ge4$.
The same connection formula and Remark~\ref{rem:wprime} give
\[
  |(v_0^+)'(r)|\lesssim r^{-1}(\log r)^{d/2-1},
  \qquad
  |(w_0^-)'(r)|\lesssim r^{1-d}(\log r)^{-d/2}.
\]
Using the integral estimates in the proof of
Proposition~\ref{prop:eta0}, the first term in \eqref{eta0prim}
satisfies
\begin{align*}
  \left|
    \frac{(v_0^+)'(r)}{W_0}
    \int_r^\infty w_0^-(s)\phi_0(s)s^{d-1}\,ds
  \right|
  &\lesssim r^{-1}(\log r)^{d/2-1}
    \int_r^\infty s^{3-d}(\log s)^{-(d-1)}\,ds\\
  &\lesssim
  \begin{cases}
    r^{-1}(\log r)^{-1}, & d=4,\\[2pt]
    r^{3-d}(\log r)^{-d/2}, & d\ge5.
  \end{cases}
\end{align*}
For the second term, the bound
\[
  \int_0^r |v_0^+(s)|\phi_0(s)s^{d-1}\,ds
  \lesssim r^2\log r
\]
gives
\[
  \left|
    \frac{(w_0^-)'(r)}{W_0}
    \int_0^r v_0^+(s)\phi_0(s)s^{d-1}\,ds
  \right|
  \lesssim r^{3-d}(\log r)^{-(d-2)/2}.
\]
Combining the two estimates proves \eqref{eq:deta0}, including $(d,p)=(4,2)$.
\end{proof}

\subsection{Convergence as \texorpdfstring{$\omega\downarrow0$}{omega decreases to zero}}

We first establish a uniform comparison estimate for the decaying
fundamental solutions.
\begin{lemma}
\label{lem:wbound}
There exist $R_0>1$ and $C>0$ such that $w_\omega^-$ is positive and
strictly decreasing on $[R_0,\infty)$ for every $0<\omega<1$.
Upon defining
\[
  \widetilde w_\omega^-(r)
  \ce \frac{w_\omega^-(r)}{w_\omega^-(R_0)},
  \qquad r>0,
\]
we have
\begin{equation}\label{eq:w-unif}
  0<\widetilde w_\omega^-(r)\le1,
  \quad \widetilde w_\omega^-(r)\le C\phi_\omega(r)
  \quad (r\ge R_0,\enspace 0<\omega<1).
\end{equation}
In particular, for $r\ge R_0$,
\begin{equation}\label{eq:wdec}
  \sup_{0<\omega<1}\widetilde w_\omega^-(r)
  \lesssim
  \begin{cases}
    r^{-\alpha}, & p\ne p^*,\\[2pt]
    r^{-(d-2)}(\log r)^{-(d-2)/2}, & p=p^*.
  \end{cases}
\end{equation}
Here $\alpha$ is defined in \eqref{eq:alpha}.
\end{lemma}

\begin{proof}
By \cite[Proposition~3.5]{FH21}, we have
\begin{equation}\label{eq:phiu}
  \sup_{0<\omega<1}\phi_\omega(r)
  \lesssim\min\{r^{-2/(p-1)},r^{-(d-2)}\},
  \qquad r\gg1.
\end{equation}
When $p=p^*$, \cite[Proposition~3.7]{FH21} gives the sharper estimate
\begin{equation}\label{eq:phicrit}
  \sup_{0<\omega<1}\phi_\omega(r)
  \lesssim r^{-(d-2)}(\log r)^{-(d-2)/2},
  \qquad r\gg1.
\end{equation}
Thus $\phi_\omega(r)\to0$ as $r\to\infty$, uniformly for
$0<\omega<1$. We may
choose $R_0>1$ so large that, for every $r\ge R_0$ and
$0<\omega<1$,
\begin{align*}
  V_\omega(r)&>0,\\
  L_\omega\phi_\omega(r)
    &=(p-1)\phi_\omega(r)^p-(q-1)\phi_\omega(r)^q\ge0.
\end{align*}
For each fixed $\omega$, the asymptotics in
Proposition~\ref{prop:wasym} and Remark~\ref{rem:wprime} give
$w_\omega^->0$ and $(w_\omega^-)'<0$ for sufficiently large $r$.
Choose $R>R_0$ so that these properties hold for $r\ge R$.
Suppose that $(w_\omega^-)'$ has a zero in $[R_0,R)$,
and let $r_*$ be its largest zero.
Then $(w_\omega^-)'<0$ on $(r_*,R]$, so
\[
  w_\omega^-(r_*)>w_\omega^-(R)>0,
  \qquad
  (w_\omega^-)''(r_*)\le0.
\]
On the other hand, \eqref{eq:pdeom} gives
\[
  (w_\omega^-)''(r_*)
  =(\omega+V_\omega(r_*))w_\omega^-(r_*)>0,
\]
a contradiction.
Therefore $(w_\omega^-)'<0$ on $[R_0,\infty)$,
and hence
\[
  0<\widetilde w_\omega^-(r)\le1,
  \qquad r\ge R_0.
\]

By Proposition~\ref{prop:reg}, \cite[Theorem~1.6]{FH21}, and local
elliptic regularity, $\omega\mapsto\phi_\omega(R_0)$ is continuous
on $(0,1]$, and its limit as $\omega\downarrow0$ is $\phi_0(R_0)>0$. Hence
\[
  \inf_{0<\omega<1}\phi_\omega(R_0)>0.
\]
Choose $C>0$, independently of $\omega$, so that
$C\phi_\omega(R_0)\ge1$. Applying the maximum principle to
$C\phi_\omega-\widetilde w_\omega^-$ as in the proof of
Lemma~\ref{lem:int}, we obtain
\[
  \widetilde w_\omega^-(r)\le C\phi_\omega(r),
  \quad r\ge R_0,\quad 0<\omega<1.
\]
Together with $0<\widetilde w_\omega^-\le1$, this proves
\eqref{eq:w-unif}. Combining this comparison with \eqref{eq:phiu}
and \eqref{eq:phicrit} proves \eqref{eq:wdec}.
\end{proof}

We now use Lemma~\ref{lem:wbound} to establish the local $C^2$ convergence of $\eta_\omega$ to $\eta_0$ as $\omega\downarrow0$.

\begin{theorem}\label{thm:etalim}
Assume that $\phi_0\in L^2(\R^d)$. Let $\eta_0$ be the function
defined by \eqref{eq:eta0}. Then
\[
  \eta_\omega\to\eta_0
  \quad\text{in}~C_{\mathrm{loc}}^2([0,\infty))
  \quad\text{as }\omega\downarrow0.
\]
\end{theorem}

\begin{proof}
Choose $R_0$ as in Lemma~\ref{lem:wbound}, large enough that
$w_0^-(R_0)>0$, and use the normalization in that lemma for
$0<\omega<1$. For the limiting solution, set
\[
  \widetilde w_0^-(r)\ce \frac{w_0^-(r)}{w_0^-(R_0)}.
\]

\medskip
\noindent\textbf{Step 1:} Local convergence of the fundamental solutions.
By \cite[Theorem~1.6]{FH21} and local elliptic regularity,
$\phi_\omega\to\phi_0$ in $C_{\mathrm{loc}}^2(\R^d)$ as
$\omega\downarrow0$. Thus
$\omega+V_\omega\to V_0$ locally uniformly, including at the origin.
The regular initial conditions $v_\omega^+(0)=1$ and
$(v_\omega^+)'(0)=0$, together with the integral form of the equation,
give
\[
  v_\omega^+\to v_0^+
  \quad\text{in }C^2([0,R])\quad\text{for every }R>0.
\]

For the solutions selected by decay at infinity, we first check the
initial derivatives at $R_0$. By \eqref{eq:w-unif} and the mean value
theorem, there exists $r_\omega\in(R_0,R_0+1)$ such that
$(\widetilde w_\omega^-)'(r_\omega)\in[-1,0]$. Integrating the
equation in divergence form yields
\begin{align*}
  R_0^{d-1}(\widetilde w_\omega^-)'(R_0)
  ={}&r_\omega^{d-1}(\widetilde w_\omega^-)'(r_\omega)\\
     &-\int_{R_0}^{r_\omega}
       s^{d-1}(\omega+V_\omega(s))\widetilde w_\omega^-(s)\,ds.
\end{align*}
The coefficients are uniformly bounded on $[R_0,R_0+1]$ and
$0<\widetilde w_\omega^-\le1$, so
$-C\le(\widetilde w_\omega^-)'(R_0)\le0$ uniformly in $\omega$.

Given any sequence $\omega_n\downarrow0$, pass to a subsequence
such that $(\widetilde w_{\omega_n}^-)'(R_0)\to\ell$ as $n\to\infty$.
Continuous dependence for the initial-value problem gives
$\widetilde w_{\omega_n}^-\to w$ in $C^2([R_0,R])$ for every
$R>R_0$, where
\[
  w''+\frac{d-1}{r}w'-V_0w=0,
  \quad w(R_0)=1,\quad w'(R_0)=\ell.
\]
Passing to the limit in \eqref{eq:wdec} shows that $w(r)\to0$ as
$r\to\infty$. By Proposition~\ref{prop:wasym}, $w_0^+$ does not tend
to zero, whereas $w_0^-$ does. Hence $w$ is a constant multiple of
$w_0^-$, and $w(R_0)=1$ implies $w=\widetilde w_0^-$. Every
subsequential limit of the initial derivatives is therefore
$(\widetilde w_0^-)'(R_0)$. It follows that
\[
  (\widetilde w_\omega^-)'(R_0)
  \to(\widetilde w_0^-)'(R_0)
  \quad\text{as }\omega\downarrow0.
\]
Applying continuous dependence both to the left and to the right of
$R_0$, we obtain, as $\omega\downarrow0$,
\begin{equation*}
  \widetilde w_\omega^-\to\widetilde w_0^-
  \quad\text{in }C^2([\delta,R]),
  \quad 0<\delta<R<\infty.
\end{equation*}
Set
\[
  \widetilde W_\omega\ce \frac{W_\omega}{w_\omega^-(R_0)},
  \qquad
  \widetilde W_0\ce \frac{W_0}{w_0^-(R_0)}.
\]
The Wronskian identity at $R_0$ then gives
\begin{equation}\label{eq:Wlim}
  \begin{split}
    \widetilde W_\omega
    &=R_0^{d-1}\bigl(
       v_\omega^+(R_0)(\widetilde w_\omega^-)'(R_0)
       -(v_\omega^+)'(R_0)\bigr) \xrightarrow[\omega\downarrow0]{} \widetilde W_0>0.
  \end{split}
\end{equation}
The representations \eqref{eq:etaom} and \eqref{eq:eta0} are unchanged
when $(w_\omega^-,W_\omega)$ is replaced by
$(\widetilde w_\omega^-,\widetilde W_\omega)$.

\medskip
\noindent\textbf{Step 2:} Convergence at the origin.
By Step~1, choose $\delta_0\in(0,1)$ and
$\omega_0\in(0,1)$ so that
\[
  \frac12\le v_\omega^+(r)\le2,
  \quad 0\le r\le \delta_0,\quad 0<\omega<\omega_0.
\]
By Step~1, the quantities $\widetilde w_\omega^-(\delta_0)$,
$\widetilde W_\omega$, and $\|\phi_\omega\|_{C([0,\delta_0])}$
are uniformly bounded for $0<\omega<\omega_0$.
Applying \eqref{eq:wnear} to the normalized
pair $(\widetilde w_\omega^-,\widetilde W_\omega)$ therefore gives,
for $0<r\le\delta_0$ and $0<\omega<\omega_0$,
\[
  |\widetilde w_\omega^-(r)\phi_\omega(r)r^{d-1}|
  \lesssim
  \begin{cases}
    r(1+\log r^{-1}), & d=2,\\
    r, & d\ge3.
  \end{cases}
\]
These bounds are integrable near the origin and independent of $\omega$.
At infinity, \eqref{eq:phiu}, \eqref{eq:phicrit}, and the sharp
asymptotics \eqref{eq:phidec} give
$\phi_\omega(r)\lesssim\phi_0(r)$ on $[R_0,\infty)$, uniformly
for $0<\omega<1$. Thus \eqref{eq:w-unif} implies
\[
  0<\widetilde w_\omega^-(r)\phi_\omega(r)r^{d-1}
  \lesssim\phi_\omega(r)^2r^{d-1}
  \lesssim\phi_0(r)^2r^{d-1},\qquad r\ge R_0.
\]
The last function is integrable by $\phi_0\in L^2(\R^d)$.
Together with the local convergence on $[\delta_0,R_0]$, these estimates provide an \(\omega\)-independent integrable majorant on \((0,\infty)\).

As shown in the proof of Proposition~\ref{prop:eta},
the term involving the integral over $(0,r)$ in \eqref{eq:etaom} tends to zero as $r\downarrow0$. Hence
\begin{equation*}
  \eta_\omega(0)
  =\frac1{\widetilde W_\omega}
   \int_0^\infty\widetilde w_\omega^-(s)\phi_\omega(s)s^{d-1}\,ds,
  \qquad 0<\omega<\omega_0.
\end{equation*}
The same formula holds for $\eta_0(0)$ with
$(\widetilde w_0^-,\widetilde W_0,\phi_0)$.
By dominated convergence and \eqref{eq:Wlim}, we conclude that
$\eta_\omega(0)\to\eta_0(0)$ as $\omega\downarrow0$.

\medskip
\noindent\textbf{Step 3:} Convergence on compact intervals.
We first prove that
\begin{align}
\label{Cloc}
  \eta_\omega\to\eta_0
  \quad\text{in }C([0,R])
  \quad\text{as }\omega\downarrow0
\end{align}
for every $R>0$. Fix $R>0$ and set
\[
  a_\omega\ce \omega+V_\omega,\qquad
  F_\omega\ce a_\omega\eta_\omega+\phi_\omega.
\]
The integral form of the equation for $\eta_\omega$ gives
\[
  \eta_\omega(r)
  =\eta_\omega(0)
  +\int_0^r t^{1-d}
    \int_0^t s^{d-1}F_\omega(s)\,ds\,dt,\qquad r>0,~\omega\ge0.
\]
Set $z_\omega\ce \eta_\omega-\eta_0$. One can write
\[
  F_\omega-F_0
  =a_\omega z_\omega
   +(a_\omega-a_0)\eta_0
   +\phi_\omega-\phi_0.
\]
To reduce the resulting double integral to a form suitable for Gronwall's lemma, we
interchange the order of integration over the region $0<s<t<r$:
\[
  \int_0^r t^{1-d}
    \int_0^t s^{d-1}h(s)\,ds\,dt
  =
  \int_0^r
    \left(
      s^{d-1}\int_s^r t^{1-d}\,dt
    \right)h(s)\,ds.
\]
For $0<s<r\le R$, the kernel satisfies
\[
  0\le
  s^{d-1}\int_s^r t^{1-d}\,dt
  =
  \int_s^r
    \left(\frac{s}{t}\right)^{d-1}\,dt
  \le r-s\le R.
\]
Since $a_\omega$ is uniformly bounded on $[0,R]$, we therefore obtain
\[
  |z_\omega(r)|
  \le
  C_R\varepsilon_\omega(R)
  +C_R\int_0^r |z_\omega(s)|\,ds,
  \qquad 0\le r\le R,
\]
where
\[
  \varepsilon_\omega(R)
  \ce
  |z_\omega(0)|
  +\|a_\omega-a_0\|_{C([0,R])}
  +\|\phi_\omega-\phi_0\|_{C([0,R])}.
\]
By Step~2 and $\phi_\omega\to\phi_0$ in $C_{\rm loc}(\R^d)$,
\[
  \varepsilon_\omega(R)\to0
  \quad\text{as }\omega\downarrow0.
\]
Then Gronwall's lemma yields
\[
  \|z_\omega\|_{C([0,R])}
  \le
  C_R\varepsilon_\omega(R)e^{C_RR}
  \to0,
\]
which proves \eqref{Cloc}.

Consequently, $F_\omega\to F_0$ uniformly on $[0,R]$ as
$\omega\downarrow0$. 
As in \eqref{eq:etaratio}, we have
\begin{equation}\label{eq:etar}
  \frac{\eta_\omega'(r)}r
  =\int_0^1 t^{d-1}F_\omega(rt)\,dt,
  \qquad r>0.
\end{equation}
This formula implies that
\[
  \|\eta_\omega'-\eta_0'\|_{C([0,R])}
  \le
  \frac{R}{d}
  \|F_\omega-F_0\|_{C([0,R])}
  \to0
  \quad\text{as }\omega\downarrow0.
\]
Moreover, by \eqref{eq:etar}, the quotient
$\eta_\omega'(r)/r$ extends continuously to $r=0$, and these continuous
extensions converge uniformly on $[0,R]$ as $\omega\downarrow0$.
Using
\[
  \eta_\omega''
  =F_\omega-(d-1)\frac{\eta_\omega'}r,
\]
we therefore obtain uniform convergence of the second derivatives on
$[0,R]$.
\end{proof}


\subsection{Weighted integrability and divergence of the mass slope}

We first relate the weighted $L^2$ condition to
the integrability of $\eta_0\phi_0$.

\begin{lemma}\label{lem:eta-pair}
Assume that $\phi_0\in L^2(\R^d)$. Then
\begin{align*}
  \phi_0\in\cF(H^1)
  &\quad\Longrightarrow\quad
  \eta_0\in\cF(H^{-1})\enspace\text{and}\enspace\eta_0\phi_0\in L^1(\R^d),
  \\
  \phi_0\notin\cF(H^1)
  &\quad\Longrightarrow\quad
  \int_{\R^d}\eta_0(x)\phi_0(x)\,dx=-\infty.
\end{align*}
\end{lemma}

\begin{proof}
We first suppose that $(d,p)\ne(4,2)$. Since $\eta_0$ and $\phi_0$
are continuous, $\eta_0\phi_0$ is integrable on every bounded subset
of $\R^d$, so it remains only to consider the behavior at infinity.
By Proposition~\ref{prop:eta0}, there exist $R>0$ and $c,C>0$ such that
\begin{equation}\label{eq:eta-comp}
  c r^2\phi_0(r)^2
  \le -\eta_0(r)\phi_0(r)
  \le C r^2\phi_0(r)^2,
  \qquad r\ge R.
\end{equation}
Hence, if $\phi_0\in\cF(H^1)$, then
\eqref{eq:eta-comp} gives $\eta_0\phi_0\in L^1(\R^d)$.
Moreover, the asymptotic relation
$\eta_0(r)\sim-r^2\phi_0(r)$ in Proposition~\ref{prop:eta0}
implies $\braket{x}^{-1}\eta_0\in L^2(\R^d)$, and thus
$\eta_0\in\cF(H^{-1})$.
If $\phi_0\notin\cF(H^1)$, then \eqref{eq:eta-comp} gives
\[
  \int_{\R^d}\eta_0(x)\phi_0(x)\,dx=-\infty.
\]

It remains to consider $(d,p)=(4,2)$. In this case,
\eqref{wl2} gives $\phi_0\notin\cF(H^1)$, while
Proposition~\ref{prop:eta0} and \eqref{eq:sharp} yield
\[
  \eta_0(r)\phi_0(r)r^3
  \sim-\frac{r}{\log r}
  \quad\text{as }r\to\infty.
\]
Therefore
\[
  \int_{\R^d}\eta_0(x)\phi_0(x)\,dx=-\infty,
\]
which completes the proof.
\end{proof}

We now use Lemma~\ref{lem:eta-pair} to show the divergence of $M'(\omega)$ as $\omega\downarrow0$ when the weighted $L^2$ condition fails.

\begin{theorem}
\label{thm:massdiv}
Assume that $\phi_0\in L^2(\R^d)$ and
$\phi_0\notin\cF(H^1)$.
Then
\[
  \lim_{\omega\downarrow0}M'(\omega)=-\infty.
\]
\end{theorem}

\begin{proof}
By Lemma~\ref{lem:eta-pair},
$\tbra[\eta_0,\phi_0]=-\infty$.
Moreover, Theorem~\ref{thm:etalim} gives
$\eta_\omega\to\eta_0$ in $C_{\mathrm{loc}}^2([0,\infty))$ as
$\omega\downarrow0$. Since also $\phi_\omega\to\phi_0$ locally
uniformly, for every $R>0$ we have
\[
  \int_{|x|\le R}\eta_\omega(x)\phi_\omega(x)\,dx
  \to
  \int_{|x|\le R}\eta_0(x)\phi_0(x)\,dx\in\R.
\]
To apply Fatou's lemma on the complement, we show that the integrand
has a fixed sign there for all sufficiently small $\omega>0$. Choose $R\ge R_0$ sufficiently large that
$\eta_0(R)<0$, where $R_0$ is as in Lemma~\ref{lem:wbound}.
By local convergence, $\eta_\omega(R)<0$ for all sufficiently small
$\omega>0$. On $\{x\in\R^d:|x|>R\}$, we have
\[
  L_\omega\eta_\omega=-\phi_\omega<0,
  \qquad \omega+V_\omega>0.
\]
Moreover, \eqref{eq:etaom} and the positive-frequency asymptotics
give $\eta_\omega(r)\to0$ as $r\to\infty$.
Together with $\eta_\omega(R)<0$, these conditions allow us to
apply the maximum principle on the exterior domain. We obtain
$\eta_\omega\le0$ there for all sufficiently small $\omega>0$.

Fatou's lemma, applied along an arbitrary sequence
$\omega_n\downarrow0$ to the nonnegative functions
$-\eta_{\omega_n}\phi_{\omega_n}$ on $\{|x|>R\}$, now gives
\[
  +\infty
  =\int_{|x|>R}(-\eta_0\phi_0)\,dx
  \le\liminf_{\omega\downarrow0}
       \int_{|x|>R}(-\eta_\omega\phi_\omega)\,dx.
\]
Finally, by \eqref{Mprime} and \eqref{eq:eta-def},
\[
  M'(\omega)
  =\int_{|x|\le R}\eta_\omega\phi_\omega\,dx
   +\int_{|x|>R}\eta_\omega\phi_\omega\,dx.
\]
As $\omega\downarrow0$, the first term has a finite limit, whereas
the second tends to $-\infty$. This completes the proof.
\end{proof}


\section{Construction of an unstable direction}
\label{sec:dir}

Throughout this section, we assume that $\phi_0\in L^2(\R^d)$ and use
the regular radial solution $\eta_0$ defined by \eqref{eq:eta0} in
Section~\ref{sec:4}. By Proposition~\ref{prop:eta0}, it satisfies
$L_0\eta_0=-\phi_0$ and
\[
  \eta_0(r)\sim-r^2\phi_0(r)
  \quad\text{as }r\to\infty
\]
unless $(d,p)=(4,2)$, in which case $\eta_0(r)\sim-1$.
The following result corresponds to the two alternatives in
Theorem~\ref{thm:ode}.

\begin{theorem}\label{thm:dir}
Assume that $\phi_0\in L^2(\R^d)$.
In each of the following cases, there
exist $R_0>0$ and a family
$\{\psi_R\}_{R\ge R_0}\subset H_{\rad}^1(\R^d)$ such that
\begin{equation}\label{eq:orth}
  \tbra[\psi_R,\phi_0]=0
\end{equation}
and the following assertions hold.
\begin{enumerate}
\item If $\phi_0\notin\cF(H^1)$, then
\begin{equation}\label{eq:lim-nw}
  \lim_{R\to\infty}\tbra[L_0\psi_R,\psi_R]
  =\tbra[L_0\phi_0,\phi_0]<0.
\end{equation}

\item If $\phi_0\in\cF(H^1)$ and
$\tbra[\eta_0,\phi_0]<0$, then
\begin{equation}\label{eq:lim-w}
  \lim_{R\to\infty}\tbra[L_0\psi_R,\psi_R]
  =\tbra[L_0\phi_0,\phi_0]
   +\frac{\norm[\phi_0]_{L^2}^4}{\tbra[\eta_0,\phi_0]}
  <0.
\end{equation}
\end{enumerate}
\end{theorem}

We now construct the family $\{\psi_R\}$ in Theorem~\ref{thm:dir}.
Choose a radial function $\varphi\in C_c^\infty(\R^d)$ such that
$0\le\varphi\le1$ and
\[
  \varphi(x)
  =\begin{cases}
    1, & |x|\le1,\\
    0, & |x|\ge2.
  \end{cases}
\]
For $R>0$, set $\varphi_R(x)\ce \varphi(x/R)$. Whenever
$\tbra[\varphi_R\eta_0,\phi_0]\ne0$, define
\begin{equation}\label{eq:psiR}
  \beta_R
  \ce -\frac{\norm[\phi_0]_{L^2}^2}
          {\tbra[\varphi_R\eta_0,\phi_0]},
  \qquad
  \psi_R\ce \phi_0+\beta_R\varphi_R\eta_0.
\end{equation}
Then $\psi_R\in H_{\rad}^1(\R^d)$ and
\eqref{eq:orth} holds for every $R$ for which the denominator
in \eqref{eq:psiR} is nonzero.

We begin with identities that will be used in both cases. Expanding the
quadratic form gives
\begin{equation}\label{eq:qform}
  \tbra[L_0\psi_R,\psi_R]
  =\tbra[L_0\phi_0,\phi_0]
   +2\beta_R\tbra[L_0(\varphi_R\eta_0),\phi_0]
   +\beta_R^2\tbra[L_0(\varphi_R\eta_0),\varphi_R\eta_0].
\end{equation}
Since $L_0\eta_0=-\phi_0$, Leibniz's rule yields
\begin{equation*}
  L_0(\varphi_R\eta_0)
  =-(\Delta\varphi_R)\eta_0
   -2\nabla\varphi_R\cdot\nabla\eta_0
   -\varphi_R\phi_0.
\end{equation*}
Consequently,
{\small
\begin{align}
  \tbra[L_0(\varphi_R\eta_0),\phi_0]
  &={-}\tbra[(\Delta\varphi_R)\eta_0,\phi_0]
    -2\tbra[\nabla\varphi_R\cdot\nabla\eta_0,\phi_0]
    -\tbra[\varphi_R\phi_0,\phi_0],
  \label{eq:mixed}
  \\
  \tbra[L_0(\varphi_R\eta_0),\varphi_R\eta_0]
  &={-}\tbra[(\Delta\varphi_R)\eta_0,\varphi_R\eta_0]
    -2\tbra[\nabla\varphi_R\cdot\nabla\eta_0,\varphi_R\eta_0]
 -\tbra[\varphi_R\eta_0,\varphi_R\phi_0].
  \label{eq:quad}
\end{align}
}%

We next record the cutoff estimates needed in the proof.

\begin{lemma}\label{lem:cut-nc}
Assume that $p\ne p^*$, and let $\alpha$ be the decay exponent defined in
\eqref{eq:alpha}. Then
\begin{align*}
  &\abs[{\tbra[(\Delta\varphi_R)\eta_0,\phi_0]}]
   +\abs[{\tbra[\nabla\varphi_R\cdot\nabla\eta_0,\phi_0]}]
   =O(R^{-2\alpha+d}),
  \\
  &\abs[{\tbra[(\Delta\varphi_R)\eta_0,\varphi_R\eta_0]}]
   +\abs[{\tbra[\nabla\varphi_R\cdot\nabla\eta_0,
                   \varphi_R\eta_0]}]
   =O(R^{-2\alpha+d+2})
\end{align*}
as $R\to\infty$.
\end{lemma}

\begin{proof}
The supports of $\nabla\varphi_R$ and $\Delta\varphi_R$ are contained in
$\{R\le|x|\le2R\}$, and
\[
  \abs[\nabla\varphi_R]\lesssim R^{-1},
  \qquad
  \abs[\Delta\varphi_R]\lesssim R^{-2}.
\]
Moreover, for every $\gamma\in\R$,
\[
  \int_{R\le|x|\le2R}|x|^\gamma\,dx\lesssim R^{\gamma+d}.
\]
By \eqref{eq:phidec}, Proposition~\ref{prop:eta0}, and
Lemma~\ref{lem:deta0}, we have, as \(r\to\infty\),
\[
  \phi_0(r)=O(r^{-\alpha}),
  \quad
  \eta_0(r)=O(r^{-\alpha+2}),
  \quad
  \eta_0'(r)=O(r^{-\alpha+1}).
\]
Hence, as \(|x|\to\infty\),
\[
  \eta_0\phi_0=O(|x|^{-2\alpha+2}),
  \qquad
  |\nabla\eta_0|\phi_0=O(|x|^{-2\alpha+1}),
\]
and therefore
\begin{align*}
  \abs[{\tbra[(\Delta\varphi_R)\eta_0,\phi_0]}]
  &\lesssim R^{-2}
    \int_{R\le|x|\le2R}|x|^{-2\alpha+2}\,dx
   \lesssim R^{-2\alpha+d},
  \\
  \abs[{\tbra[\nabla\varphi_R\cdot\nabla\eta_0,\phi_0]}]
  &\lesssim R^{-1}
    \int_{R\le|x|\le2R}|x|^{-2\alpha+1}\,dx
   \lesssim R^{-2\alpha+d}.
\end{align*}
This proves the first estimate. Similarly,
\[
  \eta_0^2=O(|x|^{-2\alpha+4}),
  \qquad
  |\nabla\eta_0|\,|\eta_0|=O(|x|^{-2\alpha+3}),
\]
which gives
\begin{align*}
  \abs[{\tbra[(\Delta\varphi_R)\eta_0,\varphi_R\eta_0]}]
  &\lesssim R^{-2}
    \int_{R\le|x|\le2R}|x|^{-2\alpha+4}\,dx
   \lesssim R^{-2\alpha+d+2},
  \\
  \abs[{\tbra[\nabla\varphi_R\cdot\nabla\eta_0,
                  \varphi_R\eta_0]}]
  &\lesssim R^{-1}
    \int_{R\le|x|\le2R}|x|^{-2\alpha+3}\,dx
   \lesssim R^{-2\alpha+d+2}.
\end{align*}
This proves the second estimate.
\end{proof}

\begin{lemma}\label{lem:cut-c}
Assume that $p=p^*$. Then $d\ge4$, and
\begin{align*}
  &\abs[{\tbra[(\Delta\varphi_R)\eta_0,\phi_0]}]
   +\abs[{\tbra[\nabla\varphi_R\cdot\nabla\eta_0,\phi_0]}]
   =
   \begin{cases}
     O((\log R)^{-1}), & d=4,\\[2pt]
     O\bigl(R^{-d+4}(\log R)^{-(d-2)}\bigr), & d\ge5,
   \end{cases}
   \\
  &\abs[{\tbra[(\Delta\varphi_R)\eta_0,\varphi_R\eta_0]}]
   +\abs[{\tbra[\nabla\varphi_R\cdot\nabla\eta_0,
                   \varphi_R\eta_0]}]
  =
   \begin{cases}
     O(R^2), & d=4,\\[2pt]
     O\bigl(R^{-d+6}(\log R)^{-(d-2)}\bigr), & d\ge5
   \end{cases}
\end{align*}
as $R\to\infty$.
\end{lemma}

\begin{proof}
On the annulus $\{R\le|x|\le2R\}$, we have, for every
$\gamma\in\R$ and $\delta\ge0$,
\[
  \int_{R\le|x|\le2R}
  |x|^\gamma(\log|x|)^{-\delta}\,dx
  \lesssim R^{\gamma+d}(\log R)^{-\delta}.
\]
Suppose first that $d\ge5$. By \eqref{eq:phidec},
Proposition~\ref{prop:eta0}, and Lemma~\ref{lem:deta0},
\begin{align*}
  \eta_0\phi_0
  &=O\bigl(|x|^{-2(d-2)+2}(\log|x|)^{-(d-2)}\bigr),
  \\
  |\nabla\eta_0|\phi_0
  &=O\bigl(|x|^{-2(d-2)+1}(\log|x|)^{-(d-2)}\bigr)
\end{align*}
as $|x|\to\infty$. Together with
$|\Delta\varphi_R|\lesssim R^{-2}$ and
$|\nabla\varphi_R|\lesssim R^{-1}$, these bounds give
\begin{align*}
  \abs[{\tbra[(\Delta\varphi_R)\eta_0,\phi_0]}]
  &\lesssim R^{-d+4}(\log R)^{-(d-2)},
  \\
  \abs[{\tbra[\nabla\varphi_R\cdot\nabla\eta_0,\phi_0]}]
  &\lesssim R^{-d+4}(\log R)^{-(d-2)}.
\end{align*}
This proves the first estimate for $d\ge5$. Likewise,
\begin{align*}
  \eta_0^2
  &=O\bigl(|x|^{-2(d-2)+4}(\log|x|)^{-(d-2)}\bigr),
  \\
  |\nabla\eta_0|\,|\eta_0|
  &=O\bigl(|x|^{-2(d-2)+3}(\log|x|)^{-(d-2)}\bigr)
\end{align*}
as $|x|\to\infty$, which gives the second estimate for $d\ge5$.

When $(d,p)=(4,2)$, Proposition~\ref{prop:eta0},
\eqref{eq:sharp}, and Lemma~\ref{lem:deta0} instead give
\[
  \eta_0(r)=O(1),\quad
  \phi_0(r)=O(r^{-2}(\log r)^{-1}),\quad
  \eta_0'(r)=O(r^{-1}(\log r)^{-1})
\]
as $r\to\infty$. Hence
\begin{align*}
  \abs[{\tbra[(\Delta\varphi_R)\eta_0,\phi_0]}]
  &\lesssim R^{-2}
       \int_{R\le|x|\le2R}|x|^{-2}(\log|x|)^{-1}\,dx
       \lesssim(\log R)^{-1},\\
  \abs[{\tbra[\nabla\varphi_R\cdot\nabla\eta_0,\phi_0]}]
  &\lesssim R^{-1}
       \int_{R\le|x|\le2R}|x|^{-3}(\log|x|)^{-2}\,dx
       \lesssim(\log R)^{-2}.
\end{align*}
Similarly,
\begin{align*}
  \abs[{\tbra[(\Delta\varphi_R)\eta_0,\varphi_R\eta_0]}]
  &\lesssim R^{-2}\int_{R\le|x|\le2R}1\,dx
       \lesssim R^2,\\
  \abs[{\tbra[\nabla\varphi_R\cdot\nabla\eta_0,
                  \varphi_R\eta_0]}]
  &\lesssim R^{-1}
       \int_{R\le|x|\le2R}|x|^{-1}(\log|x|)^{-1}\,dx
       \lesssim R^2(\log R)^{-1}.
\end{align*}
These are the remaining bounds.
\end{proof}

We are now in a position to prove Theorem~\ref{thm:dir}.

\begin{proof}[Proof of Theorem~\ref{thm:dir}]
The equation for $\phi_0$ gives
\[
  \tbra[L_0\phi_0,\phi_0]
  =-(q-1)\norm[\nabla\phi_0]_{L^2}^2
   -(q-p)\norm[\phi_0]_{L^{p+1}}^{p+1}<0.
\]
By Lemma~\ref{lem:eta-pair}, dominated convergence in case~\textup{(ii)}
and the eventual negativity of $\eta_0\phi_0$ in case~\textup{(i)} give,
as $R\to\infty$,
\[
  \tbra[\varphi_R\eta_0,\phi_0]
  \to
  \begin{cases}
    -\infty, & \text{in case~\textup{(i)}},\\[2pt]
    \tbra[\eta_0,\phi_0]<0, & \text{in case~\textup{(ii)}}.
  \end{cases}
\]
Thus $\psi_R$ is well defined for all sufficiently
large $R$ and satisfies \eqref{eq:orth}.

By \eqref{eq:phidec}, the assumption $\phi_0\in L^2(\R^d)$ implies
$2\alpha>d$ when $p\ne p^*$ and $d\ge4$ when $p=p^*$.
Hence Lemmas~\ref{lem:cut-nc} and~\ref{lem:cut-c}, together with
dominated convergence in \eqref{eq:mixed}, yield
\begin{equation}\label{eq:mixed-lim}
  \tbra[L_0(\varphi_R\eta_0),\phi_0]
  \to-\norm[\phi_0]_{L^2}^2
  \quad\text{as }R\to\infty.
\end{equation}

\medskip
\noindent\textbf{Case~\textup{(i)}:}~$\phi_0\notin\cF(H^1)$.
The preceding limit of the denominator gives $\beta_R\to0$, so
\eqref{eq:mixed-lim} implies
\[
  \beta_R\tbra[L_0(\varphi_R\eta_0),\phi_0]\to0.
\]
It remains to show that
$\beta_R^2\tbra[L_0(\varphi_R\eta_0),\varphi_R\eta_0]\to0$ as $R\to\infty$.

If $p\ne p^*$, then $2\alpha\le d+2$ by
\eqref{eq:phidec}. For $R>2$, set
\[
  h(R)
  \ce \begin{cases}
    (\log R)^{-1}, & 2\alpha=d+2,\\
    R^{2\alpha-d-2}, & 2\alpha<d+2.
  \end{cases}
\]
Then $h(R)\to0$. By \eqref{eq:phidec} and
Proposition~\ref{prop:eta0},
$\eta_0(x)\phi_0(x)\sim-|x|^{-2\alpha+2}$ as $|x|\to\infty$.
Integrating this asymptotic relation and using the properties of
$\varphi_R$, we obtain
\[
  -\tbra[\varphi_R\eta_0,\phi_0]\gtrsim h(R)^{-1}
\]
for all sufficiently large $R$, and hence
\begin{equation}\label{eq:beta-nc}
  |\beta_R|\lesssim h(R).
\end{equation}
The same asymptotics give
\begin{equation}\label{eq:pair-nc}
  \abs[{\tbra[\varphi_R\eta_0,\varphi_R\phi_0]}]
  \lesssim h(R)^{-1}.
\end{equation}
By \eqref{eq:quad}, Lemma~\ref{lem:cut-nc},
\eqref{eq:beta-nc}, and \eqref{eq:pair-nc},
\begin{align*}
  \abs[{\beta_R^2
    \tbra[L_0(\varphi_R\eta_0),\varphi_R\eta_0]}]
  &\lesssim h(R)^2
    \bigl(R^{-2\alpha+d+2}+h(R)^{-1}\bigr)
  \\
  &\lesssim h(R)\to0.
\end{align*}

If $p=p^*$, \eqref{l2} and \eqref{wl2} leave only $d=4$ and $d=5$.

\smallskip
\noindent\textit{If $(d,p)=(4,2)$,} integrating the asymptotics in
the proof of Lemma~\ref{lem:eta-pair} gives
\[
  -\tbra[\varphi_R\eta_0,\phi_0]
     \gtrsim\frac{R^2}{\log R},
  \qquad
  \abs[{\tbra[\varphi_R\eta_0,\varphi_R\phi_0]}]
     \lesssim\frac{R^2}{\log R}
\]
for all sufficiently large $R$. Here we use
$\int_{R_0}^R r(\log r)^{-1}\,dr\sim R^2/(2\log R)$ for fixed
$R_0>2$; the integral over $\{|x|\le R_0\}$ is bounded
independently of $R$. Thus $|\beta_R|\lesssim R^{-2}\log R$, and
\eqref{eq:quad} and Lemma~\ref{lem:cut-c} yield
\[
  \abs[{\beta_R^2
    \tbra[L_0(\varphi_R\eta_0),\varphi_R\eta_0]}]
  \lesssim \frac{(\log R)^2}{R^4}
       \left(R^2+\frac{R^2}{\log R}\right)\to0.
\]

\smallskip
\noindent\textit{If $p=p^*$ and $d=5$,} then \eqref{eq:phidec}
and Proposition~\ref{prop:eta0} give
\[
  \eta_0(x)\phi_0(x)
  \sim-|x|^{-4}(\log|x|)^{-3}
  \quad\text{as }|x|\to\infty.
\]
Since $\int_{R_0}^R(\log r)^{-3}\,dr\sim R(\log R)^{-3}$ for fixed
$R_0>2$, the same argument gives
\begin{equation}\label{eq:pair-c}
  -\tbra[\varphi_R\eta_0,\phi_0]
    \gtrsim R(\log R)^{-3},
  \qquad
  \abs[{\tbra[\varphi_R\eta_0,\varphi_R\phi_0]}]
    \lesssim R(\log R)^{-3}.
\end{equation}
Consequently,
\begin{equation}\label{eq:beta-c}
  |\beta_R|\lesssim R^{-1}(\log R)^3.
\end{equation}
Combining \eqref{eq:quad}, Lemma~\ref{lem:cut-c},
\eqref{eq:pair-c}, and \eqref{eq:beta-c}, we obtain
\[
  \abs[{\beta_R^2
    \tbra[L_0(\varphi_R\eta_0),\varphi_R\eta_0]}]
  \lesssim R^{-2}(\log R)^6\,R(\log R)^{-3}
  =R^{-1}(\log R)^3\to0.
\]
Thus, in every case, both correction terms in \eqref{eq:qform}
tend to zero, proving \eqref{eq:lim-nw}.

\medskip
\noindent\textbf{Case~\textup{(ii)}:}~$\phi_0\in\cF(H^1)$ and
$\tbra[\eta_0,\phi_0]<0$.
By \eqref{eq:phidec} and \eqref{wl2}, we have $2\alpha>d+2$ when
$p\ne p^*$ and $d\ge6$ when $p=p^*$. Thus the cutoff error terms
in \eqref{eq:quad} tend to zero by Lemmas~\ref{lem:cut-nc}
and~\ref{lem:cut-c}. Lemma~\ref{lem:eta-pair} and dominated
convergence therefore give, as $R\to\infty$,
\begin{align*}
  \tbra[L_0(\varphi_R\eta_0),\varphi_R\eta_0]
  &\to-\tbra[\eta_0,\phi_0],
 \qquad
  \beta_R
  \to-\frac{\norm[\phi_0]_{L^2}^2}
             {\tbra[\eta_0,\phi_0]}.
\end{align*}
Together with \eqref{eq:mixed-lim} and \eqref{eq:qform},
these limits prove \eqref{eq:lim-w}.
\end{proof}



\section{Proof of instability}\label{sec:inst}

We complete the proof of Theorem~\ref{thm:ode} by applying the following instability result for stationary solutions.

\begin{proposition}[{\cite[Proposition~5.3]{FH25}}]\label{prop:inst}
Assume that $\phi_0\in L^2(\R^d)$. If there exists a real-valued
$\psi\in H^1_{\rad}(\R^d)$ such that
\[
  \rbra[\psi,\phi_0]_{L^2}=0,
  \qquad
  \tbra[L_0\psi,\psi]<0,
\]
then $\phi_0$ is radially unstable, and hence orbitally unstable in
$H^1(\R^d)$.
\end{proposition}

Although the cited proposition is stated for $q<1+4/d$, its proof
applies throughout the energy-subcritical range $q<1+4/(d-2)_+$, provided
$\phi_0\in L^2(\R^d)$. 

\begin{proof}[Proof of Proposition~\ref{prop:inst} (outline)]
For the reader's convenience, we provide an outline of the proof here.
Since $S_0=E$ and $\phi_0\in H^1(\R^d)$, Proposition~\ref{prop:nehari}
gives
\begin{equation}\label{eq:nehari0}
  E(\phi_0)=\inf\{E(v):v\in H^1_{\rad}(\R^d)\setminus\{0\},\ K_0(v)=0\}.
\end{equation}
Since $\tbra[K_0'(\phi_0),\phi_0]=\tbra[L_0\phi_0,\phi_0]<0$,
taking second variations along the Nehari manifold gives, for
$v\in H^1_{\rad}(\R^d)$,
\begin{equation}\label{eq:ntan}
  \tbra[K_0'(\phi_0),v]=0
  \quad\Longrightarrow\quad
  \tbra[E''(\phi_0)v,v]\ge0.
\end{equation}
Since $\psi$ is real-valued, we have
$E''(\phi_0)\psi=L_0\psi$. Hence $\tbra[L_0\psi,\psi]<0$, together with the variational characterization above, implies $\tbra[K_0'(\phi_0),\psi]\ne0$.

For small $\varepsilon>0$, set
\[
  \mathcal V_\varepsilon
  \ce\{v\in H^1_{\rad}(\R^d):
       \inf_{\theta\in\R}\norm[v-e^{i\theta}\phi_0]_{H^1}<\varepsilon\}.
\]
As in \cite[Lemma~5.6]{FH25}, choose a $C^1$ phase map
$\theta\colon\mathcal V_\varepsilon\to\R/(2\pi\Z)$ with
$\theta(\phi_0)=0$ and
$\rbra[e^{i\theta(v)}v,i\phi_0]_{L^2}=0$. Define
\[
  A(v)\ce\rbra[e^{i\theta(v)}v,i\psi]_{L^2},
  \quad a(v)\ce-iA'(v),
  \quad \mathcal P(v)\ce\tbra[E'(v),a(v)].
\]
The Lyapunov functional $A$ is bounded and phase invariant on
$\mathcal V_\varepsilon$, and $A'(v)\in H^1_{\rad}(\R^d)$.
Moreover, $a(\phi_0)=\psi$ and, along solutions of \eqref{eq},
\begin{equation}\label{eq:Adot}
  \frac{d}{dt}A(u(t))=-\mathcal P(u(t))
\end{equation}
as long as $u(t)\in\mathcal V_\varepsilon$.

For $v\in\mathcal V_\varepsilon$, consider the function
$\lambda\mapsto E(v+\lambda a(v))$.
Since $a(\phi_0)=\psi$, its second derivative at $\lambda=0$ and
$v=\phi_0$ is
\begin{align*}
\tbra[E''(\phi_0)\psi,\psi]=\tbra[L_0\psi,\psi]<0.
\end{align*}
By continuity, this negativity persists for $|\lambda|$ small and
$v$ sufficiently close to $\phi_0$. By phase invariance, the same
holds throughout $\mathcal V_\varepsilon$ for sufficiently small
$\varepsilon>0$. Therefore, Taylor's formula at $\lambda=0$ yields
\begin{align}
\label{taylor-E}
  E(v+\lambda a(v))
  \le E(v)+\lambda\mathcal P(v)
  \qquad
  (v\in\mathcal V_\varepsilon,\ |\lambda|<\lambda_0)
\end{align}
for some $\lambda_0,\varepsilon>0$.
Since $\tbra[K_0'(\phi_0),\psi]\ne0$, the implicit function theorem,
together with phase invariance, gives, for sufficiently small
$\lambda_0,\varepsilon>0$, a function
$\Lambda\colon\mathcal V_\varepsilon\to(-\lambda_0,\lambda_0)$ such that
\[
  v+\Lambda(v)a(v)\ne0,
  \qquad
  K_0(v+\Lambda(v)a(v))=0.
\]
Hence, the variational characterization \eqref{eq:nehari0} and \eqref{taylor-E} imply
\begin{equation}\label{eq:egap}
  E(\phi_0)
  \le E(v+\Lambda(v)a(v))
  \le E(v)+\lambda_0|\mathcal P(v)|.
\end{equation}

Finally, $u_{0,\lambda}=\phi_0+\lambda\psi$ satisfies
$E(u_{0,\lambda})<E(\phi_0)$ for small $\lambda\ne0$.
Suppose that the corresponding solution remains in
$\mathcal V_\varepsilon$ throughout its forward lifespan.
Its $H^1$ norm is then bounded, so the blow-up alternative gives
existence for all $t\ge0$. Energy conservation and \eqref{eq:egap} yield
\[
  |\mathcal P(u(t))|
  \ge\frac{E(\phi_0)-E(u_{0,\lambda})}{\lambda_0}>0
  \qquad(t\ge0).
\]
By continuity, $\mathcal P(u(t))$ has a fixed sign. Hence
\eqref{eq:Adot} implies $|A(u(t))|\to\infty$ as $t\to\infty$,
contradicting the boundedness of $A$. Since
$u_{0,\lambda}\to\phi_0$ in $H^1(\R^d)$ as $\lambda\to0$, this proves
radial instability. Orbital instability follows from
\cite[Lemma~5.2]{FH25}.
\end{proof}

\begin{remark}
\label{rem:6.2}
At zero frequency, the essential spectrum of the linearized operator
touches zero, so the coercivity in $H^1$ required in the general theory
of \cite{GSS87} is not available. The variational characterization
on the Nehari manifold replaces this coercivity argument:
\eqref{eq:egap} bounds $|\mathcal P(u(t))|$ away from zero in terms of
the conserved energy gap. Similar variational arguments were used in \cite{O14,NOW17} to study the instability of a two-parameter family of traveling waves for the one-dimensional derivative nonlinear Schr\"odinger equation.
\end{remark}

We can now complete the proof of Theorem~\ref{thm:ode}.

\begin{proof}[Proof of Theorem~\ref{thm:ode}]
The fundamental systems in Propositions~\ref{prop:vom0}
and~\ref{prop:wasym} yield the representation \eqref{eq:eta0}.
Proposition~\ref{prop:eta0} gives the existence, uniqueness, and stated
asymptotic behavior of $\eta_0$, while Theorem~\ref{thm:etalim} shows that $\eta_\omega$ converges to $\eta_0$ locally in $C^2$ as $\omega\downarrow0$.
By Lemma~\ref{lem:eta-pair}, $\tbra[\eta_0,\phi_0]=-\infty$ when $\phi_0\notin\cF(H^1)$, whereas $\eta_0\in\cF(H^{-1})$ and the pairing
$\tbra[\eta_0,\phi_0]$ is absolutely convergent when $\phi_0\in\cF(H^1)$.

In case~\textup{(i)}, and in case~\textup{(ii)} when
$\tbra[\eta_0,\phi_0]<0$, Theorem~\ref{thm:dir} supplies, for all
sufficiently large $R$, a real-valued $\psi_R\in H^1_{\rad}(\R^d)$
satisfying
\[
  \rbra[\psi_R,\phi_0]_{L^2}=0,
  \qquad
  \tbra[L_0\psi_R,\psi_R]<0.
\]
Proposition~\ref{prop:inst}, applied with $\psi=\psi_R$, proves the
claimed instability and completes the proof.
\end{proof}


\section{Limit of the resolvent}\label{sec:res}

The proof of Theorem~\ref{thm:res}\,\textup{(ii)} is based on a resolvent
representation of $M'(\omega)$.
For $\omega>0$, Theorem~\ref{thm:shoot} ensures that
$0\in\rho((L_\omega)_{\rad})$, which gives
\begin{equation}\label{eq:Mres}
  M'(\omega)
  =\rbra[\pt_\omega\phi_\omega,\phi_\omega]_{L^2}
  =\rbra[(L_\omega)_{\rad}^{-1}(-\phi_\omega),\phi_\omega]_{L^2},
  \qquad \omega>0.
\end{equation}
To study the limit as $\omega\downarrow0$, we factorize the radial
linearized operator and apply the Hardy--Littlewood--Sobolev inequality.
Compactness, together with the fact that zero is neither an eigenvalue
nor a resonance, then allows us to pass to the limit in the factorized
inverse.
Under \eqref{eq:high-d}, this identifies the limit of $M'(\omega)$
with $\tbra[\eta_0,\phi_0]$.

\subsection{Preliminaries}

For $s\in(0,d)$, let
\[
  I_s(f)
  \ce c_{d,s}|x|^{-d+s}*f
  =c_{d,s}\int_{\R^d}\frac{f(x-y)}{|y|^{d-s}}\,dy
  =(-\Delta)^{-s/2}f
\]
be the Riesz potential. We also write
\[
  2^*\ce\frac{2d}{d-2}.
\]
We recall the following standard estimates for Riesz potentials.
\begin{lemma}[Hardy--Littlewood--Sobolev inequality]
Let $(s,a,b)\in(0,d)\times(1,\infty)^2$ satisfy
$1/a-1/b=s/d$. Then
\begin{equation*}
  \norm[(-\Delta)^{-s/2}f]_{L^b}
  \lesssim \norm[f]_{L^a}.
\end{equation*}
In particular, if $d\ge3$, then
\begin{align*}
  \norm[(-\Delta)^{-1/2}f]_{L^{2^*}}
  &\lesssim \norm[f]_{L^2},
  \\
  \norm[(-\Delta)^{-1/2}f]_{L^2}
  &\lesssim \norm[f]_{L^{2d/(d+2)}}.
\end{align*}
\end{lemma}

We write the linearized operator
\[
  (L_\omega)_{\rad}=-\Delta+\omega+V_\omega
\]
on radial functions and define
\begin{equation*}
  B_\omega
  \ce(-\Delta+\omega)^{-1/2}
    V_\omega
    (-\Delta+\omega)^{-1/2}
  \colon L^2_{\rad}(\R^d)\to L^2_{\rad}(\R^d).
\end{equation*}
Then
\[
  (L_\omega)_{\rad}
  =(-\Delta+\omega)^{1/2}
    (I+B_\omega)
    (-\Delta+\omega)^{1/2}.
\]
If $V\in L^{d/2}(\R^d)$, then H\"older's inequality gives
\begin{equation*}
  \norm[Vf]_{L^{2d/(d+2)}}
  \le \norm[V]_{L^{d/2}}\norm[f]_{L^{2^*}}.
\end{equation*}
Thus multiplication by $V$ defines a bounded operator from
$L^{2^*}(\R^d)$ to $L^{2d/(d+2)}(\R^d)$, with operator norm bounded by
$\norm[V]_{L^{d/2}}$.

\begin{lemma}\label{lem:Kcomp}
Assume that $d\ge3$ and $V_0\in L^{d/2}(\R^d)$. Then, for every
$\omega\ge0$, the operator
$B_\omega\colon L^2_{\rad}(\R^d)\to L^2_{\rad}(\R^d)$ is compact and
self-adjoint.
\end{lemma}

\begin{proof}
We present the proof for $\omega=0$; the case $\omega>0$ is analogous,
since the exponential decay of $\phi_\omega$ gives
$V_\omega\in L^{d/2}(\R^d)$.
The boundedness of $B_0$ follows because it is the composition of
\begin{align*}
  (-\Delta)^{-1/2}
  &\in\mathcal{L}\bigl(
    L^2_{\rad}(\R^d),L^{2^*}_{\rad}(\R^d)
  \bigr),
  \\
  V_0
  &\in\mathcal{L}\bigl(
    L^{2^*}_{\rad}(\R^d),L^{2d/(d+2)}_{\rad}(\R^d)
  \bigr),
  \\
  (-\Delta)^{-1/2}
  &\in\mathcal{L}\bigl(
    L^{2d/(d+2)}_{\rad}(\R^d),L^2_{\rad}(\R^d)
  \bigr).
\end{align*}

It remains to prove compactness. It suffices to show that
\[
  V_0(-\Delta)^{-1/2}
  \colon L^2_{\rad}(\R^d)\to L^{2d/(d+2)}_{\rad}(\R^d)
\]
is compact. Let $f_n\rightharpoonup f$ weakly in $L^2_{\rad}(\R^d)$. Then
$(-\Delta)^{-1/2}f_n\rightharpoonup(-\Delta)^{-1/2}f$ weakly in
$\dot H^1_{\rad}(\R^d)$. By Rellich's theorem,
\[
  (-\Delta)^{-1/2}f_n
  \to(-\Delta)^{-1/2}f
  \quad\text{in }L^r_{\mathrm{loc}}(\R^d)
\]
for every $r\in[1,2^*)$.

Let $\varepsilon>0$. Choose $Z\in C_c^\infty(\R^d)$ such that
$\norm[V_0-Z]_{L^{d/2}}<\varepsilon$. Then
\begin{align*}
  &\norm[
    V_0(-\Delta)^{-1/2}(f_n-f)
  ]_{L^{2d/(d+2)}}
  \\
  &\quad\le
  \norm[
    (V_0-Z)(-\Delta)^{-1/2}(f_n-f)
  ]_{L^{2d/(d+2)}}
  \\
  &\quad\quad+
  \norm[
    Z(-\Delta)^{-1/2}(f_n-f)
  ]_{L^{2d/(d+2)}}
  \\
  &\quad\le
  \norm[V_0-Z]_{L^{d/2}}
  \norm[(-\Delta)^{-1/2}(f_n-f)]_{L^{2^*}}
  \\
  &\quad\quad+
  \norm[Z]_{L^\infty}
  \norm[
    (-\Delta)^{-1/2}(f_n-f)
  ]_{L^{2d/(d+2)}(\operatorname{supp}Z)}
  \\
  &\quad\lesssim
  \norm[V_0-Z]_{L^{d/2}}
  +\norm[
    (-\Delta)^{-1/2}(f_n-f)
  ]_{L^{2d/(d+2)}(\operatorname{supp}Z)}
  \lesssim\varepsilon
\end{align*}
for all sufficiently large $n$. Hence $B_0$ is compact. Its
self-adjointness follows from the self-adjointness of multiplication by
$V_0$ and of $(-\Delta)^{-1/2}$.
\end{proof}

\subsection{\texorpdfstring{Convergence as $\omega\downarrow0$}{Convergence as omega tends to zero}}

We now use the compactness established above to study the
operator norm convergence of $B_\omega$ and $(I+B_\omega)^{-1}$
as $\omega\downarrow0$.

We rewrite $B_\omega$ as
\[
  B_\omega
  =J_\omega(-\Delta)^{-1/2}
    V_\omega
    (-\Delta)^{-1/2}J_\omega,
\]
where
\[
  J_\omega
  \ce\left(\frac{-\Delta}{-\Delta+\omega}\right)^{1/2}.
\]
The Fourier multiplier representation
\[
  J_\omega
  =\mathcal{F}^{-1}
    \left(\frac{|\xi|^2}{\omega+|\xi|^2}\right)^{1/2}
    \mathcal{F}
\]
shows that $\norm[J_\omega]_{\mathcal{L}(L^2_{\rad})}\le1$ and that
$J_\omega\to I$ strongly on $L^2_{\rad}(\R^d)$ as $\omega\downarrow0$.
Moreover, $J_\omega$ is self-adjoint.


\begin{lemma}\label{lem:Kconv}
Assume that $d\ge3$, $V_0\in L^{d/2}(\R^d)$, and
\[
  V_\omega\to V_0
  \quad\text{in }L^{d/2}(\R^d)
  \quad\text{as }\omega\downarrow0.
\]
Then
\[
  B_\omega\to B_0
  \quad\text{in }\mathcal{L}(L^2_{\rad}(\R^d))
  \quad\text{as }\omega\downarrow0.
\]
\end{lemma}

\begin{proof}
By the triangle inequality, we obtain
\begin{align*}
\norm[B_\omega-B_0]_{\mathcal{L}(L^2_{\rad})}
  &\le
  \norm[B_\omega-J_\omega B_0J_\omega]_{\mathcal{L}(L^2_{\rad})}
  \\
  &\quad+
  \norm[J_\omega B_0J_\omega-J_\omega B_0]_{\mathcal{L}(L^2_{\rad})}
  +\norm[J_\omega B_0-B_0]_{\mathcal{L}(L^2_{\rad})}
  \\
  &=\norm[
    J_\omega(-\Delta)^{-1/2}
    (V_\omega-V_0)
    (-\Delta)^{-1/2}J_\omega
  ]_{\mathcal{L}(L^2_{\rad})}
  \\
  &\quad+
  \norm[J_\omega B_0(J_\omega-I)]_{\mathcal{L}(L^2_{\rad})}
  +\norm[(J_\omega-I)B_0]_{\mathcal{L}(L^2_{\rad})}
  \\
  &\lesssim
  \norm[V_\omega-V_0]_{L^{d/2}}
  +2\norm[(J_\omega-I)B_0]_{\mathcal{L}(L^2_{\rad})}.
\end{align*}
Here we used
\[
  (B_0(J_\omega-I))^*=(J_\omega-I)B_0.
\]
Since $B_0$ is compact by Lemma~\ref{lem:Kcomp} and
$J_\omega\to I$ strongly on $L^2_{\rad}(\R^d)$, we have
\[
  \norm[(J_\omega-I)B_0]_{\mathcal{L}(L^2_{\rad})}\to0
  \quad\text{as }\omega\downarrow0.
\]
Together with
$V_\omega\to V_0$ in $L^{d/2}(\R^d)$, this proves the result.
\end{proof}

\begin{definition}[Zero resonance]\label{def:res}
Let $d\ge3$. We say that $0$ is a resonance of $(L_0)_{\rad}$ if there
exists a nonzero function
$\psi\in\dot H^1_{\rad}(\R^d)\setminus L^2(\R^d)$ such that
\[
  (L_0)_{\rad}\psi=0
  \quad\text{in }\dot H^{-1}_{\rad}(\R^d).
\]
\end{definition}

The following proposition is a consequence of
Theorem~\ref{thm:shoot} and its proof.

\begin{proposition}[Absence of a zero eigenvalue and a zero resonance]
\label{prop:noz}
Assume that $d\ge3$. Then $(L_0)_{\rad}$ has neither a zero eigenvalue nor a
zero resonance.
\end{proposition}

\begin{proof}
The absence of a zero eigenvalue is precisely the radial nondegeneracy
conclusion of Theorem~\ref{thm:shoot}. To exclude a zero
resonance, let $\psi\in\dot H^1_{\rad}(\R^d)$ satisfy
$(L_0)_{\rad}\psi=0$. Elliptic regularity shows that $\psi$ is regular at
the origin, and hence it is a multiple of the regular solution $v_{y^*}$.
By Proposition~\ref{prop:vstar}, this solution has a nonzero
non-decaying component at infinity. If $p<p^*$, then
$v_{y^*}(r)\sim c_+r^{\lambda_+}$ with $\lambda_+>0$; if $p>p^*$, then
$\lambda_+=0$; and if $p=p^*$, then
$v_{y^*}(r)\sim c_+(\log r)^{d/2}$. Thus
$v_{y^*}\notin L^{2^*}(\R^d)$ in every case. Since $d\ge3$, the Sobolev
embedding $\dot H^1(\R^d)\hookrightarrow L^{2^*}(\R^d)$ rules out a zero resonance.
\end{proof}

\begin{lemma}\label{lem:Kinv}
Assume that $d\ge3$, $V_0\in L^{d/2}(\R^d)$, and
\[
  V_\omega\to V_0
  \quad\text{in }L^{d/2}(\R^d)
  \quad\text{as }\omega\downarrow0.
\]
Then, for all sufficiently small $\omega\ge0$, the operator $I+B_\omega$
is invertible on $L^2_{\rad}(\R^d)$, and
\[
  (I+B_\omega)^{-1}
  \to(I+B_0)^{-1}
  \quad\text{in }\mathcal{L}(L^2_{\rad}(\R^d))
  \quad\text{as }\omega\downarrow0.
\]
\end{lemma}

\begin{proof}
We first prove that $I+B_0$ is invertible, or equivalently that
$-1\in\rho(B_0)$. Suppose instead that $-1\in\sigma(B_0)$. Since $B_0$ is
compact and $-1\ne0$, we have $-1\in\sigma_p(B_0)$. Hence there exists
$h\in L^2_{\rad}(\R^d)\setminus\{0\}$ such that $(I+B_0)h=0$. Set
\[
  f\ce(-\Delta)^{-1/2}h\in\dot H^1_{\rad}(\R^d),
  \qquad
  h=(-\Delta)^{1/2}f.
\]
Then
\[
  (L_0)_{\rad}f
  =(-\Delta)^{1/2}(I+B_0)(-\Delta)^{1/2}f
  =(-\Delta)^{1/2}(I+B_0)h
  =0.
\]
If $f\in L^2(\R^d)$, then $f$ is a zero eigenfunction;
otherwise, $0$ is a resonance.
Both possibilities contradict Proposition~\ref{prop:noz}.

We next prove invertibility for all sufficiently small $\omega>0$. Set
\[
  C_0\ce\norm[(I+B_0)^{-1}]_{\mathcal{L}(L^2_{\rad})}>0.
\]
By Lemma~\ref{lem:Kconv}, there exists $\delta>0$ such that
\[
  \norm[B_\omega-B_0]_{\mathcal{L}(L^2_{\rad})}
  <\frac{1}{2C_0},
  \qquad 0<\omega<\delta.
\]
Consequently,
\[
  \norm[
    (I+B_0)^{-1}(B_\omega-B_0)
  ]_{\mathcal{L}(L^2_{\rad})}
  \le\frac12.
\]
The Neumann series therefore shows that
$I+(I+B_0)^{-1}(B_\omega-B_0)$ is invertible and
\[
  \norm[
    \bigl(I+(I+B_0)^{-1}(B_\omega-B_0)\bigr)^{-1}
  ]_{\mathcal{L}(L^2_{\rad})}
  \le2.
\]
Since
\[
  I+B_\omega
  =(I+B_0)
    \bigl(I+(I+B_0)^{-1}(B_\omega-B_0)\bigr),
\]
it follows that $I+B_\omega$ is invertible and
\[
  \norm[(I+B_\omega)^{-1}]_{\mathcal{L}(L^2_{\rad})}
  \le2C_0.
\]
Finally, the second resolvent identity gives
\begin{align*}
  &\norm[
    (I+B_\omega)^{-1}-(I+B_0)^{-1}
  ]_{\mathcal{L}(L^2_{\rad})}
  \\
  &\quad=
  \norm[
    (I+B_\omega)^{-1}
    (B_0-B_\omega)
    (I+B_0)^{-1}
  ]_{\mathcal{L}(L^2_{\rad})}
  \\
  &\quad\le
  2C_0^2\norm[B_0-B_\omega]_{\mathcal{L}(L^2_{\rad})}
  \xrightarrow[\omega\downarrow0]{}0.
\end{align*}
This proves the assertion.
\end{proof}
%

By Lemma~\ref{lem:Kinv}, $I+B_0$ is invertible on
$L^2_{\rad}(\R^d)$. Together with the factorization
\[
  (L_0)_{\rad}
  =(-\Delta)^{1/2}(I+B_0)(-\Delta)^{1/2}
  \colon\dot H^1_{\rad}(\R^d)\to\dot H^{-1}_{\rad}(\R^d),
\]
this shows that $(L_0)_{\rad}$ is invertible, with inverse
\[
  (L_0)_{\rad}^{-1}
  =(-\Delta)^{-1/2}(I+B_0)^{-1}(-\Delta)^{-1/2}
  \colon\dot H^{-1}_{\rad}(\R^d)\to\dot H^1_{\rad}(\R^d).
\]

We now use these results to identify the zero-frequency limit
of $M'(\omega)$.
\begin{proof}[Proof of Theorem~\ref{thm:res}]
Part~\textup{(i)} follows from Theorem~\ref{thm:massdiv}
and Lemma~\ref{lem:eta-pair}.
We now prove part~\textup{(ii)}.
The local convergence of $\phi_\omega$, together with
\eqref{eq:phiu}, \eqref{eq:phicrit}, and \eqref{eq:high-d}, gives,
by dominated convergence,
\[
  \phi_\omega\to\phi_0\quad\text{in }L^{2d/(d+2)}(\R^d),
  \quad
  V_\omega\to V_0\quad\text{in }L^{d/2}(\R^d)
  \quad\text{as }\omega\downarrow0.
\]
First, we claim that
\begin{equation}\label{eq:hconv}
  (-\Delta+\omega)^{-1/2}\phi_\omega
  \to
  (-\Delta)^{-1/2}\phi_0
  \quad\text{in }L^2_{\rad}(\R^d).
\end{equation}
Indeed,
\begin{align*}
  &\norm[
    (-\Delta+\omega)^{-1/2}\phi_\omega
    -(-\Delta)^{-1/2}\phi_0
  ]_{L^2}
  \\
  &\quad=
  \norm[
    J_\omega(-\Delta)^{-1/2}\phi_\omega
    -(-\Delta)^{-1/2}\phi_0
  ]_{L^2}
  \\
  &\quad\le
  \norm[
    (-\Delta)^{-1/2}(\phi_\omega-\phi_0)
  ]_{L^2}
  +\norm[
    (J_\omega-I)(-\Delta)^{-1/2}\phi_0
  ]_{L^2}
  \\
  &\quad\lesssim
  \norm[\phi_\omega-\phi_0]_{L^{2d/(d+2)}}
  +\norm[
    (J_\omega-I)(-\Delta)^{-1/2}\phi_0
  ]_{L^2}
  \xrightarrow[\omega\downarrow0]{}0.
\end{align*}

We next justify the identity
\begin{equation}\label{eq:L0inv}
  (L_0)_{\rad}^{-1}(-\phi_0)=\eta_0.
\end{equation}
We first verify $\eta_0\in\dot H^1_{\rad}(\R^d)$. Since
$\eta_0$ is smooth near the origin and tends to zero at infinity,
it suffices to prove that $\nabla\eta_0$ is square integrable at infinity.
If $p\ne p^*$, Lemma~\ref{lem:deta0} gives
\[
  |\eta_0'(r)|\lesssim r^{-\alpha+1}
  \quad\text{as }r\to\infty.
\]
Moreover, $\braket{x}\phi_0\in L^2(\R^d)$ and
$\phi_0(r)\sim r^{-\alpha}$ imply $2\alpha>d+2$. Hence
\[
  \int_2^\infty |\eta_0'(r)|^2r^{d-1}\,dr
  \lesssim
  \int_2^\infty r^{d+1-2\alpha}\,dr
  <\infty.
\]
If $p=p^*$, Lemma~\ref{lem:deta0} yields
\[
  |\eta_0'(r)|
  \lesssim
  r^{-d+3}(\log r)^{-(d-2)/2}
  \quad\text{as }r\to\infty,
\]
and therefore
\[
  \int_2^\infty |\eta_0'(r)|^2r^{d-1}\,dr
  \lesssim
  \int_2^\infty
  r^{-d+5}(\log r)^{-(d-2)}\,dr
  <\infty,
\]
where the last integral is finite for $d\ge6$, including the borderline
case $d=6$ owing to the logarithmic factor. Thus
$\eta_0\in\dot H^1_{\rad}(\R^d)$.

The asymptotics of $\phi_0$, together with \eqref{eq:high-d}, give
\[
  \phi_0\in L^{2d/(d+2)}_{\rad}(\R^d)
  \subset\dot H^{-1}_{\rad}(\R^d).
\]
Since $V_0\in L^{d/2}(\R^d)$, the Sobolev embedding and H\"older's
inequality also yield
\[
  V_0\eta_0\in L^{2d/(d+2)}_{\rad}(\R^d)
  \subset\dot H^{-1}_{\rad}(\R^d).
\]
By the isomorphism
\[
  -\Delta\colon
  \dot H^1_{\rad}(\R^d)
  \to
  \dot H^{-1}_{\rad}(\R^d),
\]
there exists a unique $\zeta\in\dot H^1_{\rad}(\R^d)$ satisfying
\[
  -\Delta\zeta=-\phi_0-V_0\eta_0
  \quad\text{in }\dot H^{-1}_{\rad}(\R^d).
\]
On the other hand, the distributional identity
$L_0\eta_0=-\phi_0$ shows that $\eta_0$ satisfies the same equation.
The uniqueness of $\zeta$ therefore gives $\eta_0=\zeta$. Consequently,
$L_0\eta_0=-\phi_0$ holds in $\dot H^{-1}_{\rad}(\R^d)$, and the
invertibility of $(L_0)_{\rad}$ proves \eqref{eq:L0inv}.

By Lemma~\ref{lem:Kinv} and
\eqref{eq:hconv},
\[
  (I+B_\omega)^{-1}(-\Delta+\omega)^{-1/2}\phi_\omega
  \xrightarrow[\omega\downarrow0]{}
  (I+B_0)^{-1}(-\Delta)^{-1/2}\phi_0
  \quad\text{in }L^2_{\rad}(\R^d).
\]
Using \eqref{eq:Mres} and the factorization of $(L_\omega)_{\rad}$,
we conclude that
\begin{align*}
  M'(\omega)
  &=-\tbra[
    (I+B_\omega)^{-1}(-\Delta+\omega)^{-1/2}\phi_\omega,
    (-\Delta+\omega)^{-1/2}\phi_\omega
  ]
  \\
  &\xrightarrow[\omega\downarrow0]{}
  -\tbra[
    (I+B_0)^{-1}(-\Delta)^{-1/2}\phi_0,
    (-\Delta)^{-1/2}\phi_0
  ]
  \\
  &=\tbra[(L_0)_{\rad}^{-1}(-\phi_0),\phi_0]
  =\tbra[\eta_0,\phi_0].
\end{align*}
This proves \eqref{eq:limM}.
\end{proof}


\section{\texorpdfstring{An $L^2$-scaling criterion}{An L2-scaling criterion}}
\label{sec:scale}

In this section, we prove Theorem~\ref{thm:scale}. 
We first establish strict positivity of the linearized quadratic form
on the orthogonal complement of the negative eigenspace of
$(L_0)_{\rad}$ in a weighted energy space. 
We then apply this positivity to a suitable
linear combination of $\eta_0$ and the scaling derivative of $\phi_0$.

\subsection{Positivity in a weighted energy space}

All function spaces in this section are over $\R$. Define
\[
  \scX
  \ce\bigl\{
    v\in\dot H^1_{\rad}(\R^d;\R):
    \phi_0^{(p-1)/2}v\in L^2(\R^d)
  \bigr\},
\]
with inner product
\[
  (v,w)_\scX
  \ce\int_{\R^d}
    \bigl(\nabla v\cdot\nabla w+\phi_0^{p-1}vw\bigr)\,dx.
\]
Since $\phi_0$ is bounded and $q>p$, we have
$|V_0|\lesssim\phi_0^{p-1}$. Hence, for $v,w\in \scX$, the quadratic-form
pairing
\[
  \tbra[(L_0)_{\rad}v,w]
  \ce\int_{\R^d}
    \bigl(\nabla v\cdot\nabla w+V_0vw\bigr)\,dx
\]
is well defined and continuous on $\scX\times \scX$. Note that the space $H^1_{\rad}(\R^d)$ is dense in $\scX$. 

The nonnegativity on the codimension-one Nehari tangent space in
\eqref{eq:ntan}, together with $\tbra[L_0\phi_0,\phi_0]<0$,
shows that $(L_0)_{\rad}$ has a unique negative eigenvalue,
which is simple.
Let $e_0<0$ denote this eigenvalue, and let
$\chi\in H^2_{\rad}(\R^d;\R)$ be a corresponding eigenfunction.
Since $\chi$ decays exponentially at infinity, the pairing $\tbra[v,\chi]$ is well
defined for every $v\in \scX$.

\begin{lemma}
\label{lem:pos}
If $v\in \scX\setminus\{0\}$ satisfies $\tbra[v,\chi]=0$, then
\[
  \tbra[(L_0)_{\rad}v,v]>0.
\]
\end{lemma}

\begin{proof}
We first prove the nonnegativity. By the spectral properties of
$(L_0)_{\rad}$,
\[
  \tbra[(L_0)_{\rad}w,w]\ge0
\]
for every $w\in H^1_{\rad}(\R^d)$ satisfying $\tbra[w,\chi]=0$.
By the density of $H^1_{\rad}(\R^d)$ in $\scX$, there exists a sequence
$\{v_n\}\subset H^1_{\rad}(\R^d)$ such that $v_n\to v$ in $\scX$. Set
\[
  w_n
  \ce v_n-\frac{\tbra[v_n,\chi]}{\norm[\chi]_{L^2}^2}\chi
  \in H^1_{\rad}(\R^d).
\]
Then $\tbra[w_n,\chi]=0$ and
$\tbra[(L_0)_{\rad}w_n,w_n]\ge0$. Moreover,
$\tbra[v_n,\chi]\to\tbra[v,\chi]=0$, and hence $w_n\to v$ in $\scX$.
By the continuity of the quadratic form on $\scX$, we obtain
\[
  \tbra[(L_0)_{\rad}v,v]
  =\lim_{n\to\infty}\tbra[(L_0)_{\rad}w_n,w_n]
  \ge0.
\]

It remains to prove strict positivity. Suppose, to the contrary, that
\[
  \tbra[(L_0)_{\rad}v,v]=0.
\]
Set
\[
  \scX_\chi
  \ce\bigl\{w\in \scX:\tbra[w,\chi]=0\bigr\}.
\]
Since the quadratic form of $(L_0)_{\rad}$ is nonnegative on $\scX_\chi$, the
Cauchy--Schwarz inequality for nonnegative quadratic forms gives
\[
  \bigl|\tbra[(L_0)_{\rad}v,w]\bigr|^2
  \le
  \tbra[(L_0)_{\rad}v,v]
  \tbra[(L_0)_{\rad}w,w]
  =0
\]
for every $w\in \scX_\chi$. In addition,
\[
  \tbra[(L_0)_{\rad}v,\chi]
  =\tbra[v,(L_0)_{\rad}\chi]
  =e_0\tbra[v,\chi]
  =0.
\]
Since $\scX=\scX_\chi\oplus\operatorname{span}\{\chi\}$, it follows that
$(L_0)_{\rad}v=0$ in the weak sense.

By Proposition~\ref{prop:vom0}, the radial profile of $v$ is a linear
combination of the fundamental solutions $v_0^+$ and $v_0^-$. The singular
behavior of $v_0^-$ at the origin rules out its contribution because
$v\in \scX$. The regular solution $v_0^+$ coincides, up to normalization, with
the shooting derivative $v_{y^*}$, and
Proposition~\ref{prop:vstar} shows that $v_0^+\notin \scX$.
Consequently, $v=0$, which contradicts the assumption $v\ne0$.
\end{proof}

\subsection{Proof of the scaling criterion}

We now use the positivity established above to prove Theorem~\ref{thm:scale}.

\begin{proof}[Proof of Theorem~\ref{thm:scale}]
Set
\[
  \psi
  \ce\left.\frac{d}{d\lambda}\phi_0^\lambda\right|_{\lambda=1}
  =\left(x\cdot\nabla+\frac d2\right)\phi_0.
\]
The decay estimates in Section~\ref{sec:ellip} imply that $\psi\in \scX\setminus\{0\}$ and
$\psi(r)=O(\phi_0(r))$ as $r\to\infty$. Since $E'(\phi_0)=0$, we have
\[
  \left.\frac{d^2}{d\lambda^2}E(\phi_0^\lambda)\right|_{\lambda=1}
  =\tbra[(L_0)_{\rad}\psi,\psi].
\]
We also have $\eta_0\in \scX$. Indeed, Lemma~\ref{lem:deta0} and
$\braket{x}\phi_0\in L^2(\R^d)$ give
$\eta_0\in\dot H^1_{\rad}(\R^d)$, while
$\eta_0(r)\sim-r^2\phi_0(r)$ as $r\to\infty$. Moreover,
$\braket{x}^2\phi_0^{p-1}$ is bounded, and therefore the assumption
$\braket{x}\phi_0\in L^2(\R^d)$ yields
\[
  \phi_0^{p-1}\eta_0^2
  \lesssim
  \braket{x}^2\phi_0^2
  \quad\text{for large }|x|.
\]
Thus $\phi_0^{(p-1)/2}\eta_0\in L^2(\R^d)$.

The hypothesis and Lemma~\ref{lem:pos} imply that
$\tbra[\psi,\chi]\ne0$. Define
\[
  \beta
  \ce-\frac{\tbra[\eta_0,\chi]}{\tbra[\psi,\chi]}.
\]
Then
\[
  \tbra[\eta_0+\beta\psi,\chi]=0.
\]
Since $\psi(r)=O(\phi_0(r))$ whereas
$\eta_0(r)\sim-r^2\phi_0(r)$, the functions $\psi$ and $\eta_0$ are
linearly independent. Hence $\eta_0+\beta\psi\ne0$, and
Lemma~\ref{lem:pos} gives
\begin{align*}
  0
  &<\tbra[
    (L_0)_{\rad}(\eta_0+\beta\psi),
    \eta_0+\beta\psi
  ]
  \\
  &=\tbra[(L_0)_{\rad}\eta_0,\eta_0]
    +2\beta\tbra[(L_0)_{\rad}\eta_0,\psi]
    +\beta^2\tbra[(L_0)_{\rad}\psi,\psi]
  \\
  &=-\tbra[\phi_0,\eta_0]
    -2\beta\tbra[\phi_0,\psi]
    +\beta^2
      \left.\frac{d^2}{d\lambda^2}E(\phi_0^\lambda)\right|_{\lambda=1}.
\end{align*}
The $L^2$-invariance of the scaling gives
\[
  \tbra[\phi_0,\psi]
  =\left.
    \frac12\frac{d}{d\lambda}
    \norm[\phi_0^\lambda]_{L^2}^2
  \right|_{\lambda=1}
  =0.
\]
Consequently,
\[
  \tbra[\phi_0,\eta_0]
  <\beta^2
    \left.\frac{d^2}{d\lambda^2}E(\phi_0^\lambda)\right|_{\lambda=1}
  \le0.
\]
This completes the proof.
\end{proof}
\begin{remark}
\label{rem:scaling}
The argument above also applies to the positive-frequency
standing waves of \eqref{eq}. In terms of the action $S_\omega$, it yields
\[
  \left.\frac{d^2}{d\lambda^2}
  S_\omega(\phi_\omega^\lambda)\right|_{\lambda=1}
  \le0
  \quad\Longrightarrow\quad
  M'(\omega)=\tbra[\eta_\omega,\phi_\omega]<0,
  \qquad \omega>0.
\]
A corresponding result for nonlinear Schr\"odinger equations with
an attractive inverse-power potential was established in
\cite[Proposition~5.5]{F21}.

By contrast, for the mass-critical pure-power NLS, one has
\[
  \left.\frac{d^2}{d\lambda^2}
  S_\omega(\phi_\omega^\lambda)\right|_{\lambda=1}
  =M'(\omega)=0,
  \qquad \omega>0.
\]
Thus the nonpositive scaling condition does not imply
$M'(\omega)<0$ in general. A common feature of \eqref{eq} and the inverse-power-potential model is the breaking of scaling symmetry. This contrast highlights the role of broken scaling symmetry in
obtaining the strict negativity of $M'(\omega)$ even when the
second derivative along the scaling direction vanishes.
\end{remark}


\appendix
\section{The shooting method}\label{app:shoot}

We prove uniqueness and radial nondegeneracy of the positive radial solution
of \eqref{eq:spomega} under the assumptions
\[
  d\ge2,\quad \omega\ge0,\quad
  1<p<q<\frac{d+2}{(d-2)_+}.
\]
The proof uses McLeod's shooting approach \cite{M93} and follows the main argument
of \cite{F13}; see also \cite{T06, KOPV17, LN20}. To the best of the authors'
knowledge, the literature does not contain a unified proof covering both
uniqueness and nondegeneracy in the present range, including $\omega=0$.

For trajectories $u_y$ with $y\in \mathcal Y_0$, where $\mathcal Y_0$ is defined below,
we use the asymptotic results of
Section~\ref{sec:ellip}. Their proofs require only positivity, monotonicity,
and decay at infinity and therefore apply without change. The same applies
to Lemma~\ref{lem:dphi0}; we also use the derivative estimate in
Remark~\ref{rem:wprime}. When $\omega=0$ and $p\ne p^*$, $\alpha$ and
$\lambda_\pm$ are as in \eqref{eq:alpha} and \eqref{eq:lambdapm}.

For $s\in\R$, set
\[
  g(s)\ce -\omega s-|s|^{p-1}s+|s|^{q-1}s,
  \qquad
  G(s)\ce \int_0^s g(t)\,dt.
\]
On $(0,\infty)$, $g(s)=-\omega s-s^p+s^q$. By Lemma~\ref{lem:F}, it has
a unique positive zero $\beta$, with
\[
  g<0\quad\text{on }(0,\beta),
  \quad
  g>0\quad\text{on }(\beta,\infty),
  \quad
  g'(\beta)>0.
\]
For $y>0$, let $u_y$ solve
\begin{equation*}
\begin{dcases}
  u_y''+\dfrac{d-1}{r}u_y'+g(u_y)=0, \quad r>0,\\
  u_y(0)=y,\quad u_y'(0)=0.
\end{dcases}
\end{equation*}
Define
\begin{align*}
  \mathcal Y_+&\ce \bigl\{y>0:\inf_{r>0}u_y(r)>0\bigr\},\\
  \mathcal Y_0&\ce \bigl\{y>0:\inf_{r>0}u_y(r)=0\bigr\},\\
  \mathcal Y_-&\ce \bigl\{y>0:\inf_{r>0}u_y(r)<0\bigr\}.
\end{align*}
For $y\in \mathcal Y_0\cup \mathcal Y_-$, let $R_y$ be the first zero of $u_y$, with
$R_y=\infty$ for $y\in \mathcal Y_0$. Since the ground-state existence theory
gives $\mathcal Y_0\ne\varnothing$, we may set
\begin{equation}\label{eq:ystar}
  y^*\ce \inf(\mathcal Y_0\cup \mathcal Y_-).
\end{equation}

\begin{theorem}[Uniqueness and radial nondegeneracy]
\label{thm:shoot}
There is exactly one shooting parameter producing a positive solution that
converges to zero at infinity. More precisely,
\[
  \mathcal Y_+=(0,y^*),
  \quad
  \mathcal Y_0=\{y^*\},
  \quad
  \mathcal Y_-=(y^*,\infty).
\]
Consequently, the boundary-value problem
\[
\begin{dcases}
  u''+\dfrac{d-1}{r}u'+g(u)=0\quad\text{for}~ r>0,\\
  u'(0)=0,\quad u(r)\to0 \quad\text{as }r\to\infty,\\
  u(r)>0\quad \text{for}~r\ge0 
\end{dcases}
\]
has exactly one solution. In particular, the positive radial ground state
$\phi_\omega$ is unique and
\[
  \ker(L_\omega)_{\rad}=\{0\}.
\]
\end{theorem}

\begin{remark}
The shooting analysis in \cite{M93} for $\Delta u+f(u)=0$
is carried out under the assumption $f'(0)<0$.
In our setting, this corresponds to $g'(0)<0$, or equivalently
$\omega>0$.
That condition yields exponential decay and simplifies the boundary
and oscillation arguments.
At zero frequency, additional algebraic or logarithmic
decay estimates must instead be supplied explicitly, and this is
precisely where extra care is needed in our argument.
\end{remark}

The remainder of this appendix is devoted to the proof of
Theorem~\ref{thm:shoot}.

\subsection{Basic shooting properties}

We first collect the basic properties of the shooting trajectories and of the
sets $\mathcal Y_+$, $\mathcal Y_0$, and $\mathcal Y_-$. These properties locate the first transition
parameter $y^*$ and show that it produces a positive solution decaying at
infinity.

We introduce the local energy associated with the equation for $u_y$:
\begin{equation*}
  H_y(r)\ce \frac12 u_y'(r)^2+G(u_y(r)),\qquad r\ge0.
\end{equation*}
Differentiating and using the equation, we obtain
\begin{equation*}
  H_y'(r)=-\frac{d-1}{r}u_y'(r)^2\le0,
  \qquad r>0.
\end{equation*}
This monotonicity is the key ingredient in the proof of the basic
properties of $\mathcal Y_+$, $\mathcal Y_0$, and $\mathcal Y_-$
stated in Lemma~\ref{lem:basic} below. Since $H_y(0)=G(y)$, it also
gives
\begin{equation*}
  G(u_y(r))\le H_y(r)\le G(y),
  \qquad r\ge0.
\end{equation*}

\begin{lemma}\label{lem:basic}
The following statements hold.
\begin{enumerate}
\item The sets $\mathcal Y_+$, $\mathcal Y_0$, and $\mathcal Y_-$ form a disjoint partition of
  $(0,\infty)$, and
  \begin{align*}
    \mathcal Y_0
      &=\{y>0:u_y(r)>0\text{ for all }r>0
          \text{ and }u_y(r)\to0\text{ as }r\to\infty\},\\
    \mathcal Y_-
      &=\{y>0:R_y<\infty\}.
  \end{align*}
\item $(0,\beta]\subset \mathcal Y_+$.
\item If $y\in \mathcal Y_0\cup \mathcal Y_-$, then $y>\beta$ and
  \[
    u_y'(r)<0\qquad (0<r\le R_y).
  \]
\item The set $\mathcal Y_-$ is open, and $y\mapsto R_y$ is continuous on
  $\mathcal Y_-$.
\item If $y_n\in \mathcal Y_-$ and $y_n\to y\in \mathcal Y_0$, then
  $R_{y_n}\to\infty$.
\item For every $y\in \mathcal Y_+$, there are $\delta,m>0$ such that
  \[
    \inf_{r\ge0}u_z(r)\ge m
    \quad\text{whenever }z>0\text{ and }|z-y|<\delta.
  \]
  In particular, $\mathcal Y_+$ is open.
\end{enumerate}
\end{lemma}

\begin{proof}
See \cite[Lemma~7.1]{LN20}.
\end{proof}

\begin{proposition}\label{prop:first}
One has
\[
  y^*>\beta,
  \quad
  (0,y^*)\subset \mathcal Y_+,
  \quad
  y^*\in \mathcal Y_0.
\]
Thus $u_{y^*}$ is positive and tends to zero at infinity.
\end{proposition}

\begin{proof}
Lemma~\ref{lem:basic}\,\textup{(ii)}, together with the openness of
$\mathcal Y_+$ in part~\textup{(vi)}, gives $(0,\beta]\subset \mathcal Y_+$ and $y^*>\beta$,
and its partition statement gives $(0,y^*)\subset \mathcal Y_+$. Openness of
$\mathcal Y_+$ and $\mathcal Y_-$, together with \eqref{eq:ystar}, excludes
$y^*\in \mathcal Y_+\cup \mathcal Y_-$; hence $y^*\in \mathcal Y_0$.
\end{proof}

\subsection{\texorpdfstring{The $I$-function for the double-power nonlinearity}{The I-function for the double-power nonlinearity}}

We introduce the $I$-function
\begin{equation*}
  I_\lambda(s)\ce s g'(s)-\lambda g(s)
  =(\lambda-1)\omega s+(\lambda-p)s^p+(q-\lambda)s^q,
  \qquad \lambda\in\R.
\end{equation*}
This function arises naturally from the Wronskian identities in
Lemma~\ref{lem:W}. The one-sign-change property of $I_\lambda$, whose proof
relies on the structure of the nonlinearity, is a key ingredient in the
shooting argument.

For the proof of the next lemma, we also set
\[
  F(s)\ce \frac{s g'(s)}{g(s)}
  \qquad (s>0,\ g(s)\ne0).
\]

\begin{lemma}\label{lem:F}
There exists a unique $\beta>0$ such that $g(\beta)=0$, and
\[
  g(s)<0\quad(0<s<\beta),
  \quad
  g(s)>0\quad(s>\beta),
  \quad
  g'(\beta)>0.
\]
Moreover, $F(s)<q$ on $(0,\beta)$, while $F$ is strictly decreasing
on $(\beta,\infty)$ and maps this interval onto $(q,\infty)$.
\end{lemma}

\begin{proof}
Let
\[
  \Psi(s)\ce s^{-q}g(s)=1-s^{p-q}-\omega s^{1-q}.
\]
Since
\[
  s\Psi'(s)
  =(q-p)s^{p-q}+(q-1)\omega s^{1-q}>0,
\]
$\Psi$ increases strictly from $-\infty$ to $1$. It therefore has a
unique zero $\beta$, which gives the stated signs of $g=s^q\Psi$, and
$g'(\beta)=\beta^q\Psi'(\beta)>0$.

For $s\ne\beta$,
\[
  F(s)-q
  =\frac{s g'(s)-qg(s)}{g(s)}
  =\frac{s\Psi'(s)}{\Psi(s)}.
\]
This is negative on $(0,\beta)$. On $(\beta,\infty)$, both
$1/\Psi(s)$ and $s\Psi'(s)$ are positive and strictly decreasing, so
$F-q$ is strictly decreasing. Finally,
\[
  \lim_{s\downarrow\beta}F(s)=\infty,
  \qquad
  \lim_{s\to\infty}F(s)=q.
  \qedhere
\]
\end{proof}

\begin{lemma}\label{lem:Isign}
For every $\tau>\beta$, there is a unique $\lambda>q$ such that
$I_\lambda(\tau)=0$. For this $\lambda$,
\begin{equation}
\label{eq:Isign}
  I_\lambda(s)>0\quad(0<s<\tau),
  \quad
  I_\lambda(s)<0\quad(s>\tau).
\end{equation}
\end{lemma}

\begin{proof}
Lemma~\ref{lem:F} gives the unique choice $\lambda=F(\tau)>q$, and
$I_\lambda=g(F-\lambda)$ away from $\beta$. For $0<s<\beta$,
$g(s)<0$ and $F(s)<q<\lambda$, whereas
$I_\lambda(\beta)=\beta g'(\beta)>0$. On $(\beta,\infty)$, the strict
decrease of $F$ gives the asserted sign change at $s=\tau$.
\end{proof}
\begin{remark}
\label{multi-pw}
The same argument applies to
\[
  g(s)=-\omega s+\sum_{k=1}^N a_ks^{p_k},
  \qquad N\ge2,
\]
provided
\[
  \omega\ge0,
  \quad
  1<p_1<\cdots<p_N,
  \quad
  a_k<0~(1\le k<N),
  \quad
  a_N>0.
\]
Indeed, for $\Psi(s)=s^{-p_N}g(s)$,
\[
  s\Psi'(s)
  =\omega(p_N-1)s^{1-p_N}
   +\sum_{k=1}^{N-1}a_k(p_k-p_N)s^{p_k-p_N}>0,
\]
and both $1/\Psi$ and $s\Psi'$ are positive and strictly decreasing
after the unique positive zero of $g$. Thus Lemmas~\ref{lem:F} and~\ref{lem:Isign} remain valid with $q$ replaced by $p_N$. 
\end{remark}

For $y>0$, define
\[
  \mathcal L_y f
  \ce f''+\frac{d-1}{r}f'+g'(u_y)f.
\]
The shooting derivative $v_y\ce \pt_yu_y$ is the regular solution of
\begin{equation*}
\begin{dcases}
  \mathcal L_y v_y=0, \quad r>0,\\
  v_y(0)=1,\enspace v_y'(0)=0.
\end{dcases}
\end{equation*}

A direct calculation gives the following result.
\begin{lemma}[Wronskian identities]\label{lem:W}
For every $C^2$ function $f$ on $(0,\infty)$,
\begin{equation}\label{eq:W}
  \bigl(r^{d-1}(v_yf'-fv_y')\bigr)'
  =r^{d-1}v_y\mathcal L_y f.
\end{equation}
Moreover,
\begin{align}
  \mathcal L_y u_y&=u_yg'(u_y)-g(u_y)=I_1(u_y),
    \notag\\
  \mathcal L_y u_y'&=\frac{d-1}{r^2}u_y',
    \label{eq:Luprime}\\
  \mathcal L_y(ru_y')&=-2g(u_y),
    \notag\\
  \mathcal L_y(u_y+cru_y')&=I_{1+2c}(u_y)
  \qquad(c\in\R).
    \label{eq:Lf}
\end{align}
\end{lemma}

\subsection{Analysis at the first transition parameter}

We next analyze the shooting derivative at the first transition parameter
$y^*$. We first determine its zero structure and then show that its
non-decaying component at infinity has a strictly negative coefficient. 
These properties will be used to extend the shooting argument to $\mathcal Y_-$.

\begin{lemma}\label{lem:vzero}
For every $y\in \mathcal Y_0$, the function $v_y$ has at least one zero in
$(0,\infty)$.
\end{lemma}

\begin{proof}
Assume that $v_y>0$ on $(0,\infty)$, and write $u=u_y$ and $v=v_y$.
By \eqref{eq:W} with $f=u'$, Lemma~\ref{lem:basic}\,\textup{(iii)}, and \eqref{eq:Luprime},
\[
  \bigl(r^{d-1}(vu''-u'v')\bigr)'
  =(d-1)r^{d-3}vu'<0.
\]
Since $u'(0)=v'(0)=0$ and $u''$, $v'$ are bounded near the origin,
\[
  \lim_{r\downarrow0}r^{d-1}(vu''-u'v')(r)=0.
\]
Therefore
\begin{equation}\label{eq:ratio}
\begin{aligned}
  r^{d-1}v(r)^2\left(\frac{u'}{v}\right)'(r)
  &=r^{d-1}(vu''-u'v')(r)\\
  &=(d-1)\int_0^r s^{d-3}v(s)u'(s)\,ds<0.
\end{aligned}
\end{equation}
Thus $u'/v$ is strictly decreasing, and for $r\ge1$,
\[
  0<v(r)\le C(-u'(r)).
\]
Hence $v(r)\to0$ as $r\to\infty$, and $v$ is a multiple of the decaying
fundamental solution at infinity. Moreover, \eqref{eq:ratio} gives
\begin{equation}\label{eq:Wneg}
  r^{d-1}(vu''-u'v')(r)
  \le(d-1)\int_0^1s^{d-3}v(s)u'(s)\,ds<0,
  \qquad r\ge1.
\end{equation}

The left-hand side nevertheless tends to zero as $r\to\infty$. This follows
immediately from the exponential asymptotics when $\omega>0$. Now suppose that $\omega=0$. If $p\ne p^*$,
Section~\ref{sec:ellip}, Lemma~\ref{lem:dphi0},
Remark~\ref{rem:wprime}, and the equation for $u$ give, as
$r\to\infty$,
\begin{alignat*}{3}
  u&=O(r^{-\alpha}),\qquad
  &u'&=O(r^{-\alpha-1}),\qquad
  &u''&=O(r^{-\alpha-2}),\\
  v&=O(r^{\lambda_-}),\qquad
  &v'&=O(r^{\lambda_--1}).
\end{alignat*}
Consequently,
\begin{equation}\label{eq:Wdec-nc}
  r^{d-1}(vu''-u'v')
  =O(r^{d-3+\lambda_--\alpha})=O(r^{1-d})=o(1),
\end{equation}
because \eqref{eq:alpha} and \eqref{eq:lambdapm} imply
$\alpha\ge d-2$ and $\lambda_-\le-(d-2)$. If $p=p^*$, the corresponding
critical estimates give, as $r\to\infty$,
\begin{alignat*}{2}
  u'&=O\bigl(r^{-(d-1)}(\log r)^{-(d-2)/2}\bigr),\qquad
  &u''&=O\bigl(r^{-d}(\log r)^{-(d-2)/2}\bigr),\\
  v&=O\bigl(r^{-(d-2)}(\log r)^{-d/2}\bigr),\qquad
  &v'&=O\bigl(r^{-(d-1)}(\log r)^{-d/2}\bigr).
\end{alignat*}
Hence
\begin{equation}\label{eq:Wdec-c}
  r^{d-1}(vu''-u'v')
  =O\bigl(r^{1-d}(\log r)^{-(d-1)}\bigr)=o(1).
\end{equation}
This contradicts \eqref{eq:Wneg}.
\end{proof}
To prove that $v_{y^*}$ has exactly one zero,
we first study the difference $u_{y^*}-u_y$
for $y<y^*$.
\begin{lemma}\label{lem:hzero}
For every $y\in[\beta,y^*)$, the function
\[
  h_y\ce u_{y^*}-u_y
\]
has exactly one zero in $(0,\infty)$.
\end{lemma}

\begin{proof}
For $y\in[\beta,y^*)$, Proposition~\ref{prop:first} gives
$y\in\mathcal Y_+$. Thus $u_y$ is bounded away from zero, whereas
$u_{y^*}(r)\to0$ as $r\to\infty$. Hence $h_y(r)<0$ for sufficiently
large $r$. Since $h_y(0)=y^*-y>0$, $h_y$ has at least one zero.
Every zero is simple, since a double zero would imply
$u_y\equiv u_{y^*}$ by uniqueness for the ODE.

Set
\[
  \mathcal A_+\ce
  \{y\in[\beta,y^*):
    h_y\text{ has exactly one zero in }(0,\infty)\}.
\]
Since $u_\beta\equiv\beta$, the strict monotonicity of $u_{y^*}$
gives $\beta\in\mathcal A_+$.
For $Y\in\mathcal A_+$, Lemma~\ref{lem:basic}\,\textup{(vi)} gives
$\eps,\delta>0$, with $Y+\delta<y^*$, such that
\[
  u_y(r)\ge\eps
  \qquad (y\in[Y,Y+\delta],\ r\ge0).
\]
Choose $R$ with $u_{y^*}(r)<\eps/2$ for $r\ge R$. Then
$h_y<-\eps/2$ on this tail, while continuous dependence in
$C^1([0,R])$ and the simplicity of the zero of $h_Y$ ensure that
$h_y$ has exactly one zero on $[0,R]$, after possibly decreasing
$\delta$. Thus $[Y,Y+\delta]\subset\mathcal A_+$.

Set
\[
  \widebar{Y}\ce
  \sup\{Y\in[\beta,y^*):[\beta,Y]\subset\mathcal A_+\}.
\]
If $\widebar{Y}<y^*$, then $h_{\widebar{Y}}$ cannot have two zeros:
by continuous dependence and their simplicity, two such zeros
would persist for parameters just below $\widebar{Y}$, contradicting
$[\beta,\widebar{Y})\subset\mathcal A_+$. Hence
$h_{\widebar{Y}}$ has exactly one zero, so
$\widebar{Y}\in\mathcal A_+$. The continuation property above
then contradicts the definition of $\widebar{Y}$.
Therefore $\widebar{Y}=y^*$.
\end{proof}

\begin{lemma}\label{lem:vstar}
The function $v_{y^*}$ has exactly one zero in $(0,\infty)$.
\end{lemma}

\begin{proof}
Lemma~\ref{lem:vzero} gives at least one zero, and every zero is simple. If
there were two, we could choose $0<r_1<r_2$ such that
\[
  v_{y^*}(r_1)<0<v_{y^*}(r_2).
\]
For every compact $J\subset[0,\infty)$, differentiable dependence on the
initial value gives
\[
  \frac{u_{y^*}-u_y}{y^*-y}
  \to v_{y^*}
  \quad\text{in }C(J)
  \quad\text{as}~y\uparrow y^*.
\]
Hence $h_y(r_1)<0<h_y(r_2)$ for $y<y^*$ sufficiently close to $y^*$.
Together with $h_y(0)>0$, this gives at least two zeros of $h_y$, contrary
to Lemma~\ref{lem:hzero}.
\end{proof}

The next proposition is a key ingredient in the proofs of uniqueness and
radial nondegeneracy, as well as in the ODE construction of $\eta_0$. Its proof uses the
one-sign-change property of $I_\lambda$ from Lemma~\ref{lem:Isign}.

\begin{proposition}
\label{prop:vstar}
Let $r_*>0$ be the unique zero of $v_{y^*}$, and write
\begin{equation}\label{eq:vdec}
  v_{y^*}=c_+w^++c_-w^-
\end{equation}
in the fundamental system at infinity for $\mathcal L_{y^*}v=0$, where
$w^+$ is non-decaying and $w^-$ is decaying and positive for large $r$.
Then $c_+<0$, and
\begin{equation}\label{eq:vasymp}
  v_{y^*}(r)\sim
  \begin{cases}
    c_+r^{-(d-1)/2}e^{\sqrt\omega r},
      &\omega>0,\\[2pt]
    c_+r^{\lambda_+},
      &\omega=0,\ p\ne p^*,\\[2pt]
    c_+(\log r)^{d/2},
      &\omega=0,\ p=p^*
  \end{cases}
  \quad\text{as }r\to\infty.
\end{equation}
In particular, $v_{y^*}$ is eventually negative and does not belong to
$L^2((1,\infty),r^{d-1}\,dr)$.
\end{proposition}

\begin{proof}
Write $u=u_{y^*}$ and $v=v_{y^*}$. By uniqueness of the zero,
\begin{equation}\label{eq:vsign}
  v>0\quad\text{on }(0,r_*),
  \quad
  v<0\quad\text{on }(r_*,\infty),
  \quad
  v'(r_*)<0.
\end{equation}
Assume that $c_+=0$ in \eqref{eq:vdec}. Then $v$ is a negative multiple
of $w^-$ for large $r$, so $v(r)\to0$ as $r\to\infty$.

We first prove
\begin{equation}\label{eq:ubeta}
  u(r_*)>\beta.
\end{equation}
Using the equations for $u$ and $v$
\[
  \bigl(r^{d-1}u'\bigr)'=-r^{d-1}g(u),
  \qquad
  \bigl(r^{d-1}v'\bigr)'=-r^{d-1}g'(u)v,
\]
we obtain directly
\begin{equation}\label{eq:prod}
  \bigl(r^{2d-2}u'v'\bigr)'=-r^{2d-2}(g(u)v)'.
\end{equation}
Thus, for $R>r_*$,
\begin{align}
& R^{2d-2}u'(R)v'(R)
  -r_*^{2d-2}u'(r_*)v'(r_*)
  \notag\\
&\quad=-\bigl[r^{2d-2}g(u(r))v(r)\bigr]_{r_*}^{R}
  +(2d-2)\int_{r_*}^R r^{2d-3}g(u(r))v(r)\,dr.
  \label{eq:pint}
\end{align}
For $\omega>0$, all boundary terms at infinity vanish exponentially. Now
suppose that $\omega=0$. If $p\ne p^*$, then, as $r\to\infty$,
\[
  u=O(r^{-\alpha}),
  \quad
  v=O(r^{\lambda_-}),
  \quad
  g(u)=O(u^p).
\]
Therefore, we obtain
\[
  R^{2d-2}g(u(R))v(R)
  =O\bigl(R^{2d-2-p\alpha+\lambda_-}\bigr)=o(1)
  \qquad(R\to\infty).
\]
If $p=p^*$, then, as $r\to\infty$,
\[
  g(u)=O\bigl(r^{-d}(\log r)^{-d/2}\bigr),
  \qquad
  v=O\bigl(r^{-(d-2)}(\log r)^{-d/2}\bigr).
\]
Hence
\begin{equation*}
  R^{2d-2}g(u(R))v(R)=O((\log R)^{-d})=o(1)
  \qquad(R\to\infty).
\end{equation*}

For the other boundary term, the derivative estimates yield
\[
  R^{2d-2}u'(R)v'(R)
  =\begin{cases}
    O\bigl(R^{2d-4+\lambda_--\alpha}\bigr)=o(1),
      &p<p^*,\\[2pt]
    O\bigl((\log R)^{-(d-1)}\bigr)=o(1),
      &p=p^*
  \end{cases}
  \quad\text{as }R\to\infty.
\]
When $p>p^*$, decay alone gives only boundedness. On a sufficiently large
tail,
\[
  \bigl(r^{d-1}v'(r)\bigr)'
  =r^{d-1}\bigl(pu(r)^{p-1}-qu(r)^{q-1}\bigr)v(r)<0.
\]
If $v'(R_0)\le0$ at some large $R_0$, this monotonicity would imply
$v'(r)<0$ for $r>R_0$, contradicting $v(r)\to0$ as $r\to\infty$
with $v<0$. Hence
$v'(r)>0$ eventually, and since $u'<0$,
\[
  \limsup_{R\to\infty}R^{2d-2}u'(R)v'(R)\le0.
\]
Because $u'(r_*)v'(r_*)>0$, the upper limit of the left-hand side of
\eqref{eq:pint} is strictly negative. If $u(r_*)\le\beta$, then
$g(u)<0$ and $v<0$ on $(r_*,\infty)$, making the limiting integral on the
right strictly positive. This contradiction proves \eqref{eq:ubeta}.

Choose the unique $\lambda>q$ such that $I_\lambda(u(r_*))=0$, and set
\[
  f=u+\frac{\lambda-1}{2}ru'.
\]
By \eqref{eq:Lf}, $\mathcal L_{y^*}f=I_\lambda(u)$. Since $u$ decreases
strictly, Lemma~\ref{lem:Isign} and \eqref{eq:vsign} imply
\[
  v(r)I_\lambda(u(r))<0,
  \qquad r\ne r_*.
\]
Hence \eqref{eq:W} shows that $r^{d-1}(vf'-fv')$ is strictly decreasing.
Its limit as $r\downarrow0$ is zero because
$u'(0)=v'(0)=f'(0)=0$.

Under the assumption $c_+=0$, the decay estimates give
\[
  r^{d-1}vu'\to0,
  \quad
  r^{d-1}uv'\to0,
  \quad
  r^d(vu''-u'v')\to0
  \quad\text{as }r\to\infty.
\]
For the last limit, \eqref{eq:Wdec-nc} gives
$O(r^{2-d})=o(1)$ when $p\ne p^*$ and $d\ge3$; when $d=2$,
necessarily $p<p^*$, and the sharper exponent gives
$O(r^{\lambda_--\alpha})=o(1)$. In the critical case,
\eqref{eq:Wdec-c} gives
$O(r^{2-d}(\log r)^{-(d-1)})=o(1)$. Since
\[
  vf'-fv'
  =vu'-uv'+\frac{\lambda-1}{2}
    \bigl(vu'+r(vu''-u'v')\bigr),
\]
we obtain $r^{d-1}(vf'-fv')\to0$ as $r\to\infty$, contradicting its
strict decrease from zero. Thus $c_+\ne0$.

Finally, $v<0$ on $(r_*,\infty)$, while $w^+>0$ and dominates $w^-$ at
infinity; hence $c_+<0$. Formula~\eqref{eq:vasymp} follows from
Proposition~\ref{prop:wasym}, and its leading terms are not square
integrable with respect to $r^{d-1}\,dr$.
\end{proof}

\subsection{\texorpdfstring{Analysis of $\mathcal Y_-$ and completion of the proof}{Analysis of Y-minus and completion of the proof}}

We now complete the shooting argument.
The key is to show that $v_y$ has exactly one zero in $(0,R_y)$ for every $y\in\mathcal Y_-$.

\begin{lemma}\label{lem:Szero}
If $y\in \mathcal Y_-$, then $v_y$ has at least one zero in $(0,R_y)$.
\end{lemma}

\begin{proof}
Suppose that $v_y>0$ on $(0,R_y)$, and write $u=u_y$, $v=v_y$, and
$R=R_y$. Since $y>\beta$, Lemma~\ref{lem:Isign} gives a unique
$\lambda>q$ such that $I_\lambda(y)=0$. For
\[
  f=u+\frac{\lambda-1}{2}ru',
\]
we have $\mathcal L_yf=I_\lambda(u)>0$ on $(0,R)$ by \eqref{eq:Lf}, \eqref{eq:Isign}, and the monotonicity of $u$.
Moreover, $f(0)=y>0$ and
$f(R)=\frac{\lambda-1}{2}Ru'(R)<0$. If $\sigma$ is the first zero of
$f$, then $f'(\sigma)\le0$. On the other hand, \eqref{eq:W} and
$u'(0)=v'(0)=f'(0)=0$ give
\[
  \sigma^{d-1}(vf'-fv')(\sigma)
  =\int_0^\sigma r^{d-1}v(r)\mathcal L_yf(r)\,dr>0,
\]
whereas the left-hand side equals
$\sigma^{d-1}v(\sigma)f'(\sigma)\le0$, a contradiction.
\end{proof}
We next establish a sign condition at $R_y$ needed for the continuation argument.
\begin{lemma}\label{lem:end}
Let $y\in \mathcal Y_-$, and assume that $v_y$ has exactly one zero in $(0,R_y)$.
Then $v_y(R_y)<0$.
\end{lemma}

\begin{proof}
Write $u=u_y$, $v=v_y$, $R=R_y$, and let $r_y$ be the unique zero of
$v$ in $(0,R)$. Then $v>0$ on $(0,r_y)$ and $v<0$ on $(r_y,R)$. Suppose that
$v(R)=0$; then $v'(r_y)<0<v'(R)$.

Integrating \eqref{eq:prod} over $(r_y,R)$ gives
\begin{align*}
& R^{2d-2}u'(R)v'(R)
  -r_y^{2d-2}u'(r_y)v'(r_y)\\
&\quad=(2d-2)\int_{r_y}^R r^{2d-3}g(u(r))v(r)\,dr,
\end{align*}
since $g(u)v$ vanishes at both endpoints. If $u(r_y)\le\beta$, then
$g(u)<0$ and $v<0$ on $(r_y,R)$, so the right-hand side is positive,
whereas the left-hand side is negative. Thus $u(r_y)>\beta$.

Choose $\lambda>q$ with $I_\lambda(u(r_y))=0$ and set
$f=u+\frac{\lambda-1}{2}ru'$. As in Proposition~\ref{prop:vstar},
$vI_\lambda(u)<0$ away from $r_y$, so \eqref{eq:W} implies that
$r^{d-1}(vf'-fv')$ decreases strictly from zero. Its value at $R$ must be
negative, but $u(R)=v(R)=0$ gives
\[
  -R^{d-1}f(R)v'(R)
  =-\frac{\lambda-1}{2}R^d u'(R)v'(R)>0,
\]
a contradiction.
\end{proof}
We show that the one-zero property holds
for $y>y^*$ sufficiently close to $y^*$.
\begin{lemma}
\label{lem:right}
There exists $\varepsilon>0$ such that every
$y\in(y^*,y^*+\varepsilon)$ belongs to $\mathcal Y_-$ and $v_y$ has exactly one zero
in $(0,R_y)$.
\end{lemma}

\begin{proof}
Write $u_*=u_{y^*}$ and $v_*=v_{y^*}$, and choose $\delta\in(0,\beta)$ such
that
\begin{equation}\label{eq:gneg}
  g'(s)<0,\qquad 0<s\le \delta.
\end{equation}
By Proposition~\ref{prop:vstar}, $v_*<0$ on a sufficiently large tail.
Once we also have $u_*<\delta$,
\[
  \bigl(r^{d-1}v_*'(r)\bigr)'
  =-r^{d-1}g'(u_*(r))v_*(r)<0.
\]
If $\omega>0$, or if $\omega=0$ and $p\le p^*$, then
\eqref{eq:vasymp} gives $v_*(r)\to-\infty$ as $r\to\infty$.
Hence $v_*'<0$ eventually:
otherwise $v_*'\ge0$ on the whole tail, while once $v_*'<0$ at one point,
the preceding monotonicity preserves this sign.

Suppose instead that $\omega=0$ and $p>p^*$; then $d\ge3$ and
$v_*(r)\to c_+<0$ as $r\to\infty$. With
\[
  \eps\ce (d-2)(p-1)-2>0,
\]
the decay $u_*(r)=O(r^{-(d-2)})$ yields, as $r\to\infty$,
\[
  -g'(u_*(r))=O(r^{-2-\eps}),
  \qquad
  \bigl(r^{d-1}v_*'(r)\bigr)'=O(r^{d-3-\eps}).
\]
After integration,
$r^{d-1}v_*'(r)=o(r^{d-2})$, and therefore $rv_*'(r)\to0$ as
$r\to\infty$.

We may thus choose $R$ beyond the zero of $v_*$ so that $u_*(R)<\delta$,
$v_*(R)<0$, and
\begin{equation}\label{eq:tail}
\begin{cases}
  v_*'(R)<0,
    & \omega>0\ \text{or}\ (\omega=0,\ p\le p^*),\\[2pt]
  v_*(R)+\dfrac{R}{d-2}\bigl(v_*'(R)\bigr)_+<0,
    & \omega=0\ \text{and}\ p>p^*.
\end{cases}
\end{equation}
The second inequality follows from $v_*(r)\to c_+<0$ and
$rv_*'(r)\to0$ as $r\to\infty$.

For $y>y^*$, define
\[
  \zeta_y(r)\ce \frac{u_y(r)-u_*(r)}{y-y^*}.
\]
Then
\begin{equation*}
  \zeta_y''+\frac{d-1}{r}\zeta_y'+a_y(r)\zeta_y=0,
  \qquad
  a_y(r)\ce \int_0^1
    g'\bigl(u_*(r)+\theta(u_y(r)-u_*(r))\bigr)\,d\theta.
\end{equation*}
As $y\downarrow y^*$, $\zeta_y\to v_*$ in $C^1([0,R])$. Hence, for $y$
sufficiently close to $y^*$, $\zeta_y(R)<0$ and the corresponding strict
inequality in \eqref{eq:tail} holds with $\zeta_y$ in place of $v_*$.

We claim that such $y$ lies in $\mathcal Y_-$. Otherwise $u_y>0$ everywhere. On
the maximal interval $[R,R_1)$ on which $\zeta_y<0$,
\[
  0<u_y<u_*\le u_*(R)<\delta,
\]
so $a_y<0$ by \eqref{eq:gneg}, and
\begin{equation}
\label{eq:weiprim}
  \bigl(r^{d-1}\zeta_y'(r)\bigr)'
  =-r^{d-1}a_y(r)\zeta_y(r)<0.
\end{equation}
In the first case of \eqref{eq:tail}, this gives $\zeta_y'(R)<0$. Then it follows from \eqref{eq:weiprim} that
$\zeta_y'(r)<0$ on $[R,R_1)$, and hence
\[
  \zeta_y(r)\le \zeta_y(R)<0
  \qquad (R\le r<R_1).
\]
In the second case, $d\ge3$ and two integrations
give
\[
  \zeta_y(r)
  \le \zeta_y(R)+R^{d-1}\bigl(\zeta_y'(R)\bigr)_+
       \int_R^r s^{1-d}\,ds
  \le \zeta_y(R)+\frac{R}{d-2}\bigl(\zeta_y'(R)\bigr)_+<0.
\]
Thus $R_1=\infty$ and $\zeta_y$ stays bounded above by a negative constant.
Since $u_*(r)\to0$ as $r\to\infty$,
$u_y=u_*+(y-y^*)\zeta_y$ is eventually negative, a contradiction. Hence
$y\in \mathcal Y_-$.

Finally, $v_y\to v_*$ in $C^1([0,R])$ as $y\downarrow y^*$, so $v_y$
has exactly one zero on
$[0,R]$ and $v_y(R)<0$. Repeating the comparison above on $[R,R_y]$
gives $u_y<u_*<\delta$ on $[R,R_y)$, hence $g'(u_y)<0$ there. On every interval where $v_y<0$,
\[
  \bigl(r^{d-1}v_y'(r)\bigr)'
  =-r^{d-1}g'(u_y(r))v_y(r)<0.
\]
The appropriate inequality in \eqref{eq:tail} persists with $v_y$ in place
of $v_*$, so the same tail argument prevents a return to zero before
$R_y$.
\end{proof}

\begin{proposition}\label{prop:Sneg}
One has $(y^*,\infty)\subset \mathcal Y_-$. Moreover, for every $y>y^*$,
$v_y$ has exactly one zero in $(0,R_y)$, and $y\mapsto R_y$ is strictly
decreasing.
\end{proposition}

\begin{proof}
Set
\[
  \mathcal A_-\ce
  \{y\in\mathcal Y_-:
    v_y\text{ has exactly one zero in }(0,R_y)\}.
\]
Lemma~\ref{lem:right} shows that $\mathcal A_-$ contains an interval
to the right of $y^*$. We use two continuation facts.
If $y\in\mathcal A_-$, then Lemma~\ref{lem:end} gives
$v_y(R_y)<0$. Together with continuous dependence, the simplicity
of the zero of $v_y$, and Lemma~\ref{lem:basic}\,\textup{(iv)}, this yields
$[y,y+\varepsilon_y)\subset\mathcal A_-$ for some $\varepsilon_y>0$.
Conversely, if $y\in\mathcal Y_-$ and
$(y-\varepsilon,y)\subset\mathcal A_-$, then $y\in\mathcal A_-$.
Indeed, Lemma~\ref{lem:Szero} gives at least one zero of $v_y$
in $(0,R_y)$, while two simple zeros would persist for parameters
just below $y$ by continuous dependence, a contradiction.

Define
\begin{align*}
  \bar y&\ce
  \sup\{Y>y^*:(y^*,Y)\subset\mathcal Y_-\}\in(y^*,\infty],\\
  y_1&\ce
  \sup\{Y>y^*:(y^*,Y)\subset\mathcal A_-\}.
\end{align*}
If $y_1<\bar y$, the closure property from the left gives
$y_1\in\mathcal A_-$, and the continuation property to the right
extends $\mathcal A_-$ beyond $y_1$, a contradiction.
Thus $y_1=\bar y$ and $(y^*,\bar y)\subset\mathcal A_-$.
Differentiating $u_y(R_y)=0$ gives
\begin{equation}\label{eq:Rprime}
  R_y'=-\frac{v_y(R_y)}{u_y'(R_y)}<0,
\end{equation}
since $v_y(R_y)<0$ and $u_y'(R_y)<0$.

If $\bar y<\infty$, openness of $\mathcal Y_-$ and the definition
of $\bar y$ exclude $\bar y\in\mathcal Y_-$, while openness of
$\mathcal Y_+$ and $(y^*,\bar y)\subset\mathcal Y_-$ exclude
$\bar y\in\mathcal Y_+$. Hence $\bar y\in\mathcal Y_0$.
Lemma~\ref{lem:basic}\,\textup{(v)} then gives $R_y\to\infty$ as
$y\uparrow\bar y$, contradicting the decrease in \eqref{eq:Rprime}.
Therefore $\bar y=\infty$.
\end{proof}

We now combine the preceding shooting results to establish
uniqueness and radial nondegeneracy of the ground states.

\begin{proof}[Proof of Theorem~\ref{thm:shoot}]
Propositions~\ref{prop:first} and~\ref{prop:Sneg}, together with the
partition in Lemma~\ref{lem:basic}, give
\[
  \mathcal Y_+=(0,y^*),
  \quad
  \mathcal Y_0=\{y^*\},
  \quad
  \mathcal Y_-=(y^*,\infty).
\]
Thus every positive radial solution tending to zero is $u_{y^*}$, and
$\phi_\omega=u_{y^*}$.

Let $\zeta\in\ker(L_\omega)_{\rad}$. Its radial profile is regular at the
origin and solves
\[
  \zeta''+\frac{d-1}{r}\zeta'
  -\bigl(\omega+V_\omega(r)\bigr)\zeta=0,
  \qquad
  \zeta'(0)=0.
\]
The regular solution space is one-dimensional: if $\zeta(0)=0$, uniqueness
gives $\zeta\equiv0$, and otherwise $\zeta=\zeta(0)v_{y^*}$. The latter is
impossible by Proposition~\ref{prop:vstar}. If $\omega>0$, then
$v_{y^*}\notin L^2(\R^d)$, whereas $\zeta\in H^1(\R^d)$.
If $\omega=0$, the non-decaying asymptotics in \eqref{eq:vasymp}
imply $v_{y^*}\notin L^{p+1}(\R^d)$, whereas
$\zeta\in X_0=\dot H^1(\R^d)\cap L^{p+1}(\R^d)$.
Thus $\ker(L_\omega)_{\rad}=\{0\}$ in the function spaces of
Proposition~\ref{prop:nondeg}.
\end{proof}


\begingroup
\small
\section*{Acknowledgements}

Part of this work was carried out while M.H. was at the
University of Pisa in 2023--2024. He gratefully acknowledges
the hospitality received there.
Following the completion of our earlier one-dimensional paper \cite{FH25}
(\href{https://arxiv.org/abs/2304.14337}{arXiv:2304.14337}),
we studied the shooting analysis in previous works, developed the instability theory presented in this paper, and carefully checked our calculations.

M.H. also thanks Dmitry Pelinovsky for taking the time to
explain the work in \cite{HHP24} in detail.

N.F. was supported by JSPS KAKENHI Grant Numbers JP20K14349, JP24H00024, and JP26K00612. M.H. was supported by the Italian MIUR PRIN project 2020XB3EFL and by JSPS KAKENHI Grant Numbers JP24H00024 and JP25K17277.

\par
\endgroup


\end{document}